\documentclass{article}
\usepackage{graphicx} 

\usepackage{amsfonts}
\usepackage{amsmath}
\usepackage{amsthm}
\usepackage{amssymb}

\usepackage{color}
\usepackage[margin=1in]{geometry}
\usepackage{hyperref}
\usepackage{mathrsfs}
\usepackage{mathtools}
\usepackage{tikz}
\usepackage{tikz-cd}
\usepackage{xargs}
\usepackage{xcolor}
\allowdisplaybreaks
\definecolor{dark-red}{rgb}{0.7,0.25,0.25}
\definecolor{dark-blue}{rgb}{0.15,0.15,0.55}
\definecolor{medium-blue}{rgb}{0,0,0.65}
\hypersetup{
    colorlinks, linkcolor={dark-red},
    citecolor={dark-blue}, urlcolor={medium-blue}
}

\theoremstyle{definition}
\newtheorem{theorem}{Theorem}[section]
\newtheorem{example}[theorem]{Example}
\newtheorem{lemma}[theorem]{Lemma}
\newtheorem{prop}[theorem]{Proposition}
\newtheorem{remark}[theorem]{Remark}
\newtheorem{definition}[theorem]{Definition}

\newcommand{\locIdeal}[2]{\tikz[baseline=.1ex]{
	\draw[gray!40, thick, fill=gray!40, domain=-45:225] plot ({cos(\x)}, {sin(\x)}) to[out=45, in=130] (0.71, -0.71);
	\draw[thick] (-0.71, -0.71) to[out=45, in=135] (0.71, -0.71);
	\node[draw, circle, inner sep=0pt, minimum size=4pt, fill=white] (p1) at (0,-0.41) {};
	\draw[thick] (p1) -- (-0.71,0.71);
	\draw[thick] (0.71,0.71) -- (p1);
	\node (s1) at (-0.6,-0.1) {$#1$};
	\node (s2) at (0.6,-0.1) {$#2$};}
}
\newcommand{\locTangle}[2]{\tikz[baseline=.1ex]{
	\draw[gray!40, thick, fill=gray!40, domain=-45:225] plot ({cos(\x)}, {sin(\x)}) to[out=45, in=130] (0.71, -0.71);
	\draw[thick] (-0.71, -0.71) to[out=45, in=135] (0.71, -0.71);
	\node[draw, circle, inner sep=0pt, minimum size=4pt, fill=white] (p1) at (0,-0.41) {};
	\draw[thick] (-0.15, -0.173) -- (-0.71,0.71);
	\draw[thick] (0.71,0.71) -- (p1);
	\node (s1) at (-0.55,-.15) {$#1$};
	\node (s2) at (0.55,-0.15) {$#2$};}
}
\newcommandx{\MarkedTorusBackground}[6][1=, 2=, 3=, 4=, 5=, 6=]{
    \draw[draw=none, fill=gray!40] (0,1) -- (2,1) -- (2,-1) -- (0,-1) -- (1, -0.15) to[out=0, in=-90] (1.15, 0) to[out=90, in=0] (1, 0.15) to[out=180, in=90] (0.85, 0) to[out=-90, in=180] (1, -0.15) -- (0, -1) -- (0,1);
    \draw[thick, ->] (0,1) -- (0.9,1);
    \draw[thick, ->] (0.85,1) -- (1.2,1);
    \draw[thick] (1.2,1) -- (2,1);
    \draw[thick, ->] (0,-1) -- (0.9,-1);
    \draw[thick, ->] (0.85,-1) -- (1.2,-1);
    \draw[thick] (1.2,-1) -- (2,-1);
    \draw[thick, ->] (0,-1) -- (0,0);
    \draw[thick] (0,-1) -- (0,1);
    \draw[thick, ->] (2,-1) -- (2,0);
    \draw[thick] (2,-1) -- (2,1);
    \draw[thick, black] (1, 0) circle (0.15);
    \node at (0.1, -1.2) {#1};
    \node at (0.45, -1.2) {#2};
    \node at (0.8, -1.2) {#3};
    \node at (1.2, -1.2) {#4};
    \node at (1.55, -1.2) {#5};
    \node at (1.9, -1.2) {#6};
}
\newcommand{\PuncturedMonogonBackground}{
    \draw[thick, black, fill=gray!40] (0,0) circle (1);
    \path[thick, tips, ->] (-0.5, -1) -- (0.05, -1);
    \draw[thick, black, fill=white] (0, 0) circle (0.1);
}

\newcommand{\Z}{\mathbb{Z}} 
\newcommand{\C}{\mathbb{C}} 
\newcommand{\SH}{\mathrm{S\ddot H}}
\newcommand{\e}{\mathbf{e}} 
\newcommand{\HH}{{\mathrm{\ddot H}}}

\title{Stated Skeins and DAHAs}
\author{Raymond Matson and Peter Samuelson}

\begin{document}

\maketitle


\section{Introduction}
A \emph{skein module} is an invariant $Sk(M)$ of an oriented 3-manifold $M$ that was introduced independently by Przytycki \cite{Prz91} and Turaev \cite{Tur88}. The skein $Sk(M)$ is a module over a commutative ring $R$ containing an invertible element $q \in R$ (e.g. $R = \mathbb{Z}[q,q^{-1}]$), and these modules generalize polynomial invariants of knots and links in $\mathbb R^3$. As an $R$-module they are spanned by framed links (in some versions the links are  decorated with some representation-theoretic data). Typically the skein module of $\mathbb R^3$ is isomorphic to $R$ itself, and under this isomorphism a link is sent to its polynomial invariant. There are many versions\footnote{To avoid cluttering the introduction with too much notation involving versions of skein modules, we keep some statements somewhat vague.} of polynomial invariants and these typically correspond to a specific version of skein module. For example, the Jones polynomial of a knot  $K \subset \mathbb R^3$ is equal to the class $[K]$ in the Kauffman bracket skein module of $\mathbb R^3$.

The skein module construction is functorial with respect to oriented embeddings of manifolds, and this functoriality converts additional structure on a 3-manifold $M$ into additional structure on the skein $Sk(M)$. For example, if $M = \Sigma \times [0,1]$ for some surface $\Sigma$, the embedding\footnote{If the version of skein theory in question has boundary conditions, then the algebra structure on the skein of $\Sigma \times [0,1]$ requires consistent boundary conditions at the $\{0\}$ and $\{1\}$ levels.} $(M\times [0,1]) \sqcup (M\times [1,2]) \hookrightarrow M \times [0,2]$ induces an algebra structure on $Sk(\Sigma \times [0,1])$. Concretely, if the skein is generated by links (with decorations), the product $a\cdot b$ in the skein algebra $Sk(\Sigma \times [0,1])$ is given by ``stacking $a$ on top of $b$ in the $[0,1]$ direction.'' Similarly, if $M$ is a 3-manifold with boundary, then the skein module $Sk(M)$ is a module over the skein algebra $Sk(\partial M \times [0,1])$, with the action induced by gluing a collar of the boundary onto $M$.

It turns out that in many interesting cases, the algebras and modules constructed by skein theory are closely related to \emph{double affine Hecke algebras} studied in representation theory. Double affine Hecke algebras (or DAHAs) were introduced by Cherednik \cite{Che95, Che05} to prove several conjectures by Macdonald in algebraic combinatorics, and they have since found applications in many areas of mathematics, including representation theory, algebraic geometry, and knot theory. The DAHA $ \HH_{q,t}(\mathfrak g)$ is an algebra over $\mathbb Q[q^{\pm 1},t^{\pm 1}]$ that is defined by generators and relations using the root data of the Lie algebra $\mathfrak g$. In the specialization $q=t=1$, the DAHA $ \HH_{q,t}(\mathfrak g)$ is isomorphic to the group algebra of $(P\oplus Q)\rtimes W$, where $P$ and $Q$ are the root and weight lattices and $W$ is the Weyl group. The $t$ parameter deforms the Weyl group to the (finite) Hecke algebra, and the $q$ parameter deforms the Laurent polynomial ring $\mathbb{Q}[P\oplus Q]$ into a quantum torus. Remarkably, both of these deformations can be made simultaneously, and $\HH_{q,t}(\mathfrak g)$ is flat over $\mathbb Q[q^{\pm 1},t^{\pm 1}]$.

The goal of the present paper is to study some relationships between \emph{stated skein theory} and DAHAs in the case $\mathfrak g = \mathfrak{sl}_2$. Very roughly, in this case, the relationship between stated skein theory and the skein theory described above is similar to the relationship between decorated Teichm\"{u}ller space and Teichm\"{u}ller space. Concretely, the Kauffman bracket skein module is spanned by framed, unoriented links in  $M$, and the stated Kauffman bracket skein module $\mathscr{S}(M,N)$ is spanned by framed unoriented tangles in $M$, where the ends of the tangles are decorated by \emph{states}\footnote{For $\mathfrak{sl}_2$, a state is a choice of sign $\pm$.}, and are required to end on prescribed intervals in $N \subset \partial M$. For a precise definition, see Section \ref{section:SSA}.

Most of our results relate the stated skein $\mathscr{S}(M,N)$ to Terwilliger's \emph{universal $A_1$ spherical DAHA}, which we call $A$ in the introduction:
\begin{enumerate}
    \item In Theorem \ref{theorem:Generators}, we find generators of the skein algebra $\mathscr{S}(T^2\setminus D^2)$. We also list some relations between these generators in Appendix \ref{section:pbw}, but finding a complete list of relations seems like a very challenging technical problem. 
    \item In Proposition \ref{prop:embedding}, we use skein theory to construct  embeddings of $A$ into a ``rank 6'' quantum torus. (This seems to be unrelated to the standard embedding of this DAHA into a localization of a rank 2 quantum torus that Cherednik used to construct the polynomial representation of the DAHA.)  
    \item In Section \ref{section:GenConst}, we show that each marked 3-manifold with boundary $T^2$ produces a module over Terwilliger's universal $A_1$ spherical DAHA.
    \item Examples: we provide explicit formulas for one particular embedding of $A$ into a quantum torus in Proposition \ref{prop:embedding}, and we provide formulas for the action of $A$ on the skein module of the genus 1 handlebody in Section \ref{section:solidtorus}. 
\end{enumerate}

An outline of the paper is as follows: In Section \ref{section:background}, we briefly discuss historical background about skein modules, double affine Hecke algebras, and previously discovered relationships between these objects. In Section \ref{section:ssatb} we discuss the stated skein algebra of the torus. In Section \ref{section:modsQTori} we use a general theorem about embedding stated skein algebras into quantum tori to compute a specific embedding of the Terwilliger's universal $A_1$ DAHA into a quantum torus. In Section \ref{section:mods3mflds} we describe the construction of universal $A_1$ spherical DAHA modules using 3-manifolds and perform some computations with these modules. In the Appendices we include some of the lengthier diagrammatic and algebraic computations.

\vspace{3mm}

\noindent \textbf{Acknowledgements:} The authors would like to thank Yuri Berest, Thang L\^{e}, David Jordan, Charles Frohman, Chris Grossack, Alexei Oblomkov, Alexander Pokorny Space, Shane Rankin, and A.~Referee for helpful comments and for discussions about skein modules, DAHAs, and related topics over the years.

\section{Historical background}\label{section:background}
\subsection{Classical skein theory -- Kauffman bracket skein modules}
In this section we define the Kauffman bracket skein module and recall several of its properties. In particular, we describe some results mentioned in the introduction that are used to relate skein algebras to DAHAs. In this section our ground ring will be a commutative (Noetherian) ring $R$ containing an invertible element $q \in R^\times$.


\subsubsection{Kauffman bracket skein modules}\label{section:kbsm}
Recall that two maps $f,g:M \to N$ of manifolds are \emph{ambiently isotopic} if they are in the same orbit of the identity component of the diffeomorphism group of $N$. This is an equivalence relation, and a \emph{knot} in a 3-manifold $M$ is the equivalence class of a smooth embedding $K: S^1 \hookrightarrow M$. 

A \emph{framed link} is an 
embedding of a disjoint union of annuli $S^1 \times [0,1]$ into an oriented 3-manifold $M$. (The framing refers to the $[0,1]$ factor and is a technical detail that will be suppressed when possible.) We will consider framed links to be equivalent if they are ambiently isotopic.

Let  $\mathscr L(M)$ be the free $R$-module with basis 
the set of ambient isotopy classes of framed unoriented links in $M$ (including the empty link). Let $\mathscr L'(M)$ be the smallest $R$-submodule of $\mathscr L(M)$ containing the following skein expressions:
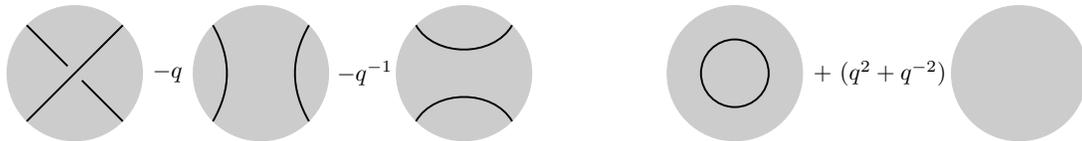
\begin{center}\resizebox{0.9\width}{!}{
   \begin{tikzpicture}
        \draw[gray!40, fill=gray!40] (0.25,0) circle (1);
        \draw[thick] (-0.46, 0.71) -- (0.96, -0.71);
        \draw[line width=3mm, gray!40] (0.95, 0.70) -- (-0.45, -0.70);
        \draw[thick] (0.96, 0.71) -- (-0.46, -0.71);
        \node[text width=1cm] at (1.9,0) {$- q$};
        \draw[gray!40, fill=gray!40] (3,0) circle (1);
        \draw[thick] (2.29, 0.71) to[out=-60, in=60] (2.29, -0.71);
        \draw[thick] (3.71, 0.71) to[out=-120, in=120] (3.71, -0.71);
        \node[text width=1.25cm] at (4.75,0) {$- q^{-1}$};
        \draw[gray!40, fill=gray!40] (6,0) circle (1);
        \draw[thick] (5.29, 0.71) to[out=-60, in=-120] (6.71, 0.71);
        \draw[thick] (5.29, -0.71) to[out=60, in=120] (6.71, -0.71);
        \draw[gray!40, fill=gray!40] (10,0) circle (1);
        \draw[thick] (10,0) circle (.5);
        \node[text width=2.4cm] at (12.36,0) {$+$ $(q^2 + q^{-2})$};
        \draw[gray!40, fill=gray!40] (14.2,0) circle (1);
   \end{tikzpicture}}
\end{center}


The meaning of the picture is as follows: suppose there are 3 links $L_+$, $L_0$, and $L_\infty$ which are identical outside of a small 3-ball (embedded as an oriented sub-manifold of $M$), and inside the 3-ball 
they appear as in the figure (where $L_+$ is the leftmost link). Then  the element $L_+ - qL_0 - q^{-1} L_\infty$ is in the submodule $\mathscr L'$. (All pictures drawn in this paper will have blackboard framing. In other words, a line on the page represents a strip $[0,1]\times [0,1]$ in a tubular neighborhood of the page, 
and the strip is always perpendicular to the paper.)

\begin{definition}[\cite{Prz91}]
The \textbf{Kauffman bracket skein module} is the vector space $K_q(M) := \mathscr L / \mathscr L'$. It contains a canonical element $\varnothing \in K_q(M)$ corresponding to the empty link.
\end{definition}

\begin{remark}
 To shorten the notation, if $M = F \times [0,1]$ for a surface $F$, we will often write $K_q(F)$ for the skein module $K_q(F\times [0,1])$.
\end{remark}

\begin{example}\label{skeins3}
One original motivation for defining $K_q(M)$ is the isomorphism of vector spaces
\[R \stackrel \sim \longrightarrow K_q(S^3), \quad 1 \mapsto \varnothing \]
Kauffman proved that this map is an isomorphism and that the inverse image of a link is the Jones polynomial\footnote{More precisely, the image of the link is a number in $\C$ that depends polynomially on $q \in \C^*$, and this (Laurent) polynomial is the Jones polynomial of the link, up to a normalization.} of the link. The skein relations  can be used to remove crossings and trivial loops of a diagram of a link until the diagram is a multiple of the empty link, which shows that the vector space map $R \to K_q(S^3)$ sending $\alpha \mapsto \alpha\cdot \varnothing$ is surjective. Showing it is injective is equivalent to showing the Jones polynomial of a link is well-defined.
\end{example}

In general $K_q(M)$ is just an $R$-module - however, if $M$ has extra structure, then $K_q(M)$ also has extra structure. In particular,
\begin{enumerate}
\item If $M = F \times [0,1]$ for some surface $F$, then $K_q(M)$ is an algebra (which is typically noncommutative). The multiplication is given by ``stacking links." 
\item If $M$ is a manifold with boundary, then $K_q(M)$ is a module over $K_q(\partial M)$. The action is given by ``pushing links from the boundary into the manifold.'' 
\item If $q=\pm 1$, then $K_q(M)$ is a commutative algebra (for any oriented 3-manifold $M$). The multiplication is given by ``disjoint union of links,'' which is well-defined because when $q=\pm 1$, the skein relations allow strands to `pass through' each other.
\end{enumerate}

\begin{example}
 Let $M = (S^1 \times [0,1]) \times [0,1]$ be the solid torus. The skein relations can be applied to remove crossings and trivial loops in a diagram of any link, and the result is a sum of unions of parallel copies of the loop $u$ generating $\pi_1(M)$. This shows that algebra map $R[u] \to K_q(S^1\times[0,1])$ sending $u^n$ to $n$ parallel copies of $u$ is surjective, and it follows from \cite{SW07} that this map is injective.
\end{example}

\begin{remark}
    If $M$ is a 3-manifold with $\partial M \cong \Sigma$, then $K_q(M)$ is a module over the skein algebra $K_q(\Sigma \times [0,1])$, where the action comes from ``pushing links from the boundary into $M$.'' In particular, if $K \subset S^3$ is a knot and $K_\epsilon$ is a small open neighborhood of $K$, then $S^3 \setminus K_\epsilon$ is a module over the skein algebra $K_q(T^2)$. Since the skein algebra $K_q(T^2)$ is isomorphic to the $t=1$ specialization of the spherical DAHA (see Remark \ref{rmk:dahateq1}), each 3-manifold with torus boundary produces a module over the spherical $A_1$ DAHA at $t=1$ (or, at $t=q$, since these are isomorphic algebras).
\end{remark}

\subsubsection{The Kauffman bracket skein algebra in genus 1}
We recall that the \emph{quantum torus} is the algebra
\[
A_q := \frac{R\langle X^{\pm 1},Y^{\pm 1}\rangle}{XY-q^2YX}
\]
There is a $\mathbb Z_2$ action by algebra automorphisms on $A_q$ where the generator simultaneously inverts $X$ and $Y$. We define $e_{r,s} = q^{-rs}X^{r}Y^s \in A_q$, which form a linear basis for the quantum torus $A_q$ and satisfy the relations
\[e_{r,s}e_{u,v} = q^{rv-us}e_{r+u,s+v}.\]

In this section we recall a beautiful theorem of Frohman and Gelca in \cite{FG00} that gives a connection between skein modules and the invariant subalgebra $A_q^{\mathbb Z_2}$. First we establish some notation. Let $T_n \in \C[x]$ be the Chebyshev polynomials defined 
by $T_0 = 2$, $T_1 = x$, and the relation $T_{n+1} = xT_n-T_{n-1}$. If $m,l$ are relatively prime, write $(m,l)$ for the $m,l$ curve on the torus (which is the simple curve wrapping around 
the torus $l$ times in the longitudinal direction and $m$ times in the meridian's direction). It is clear that the links $(m,l)^n$ span $K_q(T^2)$, and it follows from \cite{SW07} that this set is a basis.  However, a more convenient basis is given by the elements $(m,l)_T = T_d((\frac m {d}, \frac l {d}))$ (where $d = \mathrm{gcd}(m,l)$). 
 (We point out that since we are considering unoriented curves, $(m,l) = (-m,-l)$.)

\begin{theorem}[\cite{FG00}]\label{fg00}
The map $f:K_q(T^2) \to A_q^{\Z_2}$ given by $f((m,l)_T) = e_{m,l}+e_{-m,-l}$ is an isomorphism of algebras.
\end{theorem}

In this paper, we will mainly focus on the skein algebra of the punctured torus $T^2 \setminus D^2$. A presentation for this algebra was given by Bullock and Przytycki:

\begin{theorem}[\cite{BP00}] \label{theorem:bp}
The Kauffman bracket skein algebra of the punctured torus is isomorphic to the algebra generated by $x,y,z$, subject to the relations
\begin{equation}\label{relationsforB'}
[x,y]_q = (q^2-q^{-2})z,\quad
[z,x]_q = (q^2-q^{-2})y,\quad
[y,z]_q = (q^2-q^{-2})x
\end{equation}
where we have used the notation $[x,y]_q := qxy-q^{-1}yx$. 
The elements $x$, $y$, and $z$ correspond to the simple closed curves of homology classes $(1,0)$, $(0,1)$, and $(1,1)$, respectively.
\end{theorem}

\subsection{The $\mathfrak{sl}_2$ double affine Hecke algebra}\label{sec_dahabackground}
In this section we recall background about the double affine Hecke algebra $\HH_{q,t}$ of type $A_1$ that we will use later. The standard reference for the material in this section is \cite{Che05}.

\subsubsection{The Poincar\`e-Birkhoff-Witt property}
We first give a presentation of the algebra $\HH_{q,t}$.
\begin{definition}
Let $\HH_{q, t}$ be the algebra generated by $X^{\pm 1}$, $Y^{\pm 1}$, and $T$ subject to the relations
\begin{equation}
\label{daha}
TXT=X^{-1},\quad TY^{-1}T = Y, \quad XY=q^2YXT^2, \quad
(T-t)(T+t^{-1})=0
\end{equation}
\end{definition}

We remark that we have replaced the $q$ that is standard in the third relation with $q^2$ to agree with the standard conventions for  skein relations. Also, note that the fourth relation implies that $T$ is invertible, with inverse $T^{-1} = T + t^{-1} - t$. Finally, if we set $t=1$, then the fourth relation reduces to $T^2 = 1$, and the third relation becomes $XY=q^2YX$. These imply that $\HH_{q,1}$ is isomorphic to the cross product $A_q\rtimes \Z_2$ (where the generator of $\Z_2$ acts by inverting $X$ and $Y$). 

One of the key propeties of $ \HH_{q,t} $ is the so-called PBW property, which says that the multiplication map yields a linear isomorphism
$$
R[X^{\pm 1}] \otimes R[\Z_2] \otimes R[Y^{\pm 1}] \stackrel{\sim}{\to} \HH_{q,t}
$$
Another way of stating this property is that the elements $\,\{X^n T^\varepsilon Y^m\, :\, m,n \in \Z\,,\,\varepsilon = 0,\,1\}\,$
form a linear basis in $ \HH_{q,t} $. (See \cite{Che05}, Theorem~2.5.6(a).)

\subsubsection{The spherical subalgebra}\label{sph}
If $ t \ne \pm i $, the algebra $ \HH_{q,t} $ contains the idempotent $\e := (T+t^{-1})/(t+t^{-1})$ (the identity
$\,\e^2 = \e \,$ is equivalent to the last relation in (\ref{daha})). The \emph{spherical subalgebra} of $ \HH_{q,t} $ is
\begin{equation}
\label{salg}
\SH_{q,t} := \e\HH_{q,t}\e 
\end{equation}
Note that $\SH_{q,t}$ inherits its additive and multiplicative structure from $ \HH_{q,t} $, but the identity element of $\SH_{q,t}$ is
$ \e $, which is different from $\, 1 \in \HH_{q,t} $.

\begin{remark}\label{rmk:dahateq1}
In the $t=1$ specialization, there is an isomorphism $A_q\rtimes \Z_2 \cong H_{q,1}$ given by sending $T$ to the generator of $\Z_2$ (and $X,Y$ map to $X,Y$). This induces an isomorphism
$A_q^{\Z_2} \cong \SH_{q,1}$ given by $w \mapsto \e \bar w \e$, where $w \in A_q^{\Z_2}$ is a symmetric word in $X,Y$, and $\bar w$ is the same word, viewed as an element of $\HH_{q,t}$. Combining this isomorphism with the main theorem of \cite{FG00} shows that the skein algebra $K_q(T^2)$ is isomorphic to $\SH_{q,1}$. For later use we will need a presentation of $\SH_{q,t}$ which we give here. 
\end{remark}

\begin{definition}
\emph{Terwilliger's universal $A_1$ spherical DAHA} $B'_q$ is the algebra generated by $x,y,z$ modulo the following relations:
\begin{equation}\label{relationsforB'}
[x,y]_q = (q^2-q^{-2})z,\quad
[z,x]_q = (q^2-q^{-2})y,\quad
[y,z]_q = (q^2-q^{-2})x
\end{equation}
Also, define $B_{q,t}$ to be the quotient of $B'_q$ by the additional relation
\begin{equation}\label{casimir_rel}
q^2x^2 + q^{-2}y^2+ q^2z^2 -qxyz= \left( \frac t q - \frac q t\right)^2 + \left( q + \frac 1 q\right)^2
\end{equation}
\end{definition}

\begin{remark}\label{rmkcascentral}
    The element on the left hand side of (\ref{casimir_rel}) is central in $B'_q$, so $B_q$ is the quotient of $B'_q$ by a central character.
\end{remark}

\begin{theorem}[\cite{Ter13}]\label{theorem:ter}
There is an algebra isomorphism $f:B_{q,t} \to \SH_{q,t}$ defined by the following formulas:
\begin{align}\label{xyz}
x &\mapsto (X+X^{-1})\e\notag\\
y &\mapsto (Y+Y^{-1})\e\\
z &\mapsto q^{-1}(XYT^{-2}+X^{-1}Y^{-1}) \e\notag
\end{align}
\end{theorem}
\begin{proof}
The fact that (\ref{xyz}) gives a well-defined algebra map can be checked directly, and the fact that it is an isomorphism is proved in \cite{Ter13}. (See also \cite[Thm. 2.20]{BS16} for the precise conversion between Terwilliger's notation and ours.)
\end{proof}

\begin{remark}
    A-priori, it isn't obvious that the elements on the right hand side of (\ref{xyz}) are contained in $\e \HH_{q,t} \e$. However, short computations show that if we take $a \in \HH_{q,t}$ to be either $X+X^{-1}$, $Y+Y^{-1}$, or $XYT^{-2}+X^{-1}Y^{-1}$, then $a\e = \e a$, and this implies $a \e = \e a \e \in \e \HH_{q,t}\e$.
\end{remark}

\subsection{Previous Relationships}

Even though skein algebras and DAHAs have completely different definitions, in the last two decades a number of relationships have been found between various flavors of these algebras. Here we give a brief, incomplete, imprecise summary of some of these relationships.
\begin{enumerate}
    \item \cite{FG00} The $\mathfrak{sl}_2$ (Kauffman bracket) skein algebra of $T^2$ is isomorphic to the $t=1$ specialization of the $\mathfrak{sl}_2$ spherical DAHA (and also to the $t=q$ specialization), and the $SL_2(\mathbb Z)$ actions on both sides agree.
    \item The skein algebra of the 1-punctured torus is isomorphic to Terwilliger's universal $A_1$ spherical DAHA (see Theorems \ref{theorem:bp} and \ref{theorem:ter} above) \cite{BP00, Ter13}.
    \item The  Kauffman bracket skein algebra of the 4-punctured sphere is isomorphic to the $BC_1$ spherical DAHA. This follows directly from results in \cite{BP00} and \cite{Ter13}. Under this identification, finite dimensional modules over the $BC_1$ spherical DAHA are submodules of the skein module of the genus 2 handlebody \cite{CS21}.
    \item \cite{Che13, Sam19, CD16} Jones polynomials (and Reshetikhin-Turaev invariants) of iterated torus knots have $q,t$ versions constructed using DAHAs, and the $t=q$ specialization recovers the original knot invariant.
    \item \cite{BS16} Conjecturally, $\mathfrak{sl}_2$ skein modules of knot complements deform canonically to modules over the 2-parameter $\mathfrak{sl}_2$ spherical DAHA.
    \item \cite{MS17} The $\mathfrak{gl}_\infty$ DAHA (or, elliptic Hall algebra of Burban and Schiffmann \cite{BS12}) specializes at $t=q$ to the Homflypt skein algebra of the torus, and the $SL_2(\mathbb Z)$ actions agree.
    \item \cite{AS19} Arthamonov and Shakirov constructed an algebra depending on parameters $q,t$ that specializes at $t=q$ to the Kauffman bracket skein algebra of the genus 2 surface, and is acted on by the genus 2 mapping class group. The specialization of the genus 2 DAHA to skein algebras was studied in \cite{CS21}, and Hikami has closely related work in \cite{Hik19}.
    \item \cite{MS21, BCMN23} The $\mathfrak{gl}_n$ DAHA (not spherical, with both parameters) can be realized as braids/tangles in the torus modulo the Homflypt skein relations (and a special relation involving a marked point). This gives a definition of higher genus DAHAs (in type $A_n$). The marked point relation was given a conceptual explanation in upcoming work of Ion and Roller.
    \item \cite{HTY23} Higher genus DAHAs (defined as in \cite{MS21}) can be realized using higher dimensional Heegard-Floer homology.
    \item \cite{GJV24} In type $A_n$, higher rank  DAHAs (specialized at $t=q$) can be realized through factorization homology of the torus (which, very roughly, is a categorical generalization of skein algebras). 
\end{enumerate}

\section{The Stated Skein Algebra of the Torus with Boundary}\label{section:ssatb}

\subsection{Stated Skein Algebras}\label{section:SSA}

A \textit{marked surface} is a pair $(\Sigma, \mathcal{P})$ where $\Sigma$ is a compact oriented surface with (possibly empty) boundary $\partial \Sigma$, and $\mathcal{P} \subset \partial \Sigma$ is a finite set, called the set of marked points. The associated marked $3$-manifold $(M, \mathcal{N})$ is defined by $M = \Sigma \times I$ and its markings $\mathcal{N} = \mathcal{P} \times I$. A \textit{stated $\mathcal{N}$-tangle} or a \textit{$\partial M$-tangle} is a pair ($\alpha$, $s$) where $\alpha$ is a compact $1$-dimensional unoriented submanifold with a framing such that $\partial \alpha = \alpha \cap \mathcal{N}$ and $s$ is a map $s : \partial \alpha \to \{\pm\}$.


A diagram $D$ in $\Sigma$ representing a stated $\mathcal{N}$-tangle in $\Sigma \times I$ is \emph{simple} if $D$ does not contain any crossings in the interior of $\Sigma$ and has no trivial components. (Throughout this work we will only use the term ``simple'' in the case $M = \Sigma \times I$, so that a simple $\partial M$–tangle diagram may be understood as the projection of the tangle onto $\Sigma$.) A closed component of $\alpha$ is \textit{trivial} if it bounds a disk in $M$. Similarly, a tangle component is \textit{trivial} if it can be homotoped, relative to its endpoints, into a single marking. A \textit{parallel tangle} is one that can be homotoped, relative to its endpoints, to a boundary edge, that is a component of $\partial M \setminus \mathcal{N}$.

The \textit{stated skein module} of $(M, \mathcal{N})$, denoted $\mathscr{S}\left(M, \mathcal{N} \right)$, is the quotient of the free $R$-module spanned by isotopy classes of stated $\mathcal{N}$-tangles subject to the following local relations.
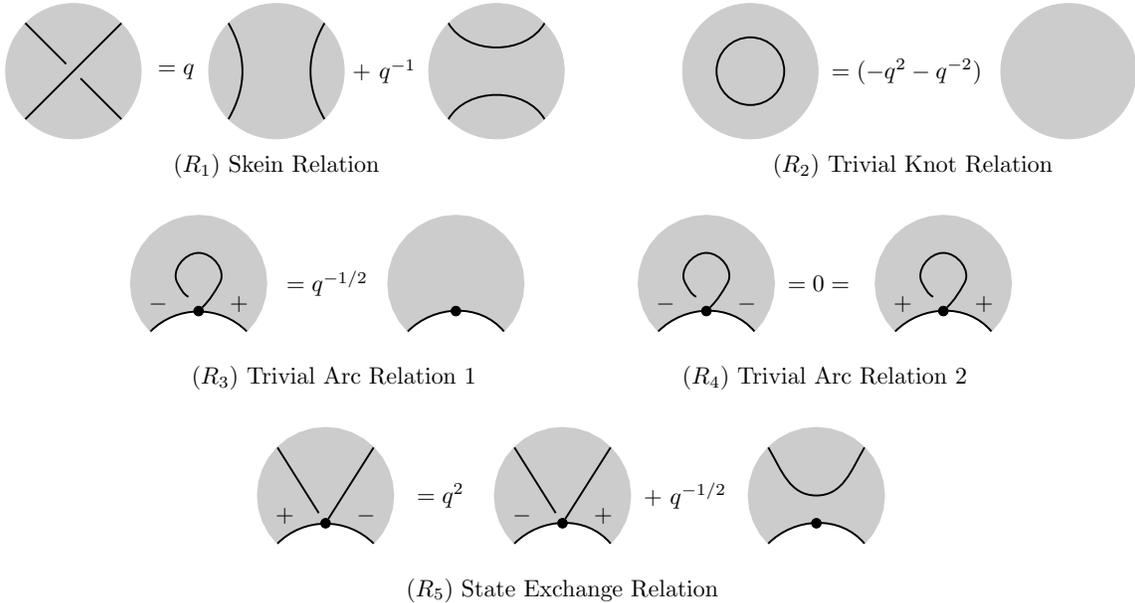
\begin{center}\resizebox{0.9\width}{!}{
   \begin{tikzpicture}
        \draw[gray!40, fill=gray!40] (0,0) circle (1);
        \draw[thick] (-0.71, 0.71) -- (0.71, -0.71);
        \draw[line width=3mm, gray!40] (0.70, 0.70) -- (-0.70, -0.70);
        \draw[thick] (0.71, 0.71) -- (-0.71, -0.71);
        \node[text width=1cm] at (1.75,0) {$= q$};
        \draw[gray!40, fill=gray!40] (3,0) circle (1);
        \draw[thick] (2.29, 0.71) to[out=-60, in=60] (2.29, -0.71);
        \draw[thick] (3.71, 0.71) to[out=-120, in=120] (3.71, -0.71);
        \node[text width=1.25cm] at (4.75,0) {$+$ $q^{-1}$};
        \draw[gray!40, fill=gray!40] (6.25,0) circle (1);
        \draw[thick] (5.54, 0.71) to[out=-60, in=-120] (6.96, 0.71);
        \draw[thick] (5.54, -0.71) to[out=60, in=120] (6.96, -0.71);
        \node at (3,-1.4) {$(R_1)$ Skein Relation};
        \draw[gray!40, fill=gray!40] (10,0) circle (1);
        \draw[thick] (10,0) circle (.5);
        \node[text width=2.4cm] at (12.4,0) {$= (-q^2 - q^{-2})$};
        \draw[gray!40, fill=gray!40] (14.7,0) circle (1);
        \node at (12.4,-1.4) {$(R_2)$ Trivial Knot Relation};
   \end{tikzpicture}}
\end{center}
\begin{center}\resizebox{0.9\width}{!}{
   \begin{tikzpicture}
        \draw[gray!40, thick, fill=gray!40, domain=-45:225] plot ({cos(\x)}, {sin(\x)}) to[out=45, in=130] (0.71, -0.71);
        \draw[thick] (-0.71, -0.71) to[out=45, in=135] (0.71, -0.71);
        \node[draw, circle, inner sep=0pt, minimum size=4pt, fill=black] (p1) at (0,-0.41) {};
        \draw[thick] (p1) to[out=45, in=-120] (0.31,0) to[out=60, in=0] (0,0.45) to[out=180, in=120] (-0.31,0) to[out=-60, in=150] (-0.15,-0.2);
        \node (s1) at (-0.6,-0.3) {$-$};
        \node (s2) at (0.6,-0.3) {$+$};
        \node[text width=1.4cm] at (2,0) {$= q^{-1/2}$};
        \draw[gray!40, thick, fill=gray!40, domain=-45:225] plot ({3.8+cos(\x)}, {sin(\x)}) to[out=45, in=130] (4.51, -0.71);
        \draw[thick] (3.09, -0.71) to[out=45, in=130] (4.51, -0.71);
        \node[draw, circle, inner sep=0pt, minimum size=4pt, fill=black] (p2) at (3.8,-0.41) {};
        \node at (2,-1.4) {$(R_3)$ Trivial Arc Relation 1};
        \draw[gray!40, thick, fill=gray!40, domain=-45:225] plot ({7.5+cos(\x)}, {sin(\x)}) to[out=45, in=130] (8.21, -0.71);
        \draw[thick] (6.79, -0.71) to[out=45, in=130] (8.21, -0.71);
        \node[draw, circle, inner sep=0pt, minimum size=4pt, fill=black] (p3) at (7.5,-0.41) {};
        \draw[thick] (p3) to[out=45, in=-120] (7.81,0) to[out=60, in=0] (7.5,0.45) to[out=180, in=120] (7.19,0) to[out=-60, in=150] (7.35,-0.2);
        \node (s1) at (6.9,-0.3) {$-$};
        \node (s2) at (8.1,-0.3) {$-$};
        \node[text width=1.1cm] at (9.25,0) {$= 0 =$};
        \draw[gray!40, thick, fill=gray!40, domain=-45:225] plot ({11+cos(\x)}, {sin(\x)}) to[out=45, in=130] (11.71, -0.71);
        \draw[thick] (10.29, -0.71) to[out=45, in=130] (11.71, -0.71);
        \node[draw, circle, inner sep=0pt, minimum size=4pt, fill=black] (p4) at (11,-0.41) {};
        \draw[thick] (p4) to[out=45, in=-120] (11.31,0) to[out=60, in=0] (11,0.45) to[out=180, in=120] (10.69,0) to[out=-60, in=150] (10.85,-0.2);
        \node (s1) at (10.4,-0.3) {$+$};
        \node (s2) at (11.6,-0.3) {$+$};
        \node at (9.25,-1.4) {$(R_4)$ Trivial Arc Relation 2};
   \end{tikzpicture}}
\end{center}
\begin{center}\resizebox{0.9\width}{!}{
   \begin{tikzpicture}
        \draw[gray!40, thick, fill=gray!40, domain=-45:225] plot ({cos(\x)}, {sin(\x)}) to[out=45, in=130] (0.71, -0.71);
        \draw[thick] (-0.71, -0.71) to[out=45, in=135] (0.71, -0.71);
        \node[draw, circle, inner sep=0pt, minimum size=4pt, fill=black] (p1) at (0,-0.41) {};
        \draw[thick] (-0.1, -0.26) -- (-0.71,0.71);
        \draw[thick] (p1) -- (0.71,0.71);
        \node (s1) at (-0.6,-0.3) {$+$};
        \node (s2) at (0.6,-0.3) {$-$};
        \node[text width=1.1cm] at (1.9,0) {$= q^{2}$};
        \draw[gray!40, thick, fill=gray!40, domain=-45:225] plot ({3.5+cos(\x)}, {sin(\x)}) to[out=45, in=130] (4.21, -0.71);
        \draw[thick] (2.79, -0.71) to[out=45, in=130] (4.21, -0.71);
        \node[draw, circle, inner sep=0pt, minimum size=4pt, fill=black] (p2) at (3.5,-0.41) {};
        \draw[thick] (3.4,-0.26) -- (2.79,0.71);
        \draw[thick] (p2) -- (4.21,0.71);
        \node (s3) at (2.9,-0.3) {$-$};
        \node (s4) at (4.1,-0.3) {$+$};
        \node[text width=1.4cm] at (5.4,0) {$+$ $q^{-1/2}$};
        \draw[gray!40, thick, fill=gray!40, domain=-45:225] plot ({7.25+cos(\x)}, {sin(\x)}) to[out=45, in=130] (7.96, -0.71);
        \draw[thick] (6.54, -0.71) to[out=45, in=130] (7.96, -0.71);
        \node[draw, circle, inner sep=0pt, minimum size=4pt, fill=black] (p3) at (7.25,-0.41) {};
        \draw[thick] (6.54, 0.71) to[out=-60, in=180] (7.25,0) to[out=0, in=240] (7.96,0.71);
        \node at (3.5,-1.4) {$(R_5)$ State Exchange Relation};
   \end{tikzpicture}}
\end{center}

Using relations $(R_1) - (R_4)$, we see that $(R_5)$ is equivalent to the following height exchange relation.
\begin{center}
    \begin{tikzpicture}
        \draw[gray!40, thick, fill=gray!40, domain=-45:225] plot ({cos(\x)}, {sin(\x)}) to[out=45, in=130] (0.71, -0.71);
        \draw[thick] (-0.71, -0.71) to[out=45, in=135] (0.71, -0.71);
        \node[draw, circle, inner sep=0pt, minimum size=4pt, fill=black] (p1) at (0,-0.41) {};
        \draw[thick] (p1) -- (-0.71,0.71);
        \draw[thick] (0.71,0.71) -- (0.1,-0.26);
        \node (s1) at (-0.6,-0.3) {$+$};
        \node (s2) at (0.6,-0.3) {$-$};
        \node[text width=1.1cm] at (1.8,0) {$= q^{-3}$};
        \draw[gray!40, thick, fill=gray!40, domain=-45:225] plot ({3.5+cos(\x)}, {sin(\x)}) to[out=45, in=130] (4.21, -0.71);
        \draw[thick] (2.79, -0.71) to[out=45, in=130] (4.21, -0.71);
        \node[draw, circle, inner sep=0pt, minimum size=4pt, fill=black] (p2) at (3.5,-0.41) {};
        \draw[thick] (3.4,-0.26) -- (2.79,0.71);
        \draw[thick] (p2) -- (4.21,0.71);
        \node (s3) at (2.9,-0.3) {$+$};
        \node (s4) at (4.1,-0.3) {$-$};
        \node[text width=3.5cm] at (6.5,0) {$+$ $q^{-3/2} \left( q^2 - q^{-2} \right)$};
        \draw[gray!40, thick, fill=gray!40, domain=-45:225] plot ({9.25+cos(\x)}, {sin(\x)}) to[out=45, in=130] (9.96, -0.71);
        \draw[thick] (8.54, -0.71) to[out=45, in=130] (9.96, -0.71);
        \node[draw, circle, inner sep=0pt, minimum size=4pt, fill=black] (p3) at (9.25,-0.41) {};
        \draw[thick] (8.54, 0.71) to[out=-60, in=180] (9.25,0) to[out=0, in=240] (9.96,0.71);
        \node at (4.6,-1.4) {$(R_6)$ Height Exchange Relation};
    \end{tikzpicture}
\end{center}

Lastly, we can also quickly find all trivial arc relations. Below are the trivial arc relations corresponding to different states, $(R_3)$.
$$\begin{tikzpicture}[baseline=-3]
    \draw[gray!40, thick, fill=gray!40, domain=-45:225] plot ({cos(\x)}, {sin(\x)}) to[out=45, in=130] (0.71, -0.71);
    \draw[thick] (-0.71, -0.71) to[out=45, in=135] (0.71, -0.71);
    \node[draw, circle, inner sep=0pt, minimum size=4pt, fill=black] (p1) at (0,-0.41) {};
    \draw[thick] (p1) to[out=135, in=-60] (-0.31, 0) to[out=120, in=180] (0, 0.45) to[out=0, in=60] (0.31, 0) to[out=-120, in=30] (0.15,-0.2);
    \node (s1) at (-0.6,-0.3) {$+$};
    \node (s2) at (0.6,-0.3) {$-$};
\end{tikzpicture} = -q^{5/2} \quad\quad
\begin{tikzpicture}[baseline=-3]
    \draw[gray!40, thick, fill=gray!40, domain=-45:225] plot ({cos(\x)}, {sin(\x)}) to[out=45, in=130] (0.71, -0.71);
    \draw[thick] (-0.71, -0.71) to[out=45, in=135] (0.71, -0.71);
    \node[draw, circle, inner sep=0pt, minimum size=4pt, fill=black] (p1) at (0,-0.41) {};
    \draw[thick] (p1) to[out=45, in=-120] (0.31,0) to[out=60, in=0] (0,0.45) to[out=180, in=120] (-0.31,0) to[out=-60, in=150] (-0.15,-0.2);
    \node (s1) at (-0.6,-0.3) {$-$};
    \node (s2) at (0.6,-0.3) {$+$};
\end{tikzpicture} = q^{1/2} \quad\quad
\begin{tikzpicture}[baseline=-3]
    \draw[gray!40, thick, fill=gray!40, domain=-45:225] plot ({cos(\x)}, {sin(\x)}) to[out=45, in=130] (0.71, -0.71);
    \draw[thick] (-0.71, -0.71) to[out=45, in=135] (0.71, -0.71);
    \node[draw, circle, inner sep=0pt, minimum size=4pt, fill=black] (p1) at (0,-0.41) {};
    \draw[thick] (p1) to[out=45, in=-120] (0.31,0) to[out=60, in=0] (0,0.45) to[out=180, in=120] (-0.31,0) to[out=-60, in=150] (-0.15,-0.2);
    \node (s1) at (-0.6,-0.3) {$+$};
    \node (s2) at (0.6,-0.3) {$-$};
\end{tikzpicture} = -q^{-5/2} \quad\quad
\begin{tikzpicture}[baseline=-3]
    \draw[gray!40, thick, fill=gray!40, domain=-45:225] plot ({cos(\x)}, {sin(\x)}) to[out=45, in=130] (0.71, -0.71);
    \draw[thick] (-0.71, -0.71) to[out=45, in=135] (0.71, -0.71);
    \node[draw, circle, inner sep=0pt, minimum size=4pt, fill=black] (p1) at (0,-0.41) {};
    \draw[thick] (p1) to[out=135, in=-60] (-0.31, 0) to[out=120, in=180] (0, 0.45) to[out=0, in=60] (0.31, 0) to[out=-120, in=30] (0.15,-0.2);
    \node (s1) at (-0.6,-0.3) {$-$};
    \node (s2) at (0.6,-0.3) {$+$};
\end{tikzpicture} = q^{-1/2}$$
The first two diagrams, as well as the last two diagrams, differ by a twist, which introduces a multiplicative factor of $-q^{-3}$. For consistency and convenience, we adopt L\^{e}'s notation from \cite{Le18} for these constants. That is,
$$C_{+}^{+} = 0, \quad\quad
C_{+}^{-} = -q^{-5/2}, \quad\quad
C_{-}^{+} = q^{-1/2}, \quad\quad
C_{-}^{-} = 0.$$

Analogous to the Kauffman bracket, if $M = \Sigma \times I$ and $\mathcal{N} = \mathcal{P} \times I$ for $\mathcal{P} \subset \partial \Sigma$, then we can endow $\mathscr{S}(M, \mathcal{N})$ with an algebra structure by defining a product $\alpha \alpha'$ as stacking $\alpha$ above $\alpha'$.
\begin{center}
    \begin{tikzpicture}
        \draw[thick, fill=gray!40] (-1,1) -- (1,1) -- (1,-1) -- (-1,-1) -- (-1,1);
        \node[draw, circle, inner sep=0pt, minimum size=4pt, fill=black] (p11) at (-1,0) {};
        \node[draw, circle, inner sep=0pt, minimum size=4pt, fill=black] (p12) at (0,1) {};
        \node[draw, circle, inner sep=0pt, minimum size=4pt, fill=black] (p13) at (1,0) {};
        \node[draw, circle, inner sep=0pt, minimum size=4pt, fill=black] (p14) at (0,-1) {};
        \draw[thick] (p12) -- (p14);
        \node at (0.2, 0.8) {\scriptsize{$+$}};
        \node at (0.2, -0.8) {\scriptsize{$+$}};
        \node at (0, -1.3) {$\alpha$};
        \node at (1.5,0) {$\cdot$};
        \draw[thick, fill=gray!40] (2,1) -- (4,1) -- (4,-1) -- (2,-1) -- (2,1);
        \node[draw, circle, inner sep=0pt, minimum size=4pt, fill=black] (p21) at (2,0) {};
        \node[draw, circle, inner sep=0pt, minimum size=4pt, fill=black] (p22) at (3,1) {};
        \node[draw, circle, inner sep=0pt, minimum size=4pt, fill=black] (p23) at (4,0) {};
        \node[draw, circle, inner sep=0pt, minimum size=4pt, fill=black] (p24) at (3,-1) {};
        \draw[thick] (p21) -- (p23);
        \node at (2.2, 0.2) {\scriptsize{$-$}};
        \node at (3.8, 0.2) {\scriptsize{$-$}};
        \node at (3, -1.3) {$\alpha'$};
        \node at (4.5,0) {$=$};
        \draw[thick, fill=gray!40] (5,1) -- (7,1) -- (7,-1) -- (5,-1) -- (5,1);
        \node[draw, circle, inner sep=0pt, minimum size=4pt, fill=black] (p31) at (5,0) {};
        \node[draw, circle, inner sep=0pt, minimum size=4pt, fill=black] (p32) at (6,1) {};
        \node[draw, circle, inner sep=0pt, minimum size=4pt, fill=black] (p33) at (7,0) {};
        \node[draw, circle, inner sep=0pt, minimum size=4pt, fill=black] (p34) at (6,-1) {};
        \draw[thick] (p31) -- (5.9,0);
        \draw[thick] (6.1,0) -- (p33);
        \draw[thick] (p32) -- (p34);
        \node at (6.2, 0.8) {\scriptsize{$+$}};
        \node at (6.2, -0.8) {\scriptsize{$+$}};
        \node at (5.2, 0.2) {\scriptsize{$-$}};
        \node at (6.8, 0.2) {\scriptsize{$-$}};
        \node at (6, -1.3) {$\alpha\alpha'$};
    \end{tikzpicture}
\end{center}

Throughout this work, we restrict ourselves to manifolds with at most one boundary component and only one marking. Additionally, the boundary of our manifolds should be obvious. For these reasons, we will often omit $\mathcal{N}$ in the notation. Additionally, when $M = \Sigma \times I$ for a marked surface $\Sigma$, we write $\mathscr{S}(\Sigma)$ for $\mathscr{S}(M)$ to emphasize the additional algebra structure, which comes from stacking in the vertical direction.

The familiar reader will recognize that this definition of a stated skein algebra slightly differs from the typical definition found in the literature. Our approach to stated tangles is somewhat closer to the notion of \textit{ideal arcs} described in \cite{CL22, LY22}. We will discuss and use the conventional model later in Section \ref{section:mods3mflds}.

\subsection{Generators of $\mathscr{S}(T^2 \setminus D^2)$}\label{section:Generators}
In this section we will show that the algebra $\mathscr{S}(T^2 \setminus D^2)$ with one marking is generated by the twelve elements $B = \{ X_{1,0}(\mu_1, \nu_1), X_{2,0}(\mu_2, \nu_2), X_{3,0}(\mu_3, \nu_3) \, \mid \mu_i, \nu_i \in \{ \pm \} \}$, where
$$X_{1,0}(\mu_1, \nu_1) =
\begin{tikzpicture}[baseline=-1]
    \draw[draw=none, fill=gray!40] (0,1) -- (2,1) -- (2,-1) -- (0,-1) -- (1, -0.15) to[out=0, in=-90] (1.15, 0) to[out=90, in=0] (1, 0.15) to[out=180, in=90] (0.85, 0) to[out=-90, in=180] (1, -0.15) -- (0, -1) -- (0,1);
    \draw[thick] (0, 0.15) -- (0.85, 0.15);
    \draw[thick] (1, 0.15) -- (2, 0.15);
    \draw[thick, ->] (0,1) -- (0.9,1);
    \draw[thick, ->] (0.85,1) -- (1.2,1);
    \draw[thick] (1.2,1) -- (2,1);
    \draw[thick, ->] (0,-1) -- (0.9,-1);
    \draw[thick, ->] (0.85,-1) -- (1.2,-1);
    \draw[thick] (1.2,-1) -- (2,-1);
    \draw[thick, ->] (0,-1) -- (0,0);
    \draw[thick] (0,-1) -- (0,1);
    \draw[thick, ->] (2,-1) -- (2,0);
    \draw[thick] (2,-1) -- (2,1);
    \draw[thick, black] (1, 0) circle (0.15);
    \node[draw, circle, inner sep=0pt, minimum size=3pt, fill=white] at (1, 0.15) {};
    \node at (0.7, 0.3) {\scriptsize{$\mu_1$}};
    \node at (1.3, 0.3)  {\scriptsize{$\nu_1$}};
\end{tikzpicture} \quad
X_{2,0}(\mu_2, \nu_2) = \begin{tikzpicture}[baseline=-1]
    \draw[draw=none, fill=gray!40] (4,1) -- (6,1) -- (6,-1) -- (4,-1) -- (5, -0.15) to[out=0, in=-90] (5.15, 0) to[out=90, in=0] (5, 0.15) to[out=180, in=90] (4.85, 0) to[out=-90, in=180] (5, -0.15) -- (4, -1) -- (4,1);
    \draw[thick] (4.7, -1) to[out=90, in=270] (4.7, 0) to[out=90, in=180] (4.85, 0.15);
    \draw[thick] (5, 0.15) to[out=120, in=270] (4.7, 1);
    \draw[thick, ->] (4,1) -- (4.9,1);
    \draw[thick, ->] (4.85,1) -- (5.2,1);
    \draw[thick] (5.2,1) -- (6,1);
    \draw[thick, ->] (4,-1) -- (4.9,-1);
    \draw[thick, ->] (4.85,-1) -- (5.2,-1);
    \draw[thick] (5.2,-1) -- (6,-1);
    \draw[thick, ->] (4,-1) -- (4,0);
    \draw[thick] (4,-1) -- (4,1);
    \draw[thick, ->] (6,-1) -- (6,0);
    \draw[thick] (6,-1) -- (6,1);
    \draw[thick, black] (5, 0) circle (0.15);
    \node[draw, circle, inner sep=0pt, minimum size=3pt, fill=white] at (5, 0.15) {};
    \node at (4.5, 0) {\scriptsize{$\mu_2$}};
    \node at (5.2, 0.3)  {\scriptsize{$\nu_2$}};
\end{tikzpicture}\quad
X_{3,0}(\mu_3, \nu_3) =\begin{tikzpicture}[baseline=-1]
    \draw[draw=none, fill=gray!40] (8,1) -- (10,1) -- (10,-1) -- (8,-1) -- (9, -0.15) to[out=0, in=-90] (9.15, 0) to[out=90, in=0] (9, 0.15) to[out=180, in=90] (8.85, 0) to[out=-90, in=180] (9, -0.15) -- (8, -1) -- (8,1);
    \draw[thick] (8.4, -1) -- (8.4, -0.2) to[out=90, in=180] (8.85, 0.15);
    \draw[thick] (9, 0.15) to[out=45, in=180] (10, 0.6);
    \draw[thick] (8, 0.6) to[out=0, in=270] (8.4, 1);
    \draw[thick, ->] (8,1) -- (8.9,1);
    \draw[thick, ->] (8.85,1) -- (9.2,1);
    \draw[thick] (9.2,1) -- (10,1);
    \draw[thick, ->] (8,-1) -- (8.9,-1);
    \draw[thick, ->] (8.85,-1) -- (9.2,-1);
    \draw[thick] (9.2,-1) -- (10,-1);
    \draw[thick, ->] (8,-1) -- (8,0);
    \draw[thick] (8,-1) -- (8,1);
    \draw[thick, ->] (10,-1) -- (10,0);
    \draw[thick] (10,-1) -- (10,1);
    \draw[thick, black] (9, 0) circle (0.15);
    \node[draw, circle, inner sep=0pt, minimum size=3pt, fill=white] at (9, 0.15) {};
    \node at (8.7, 0.3) {\scriptsize{$\mu_3$}};
    \node at (9.4, 0.15)  {\scriptsize{$\nu_3$}};
\end{tikzpicture}$$

As we're only considering $\mathscr{S}(T^2 \setminus D^2)$ with a single marking, there is only one simple parallel tangle and one boundary curve.
Let $Y_1$, $Y_2$, and $Y_3$ be the meridian, longitude, and (1,1)-curve, respectively.
The boundary curve can be expressed completely in terms of the $Y_i$ and constants, as a quick calculation shows that
\begin{align*}
    \begin{tikzpicture}[baseline=-1]
        \MarkedTorusBackground
        \draw[thick, fill=none] (1,0) circle (0.5);
        \node[draw, circle, inner sep=0pt, minimum size=3pt, fill=white] at (1, 0.15) {};
    \end{tikzpicture} = q Y_1 Y_2 Y_3 - q^2 Y_1^2 - q^{-2} Y_2^2 - q^2 Y_3^2 + q^2 + q^{-2},
\end{align*}
and for any $\mu, \nu \in \left\{\pm\right\}$,
\begin{align*}
    \begin{tikzpicture}[baseline=-1]
        \MarkedTorusBackground
        \draw[thick] (1, 0.15) to[out=0, in=90] (1.3, -0.2) to[out=270, in=0] (1, -0.5) to[out=180, in=270] (0.7, -0.2) to[out=90, in=210] (0.87, 0.15);
        \node[draw, circle, inner sep=0pt, minimum size=3pt, fill=white] at (1, 0.15) {};
        \node at (0.7, 0.2) {$\mu$};
        \node at (1.3, 0.2) {$\nu$};
    \end{tikzpicture} = q Y_1 Y_2 X_{3,0}\left(\mu, \nu\right) - q^2 X_{1,0}\left(\mu, \nu\right) Y_1 - q^{-2} X_{2,0}\left(\mu, \nu\right) Y_2 - q^2 X_{3,0}\left(\mu, \nu\right) Y_3 - C_{\mu}^{\nu}.
\end{align*}

Consider the stated skein algebra generated by $B$. Using the state-exchange relation, we can express each $Y_i$ in this algebra as
$$Y_i = q^{1/2} X_{i,0}(+,-) - q^{5/2} X_{i,0}(-,+)$$
for $i=\{1,2,3\}$. Moreover, we can interchange heights using the proper height exchange relation. We will disregard the states in the following lemmas, as the proofs are independent of the possible states. Additionally, we will overlook relative heights along our marking, as we can readily obtain all possible heights using $(R_6)$ alongside $Y_1$, $Y_2$, and $Y_3$.

While simple closed (non-boundary) curves on the torus (with boundary) are conventionally classified by their slopes, which are numbers in $\mathbb{Q} \cup \frac{1}{0}$, one might expect a similar classification for our tangles, considering that each tangle must start and end at the same marking.
However, this analogy breaks down when considering tangles with what we'll call ``twists'' around the boundary.

For instance, if we trace along the path of $X_{3,0}$ and introduce a twist around the boundary just before reaching the marking, we get a different element back.
$$
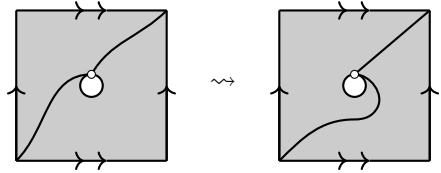
\begin{tikzpicture}[baseline=-1]
    \MarkedTorusBackground
    \draw[thick] (0, -1) to[out=45, in=180] (1, 0.15) to[out=60, in=225] (2, 1);
    \node[draw, circle, inner sep=0pt, minimum size=3pt, fill=white] at (1, 0.15) {};
\end{tikzpicture}
\quad\rightsquigarrow\quad
\begin{tikzpicture}[baseline=-1]
    \MarkedTorusBackground
    \draw[thick] (2, 1) -- (1, 0.15) to[out=0, in=60] (1.3, -0.3) to[out=240, in=0] (1,-0.45) to[out=180, in=45] (0, -1);
    \node[draw, circle, inner sep=0pt, minimum size=3pt, fill=white] at (1, 0.15) {};
\end{tikzpicture}$$
As these twists resemble half Dehn twists around the boundary, we distinguish these twists by indexing over $\frac{1}{2}\mathbb{Z}$ instead of $\mathbb{Z}$. In particular, $X_{i, r}$ is a \textit{full} Dehn twist of $X_{i, r-1}$.
The following examples illustrate $X_{1,r}$ for various choices of $r \in \frac{1}{2}\mathbb{Z}$.
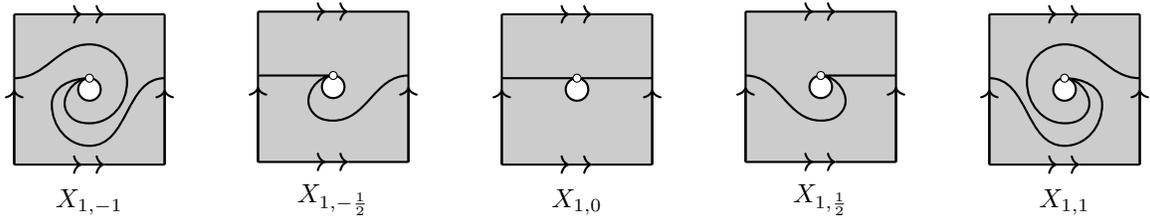
$$
\begin{tikzpicture}
    \MarkedTorusBackground
    \draw[thick] (1, 0.15) to[out=180, in=90] (0.5, -0.2) to[out=270, in=180] (1, -0.75) to[out=0, in=180] (2, 0.15);
    \draw[thick] (1, 0.15) to[out=180, in=120] (0.7, -0.3) to[out=300, in=180] (1,-0.45) to[out=0, in=270] (1.5, 0.1) to[out=90, in=0] (1, 0.6) to[out=180, in=0] (0, 0.15);
    \node[draw, circle, inner sep=0pt, minimum size=3pt, fill=white] at (1, 0.15) {};
    \node at (1,-1.5) {$X_{1,-1}$};
\end{tikzpicture}
\begin{tikzpicture}
    \MarkedTorusBackground
    \draw[thick] (0, 0.15) -- (1, 0.15) to[out=180, in=120] (0.7, -0.3) to[out=300, in=180] (1,-0.45) to[out=0, in=180] (2, 0.15);
    \node[draw, circle, inner sep=0pt, minimum size=3pt, fill=white] at (1, 0.15) {};
    \node at (1,-1.5) {$X_{1,-\frac{1}{2}}$};
    \node (space) at (-1, 0) {};
\end{tikzpicture}
\begin{tikzpicture}
    \MarkedTorusBackground
    \draw[thick] (0, 0.15) -- (1, 0.15) -- (2, 0.15);
    \node[draw, circle, inner sep=0pt, minimum size=3pt, fill=white] at (1, 0.15) {};
    \node at (1,-1.5) {$X_{1,0}$};
    \node (space) at (-1, 0) {};
\end{tikzpicture}
\begin{tikzpicture}
    \MarkedTorusBackground
    \draw[thick] (2, 0.15) -- (1, 0.15) to[out=0, in=60] (1.3, -0.3) to[out=240, in=0] (1,-0.45) to[out=180, in=0] (0, 0.15);
    \node[draw, circle, inner sep=0pt, minimum size=3pt, fill=white] at (1, 0.15) {};
    \node at (1,-1.5) {$X_{1,\frac{1}{2}}$};
    \node (space) at (-1, 0) {};
\end{tikzpicture}
\begin{tikzpicture}
    \MarkedTorusBackground
    \draw[thick] (1, 0.15) to[out=0, in=90] (1.5, -0.2) to[out=270, in=0] (1, -0.75) to[out=180, in=0] (0, 0.15);
    \draw[thick] (1, 0.15) to[out=0, in=60] (1.3, -0.3) to[out=240, in=0] (1,-0.45) to[out=180, in=270] (0.5, 0.1) to[out=90, in=180] (1, 0.6) to[out=0, in=180] (2, 0.15);
    \node[draw, circle, inner sep=0pt, minimum size=3pt, fill=white] at (1, 0.15) {};
    \node at (1,-1.5) {$X_{1,1}$};
    \node (space) at (-1, 0) {};
\end{tikzpicture}
$$

Given any marked surface, $\Sigma$, L\^{e} showed in \cite{Le18} that the set of all isotopy classes of increasingly stated, simple $\partial\Sigma$-tangle diagrams forms a basis for $\mathscr{S}(\Sigma)$. Therefore, our approach will be to first classify all possible simple (stateless) tangles in $T^2 \setminus D^2$ that start and end on $x \in \partial \left( T^2 \setminus D^2 \right)$ and then demonstrate that all basis elements are contained in the subalgebra generated by $\left\{ X_{1,0}(\mu_1, \nu_1), X_{2,0}(\mu_2, \nu_2), X_{3,0}(\mu_3, \nu_3)\right\}$ for all $\mu_i, \nu_j \in \{ \pm \}$.

In Theorem~\ref{theorem:classification}, we classify a large class of closed curves by their slope and number of twists. We construct this classification by covering a neighborhood of $\partial \left( T^2 \setminus D^2 \right)$ with an annulus and proceed to use a Seifert-Van Kampen-style approach to calculate the slope outside the annulus as well as determine the number of twists inside the annulus.

To clean things up a bit, we'll introduce some more notation. 

\begin{definition} We will denote the following set of tuples as
\[ \mathcal{I} := \left.\left\{ (p,q,r) \in \mathbb{Z} \times \mathbb{Z} \times \frac{1}{2}\mathbb{Z} \, \mid \gcd(p,q) = 1 \right\} \right/ \sim \]
where $(p,q,r) \sim (-p,-q,r)$. Additionally, for any $x \in T^2 \setminus D^2$, define $\Omega_{x}$ to be the set of isotopy classes of simple unoriented closed non-parallel curves that begin and end on $x$. 
\end{definition}
Any time we refer to a representative or an isotopic representative of an element in $\Omega_x$, we are always assuming that this representative begins and ends at the point $x$.

\begin{theorem}\label{theorem:classification}
    Let $x \in \partial \left( T^2 \setminus D^2 \right)$. There exists a bijection $f_{A}: \mathcal{I} \longrightarrow \Omega_x$.
\end{theorem}
\begin{proof}
    We will construct $f_{A}$ by first constructing a map, $\widetilde{f}_{A}$, from $\mathcal{I}$ to $\widetilde{\Omega}_x$, the set of simple closed unoriented non-parallel curves that begin and end on $x$, and then projecting onto the corresponding set of isotopy classes of curves.
    \begin{center}
        \begin{tikzcd}
            \mathcal{I} \arrow[rr, hook, "\widetilde{f}_{A}"] \arrow[rrd, "f_{A}"'] & & \widetilde{\Omega}_x \arrow[d, two heads, "\pi_{\text{iso}}"] \\
            & & \Omega_x
        \end{tikzcd}
    \end{center}
    
    We'll first introduce some definitions. Let $\varepsilon > 0$ be sufficiently small. Define $B_{\varepsilon}(m,n)$ to be the open ball of radius $\varepsilon$ at point $(m,n) \in \mathbb{R}^2$, $E_{\varepsilon} := \mathbb{R}^{2} \setminus \left( \bigcup_{(m,n) \in \mathbb{Z}^2} B_{\varepsilon}(m,n) \right)$, and $\Phi_{\varepsilon}$ to be the restriction of the identity map on $\mathbb{R}^2$ to $E_{\varepsilon}$, composed with the obvious covering map.
    $$\Phi_{\varepsilon} : E_{\varepsilon} \longrightarrow T^2 \setminus D^2$$
    Let $A := \Phi_{\varepsilon}\left( \overline{E_{\epsilon} \setminus E_{2\varepsilon}} \right)$ be the closed annulus and define $\partial_A := \partial A \setminus \partial (T^2 \setminus D^2)$, the ``outer boundary''. The element $\partial_A$ will constantly serve as a reference point throughout the rest of this proof and is the dividing bridge between the $\{ p, q \}$ and $\{ r \}$ in the tuple.
    
    Choose any triple $(p,q,r) \in \mathcal{I}$. Let $\gamma_{p,q}$ be the line with slope $p/q$ going through the origin in $\mathbb{R}^2$ and consider $\gamma_0 := \Phi_{\varepsilon}\left( \gamma_{p,q} \right) \cap \left( T^2 \setminus A \right)$. Since $\gamma_{p,q}$ has constant slope, the two endpoints of $\gamma_{0}$, labeled $y_1$ and $y_2$, must lie on $\partial_A$ as antipodal points.
    
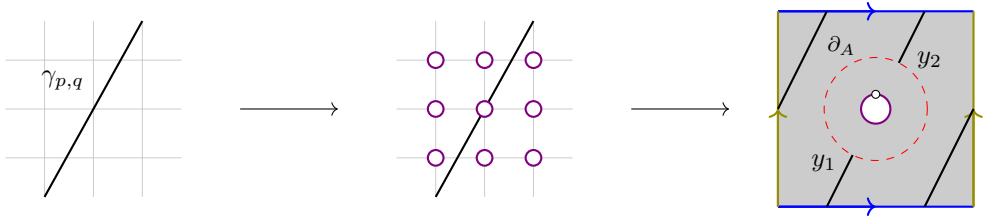
\begin{center}
        \begin{tikzpicture}[scale=1.3]
            \draw[help lines, ystep=0.5, xstep=0.5, color=gray!40] (-4.9, -0.9) grid (-3.1, 0.9);
            \draw[thick] (-4.5, -0.9) -- (-3.5, 0.9);
            \node[text width=0.6cm] at (-4.3, 0.3) {$\gamma_{p,q}$};
            \draw[->] (-2.5, 0) -- (-1.5, 0);
            \draw[help lines, ystep=0.5, xstep=0.5, color=gray!40] (-0.9, -0.9) grid (0.9, 0.9);
            \draw[thick] (-0.5, -0.9) -- (0.5, 0.9);
            \foreach \x in {-0.5, 0, 0.5}
                \foreach \y in {-0.5, 0, 0.5}
                    \draw[thick, draw=violet, fill=white] (\x, \y) circle (0.08);
            \draw[->] (1.5, 0) -- (2.5, 0);
            \draw[draw=none, fill=gray!40] (3,1) -- (5,1) -- (5,-1) -- (3,-1) -- (4, -0.15) to[out=0, in=-90] (4.15, 0) to[out=90, in=0] (4, 0.15) to[out=180, in=90] (3.85, 0) to[out=-90, in=180] (4, -0.15) -- (3, -1) -- (3,1);
            \draw[thick, blue, ->] (3,1) -- (4,1);
            \draw[thick, blue] (3,1) -- (5,1);
            \draw[thick, blue, ->] (3,-1) -- (4,-1);
            \draw[thick, blue] (3,-1) -- (5,-1);
            \draw[thick, olive, ->] (3,-1) -- (3,0);
            \draw[thick, olive] (3,-1) -- (3,1);
            \draw[thick, olive, ->] (5,-1) -- (5,0);
            \draw[thick, olive] (5,-1) -- (5,1);
            \draw[thick] (3.5, -1) -- (4.5, 1);
            \draw[thick] (4.5, -1) -- (5, 0);
            \draw[thick] (3, 0) -- (3.5, 1);
            \draw[draw=red, fill=gray!40, dashed] (4, 0) circle[radius=15pt];
            \draw[thick, violet, fill=white] (4, 0) circle (0.15);
            \node[draw, circle, inner sep=0pt, minimum size=3pt, fill=white] at (4, 0.15) {};
            \node at (3.47, -0.55) {$y_1$};
            \node at (4.55, 0.5) {$y_2$};
            \node at (3.65, 0.65) {\footnotesize{$\partial_A$}};
        \end{tikzpicture}
    \end{center}
    
    Let $x_0$ be the unique point on $\partial_A$ without a unique geodesic to $x$ through $A$. Let $x_1, x_2 \in \partial_A$ be the midpoints between $x_0$ and its antipodal counterpart, labeled in clockwise order from $x_0$, and let $c_1$ be the geodesic path from $x_1$ to $x$ within $A$.
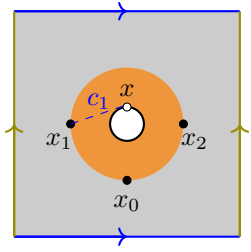
    \begin{center}
        \begin{tikzpicture}[scale=1.5]
            \draw[draw=none, fill=gray!40] (4,1) -- (6,1) -- (6,-1) -- (4,-1) -- (5, -0.15) to[out=0, in=-90] (5.15, 0) to[out=90, in=0] (5, 0.15) to[out=180, in=90] (4.85, 0) to[out=-90, in=180] (5, -0.15) -- (4, -1) -- (4,1);
            \draw[thick, blue, ->] (4,1) -- (5,1);
            \draw[thick, blue] (4,1) -- (6,1);
            \draw[thick, blue, ->] (4,-1) -- (5,-1);
            \draw[thick, blue] (4,-1) -- (6,-1);
            \draw[thick, olive, ->] (4,-1) -- (4,0);
            \draw[thick, olive] (4,-1) -- (4,1);
            \draw[thick, olive, ->] (6,-1) -- (6,0);
            \draw[thick, olive] (6,-1) -- (6,1);
            \draw[draw=none, fill=orange, opacity=0.7] (5,0) circle[radius=0.5cm];
            \node[draw, circle, inner sep=0pt, minimum size=3pt, fill=black] at (5.5, 0) {};
            \node[draw, circle, inner sep=0pt, minimum size=3pt, fill=black] at (4.5, 0) {};
            \node[draw, circle, inner sep=0pt, minimum size=3pt, fill=black] at (5, -0.5) {};
            \draw[thick, black, fill=white] (5, 0) circle (0.15);
            \node[draw, circle, inner sep=0pt, minimum size=3pt, fill=white] at (5, 0.15) {};
            \draw[dashed, blue] (4.5, 0) -- (5, 0.17);
            \node at (4.4, -0.15) {$x_1$};
            \node at (5.6, -0.15) {$x_2$};
            \node at (5, -0.7) {$x_0$};
            \node at (5, 0.3) {$x$};
            \node[blue] at (4.75, 0.2) {$c_1$};
        \end{tikzpicture}
    \end{center}

    We'll assume $y_1$ and $y_2$ are labeled such that $y_1$ is closer to $x_1$ on $\partial_A$. If they are both the same distance away from $x_1$, then we must be in the case some $y_i = x_0$, as $y_1$ and $y_2$ are antipodal points. If this happens, label this $y_i$ (the south pole) as $y_1$ and label the other (the north pole) as $y_2$. Let $\gamma_i$ be the geodesic paths (containing their endpoints) on $\partial_A$ from $y_i$ to $x_i$. This should correspond to a clockwise path when $\frac{p}{q}$ is a positive slope or equal to $\frac{1}{0}$, and a counterclockwise path when $\frac{p}{q}$ is a negative slope.\footnote{Our choice of clockwise rotation when the slope is $\frac{1}{0}$ was made for notational convenience outside of this proof. Alternatively, one could instead use a counterclockwise rotation by swapping the labels $y_1$ and $y_2$. The difference between these classifications results in a shift of $r$ by $1/2$ whenever $(p,q) = (1,0)$.} Define $\delta := \gamma_0 \sqcup \gamma_1 \sqcup \gamma_2$ and notice that this is simple.
    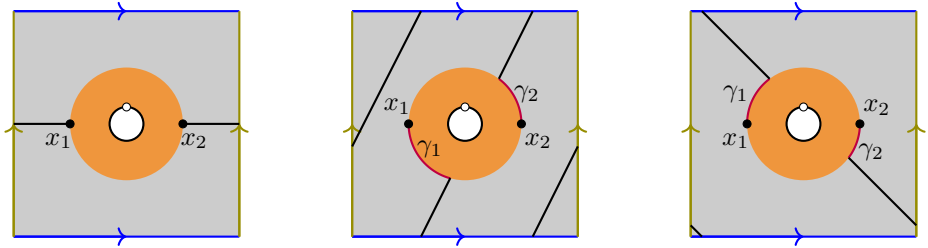
\begin{figure}[h]
    \begin{center}
        \begin{tikzpicture}[scale=1.5]
            \draw[draw=none, fill=gray!40] (0,1) -- (2,1) -- (2,-1) -- (0,-1) -- (1, -0.15) to[out=0, in=-90] (1.15, 0) to[out=90, in=0] (1, 0.15) to[out=180, in=90] (0.85, 0) to[out=-90, in=180] (1, -0.15) -- (0, -1) -- (0,1);
            \draw[thick, blue, ->] (0,1) -- (1,1);
            \draw[thick, blue] (0,1) -- (2,1);
            \draw[thick, blue, ->] (0,-1) -- (1,-1);
            \draw[thick, blue] (0,-1) -- (2,-1);
            \draw[thick, olive, ->] (0,-1) -- (0,0);
            \draw[thick, olive] (0,-1) -- (0,1);
            \draw[thick, olive, ->] (2,-1) -- (2,0);
            \draw[thick, olive] (2,-1) -- (2,1);
            \draw[draw=none, fill=orange, opacity=0.7] (1,0) circle[radius=0.5cm];
            \draw[thick] (0, 0) -- (0.5, 0);
            \draw[thick] (1.5, 0) -- (2, 0);
            \draw[thick, black, fill=white] (1, 0) circle (0.15);
            \node[draw, circle, inner sep=0pt, minimum size=3pt, fill=white] at (1, 0.15) {};
            \node[draw, circle, inner sep=0pt, minimum size=3pt, fill=black] at (1.5, 0) {};
            \node[draw, circle, inner sep=0pt, minimum size=3pt, fill=black] at (0.5, 0) {};
            \node at (0.4, -0.15) {$x_1$};
            \node at (1.6, -0.15) {$x_2$};
            \draw[draw=none, fill=gray!40] (3,1) -- (5,1) -- (5,-1) -- (3,-1) -- (4, -0.15) to[out=0, in=-90] (4.15, 0) to[out=90, in=0] (4, 0.15) to[out=180, in=90] (3.85, 0) to[out=-90, in=180] (4, -0.15) -- (3, -1) -- (3,1);
            \draw[thick, blue, ->] (3,1) -- (4,1);
            \draw[thick, blue] (3,1) -- (5,1);
            \draw[thick, blue, ->] (3,-1) -- (4,-1);
            \draw[thick, blue] (3,-1) -- (5,-1);
            \draw[thick, olive, ->] (3,-1) -- (3,0);
            \draw[thick, olive] (3,-1) -- (3,1);
            \draw[thick, olive, ->] (5,-1) -- (5,0);
            \draw[thick, olive] (5,-1) -- (5,1);
            \draw[draw=none, fill=orange, opacity=0.7] (4,0) circle[radius=0.5cm];
            \draw[thick, purple, domain=180:255] plot ({0.5*cos(\x)+4}, {0.5*sin(\x)});
            \draw[thick, purple, domain=0:53] plot ({0.5*cos(\x)+4}, {0.5*sin(\x)});
            \draw[thick] (3.61, -1) -- (3.87, -0.48);
            \draw[thick] (4.3, 0.4) -- (4.6, 1);
            \draw[thick] (4.6, -1) -- (5, -0.2);
            \draw[thick] (3, -0.2) -- (3.61, 1);
            \draw[thick, black, fill=white] (4, 0) circle (0.15);
            \node[draw, circle, inner sep=0pt, minimum size=3pt, fill=white] at (4, 0.15) {};
            \node[draw, circle, inner sep=0pt, minimum size=3pt, fill=black] at (4.5, 0) {};
            \node[draw, circle, inner sep=0pt, minimum size=3pt, fill=black] at (3.5, 0) {};
            \node at (3.4, 0.16) {$x_1$};
            \node at (4.65, -0.15) {$x_2$};
            \node at (3.69, -0.22) {$\gamma_1$};
            \node at (4.55, 0.27) {$\gamma_2$};
            \draw[draw=none, fill=gray!40] (6,1) -- (8,1) -- (8,-1) -- (6,-1) -- (7, -0.15) to[out=0, in=-90] (7.15, 0) to[out=90, in=0] (7, 0.15) to[out=180, in=90] (6.85, 0) to[out=-90, in=180] (7, -0.15) -- (6, -1) -- (6, 1);
            \draw[thick, blue, ->] (6,1) -- (7,1);
            \draw[thick, blue] (6,1) -- (8,1);
            \draw[thick, blue, ->] (6,-1) -- (7,-1);
            \draw[thick, blue] (6,-1) -- (8,-1);
            \draw[thick, olive, ->] (6,-1) -- (6,0);
            \draw[thick, olive] (6,-1) -- (6,1);
            \draw[thick, olive, ->] (8,-1) -- (8,0);
            \draw[thick, olive] (8,-1) -- (8,1);
            \draw[draw=none, fill=orange, opacity=0.7] (7,0) circle[radius=0.5cm];
            \draw[thick, purple, domain=127:180] plot ({0.5*cos(\x)+7}, {0.5*sin(\x)});
            \draw[thick, purple, domain=-37:0] plot ({0.5*cos(\x)+7}, {0.5*sin(\x)});
            \draw[thick] (6.1, 1) -- (6.7, 0.4);
            \draw[thick] (7.4, -0.3) -- (8, -0.9);
            \draw[thick] (6, -0.9) -- (6.1, -1);
            \draw[thick, black, fill=white] (7, 0) circle (0.15);
            \node[draw, circle, inner sep=0pt, minimum size=3pt, fill=white] at (7, 0.15) {};
            \node[draw, circle, inner sep=0pt, minimum size=3pt, fill=black] at (7.5, 0) {};
            \node[draw, circle, inner sep=0pt, minimum size=3pt, fill=black] at (6.5, 0) {};
            \node at (6.4, -0.15) {$x_1$};
            \node at (7.65, 0.15) {$x_2$};
            \node at (6.4, 0.27) {$\gamma_1$};
            \node at (7.6, -0.25) {$\gamma_2$};
        \end{tikzpicture}
        \caption[Constructing geodesic paths, $\gamma_i$]{Left: When $(p,q) = (0,1)$ we get $y_1 = x_1$ and $y_2 = x_2$ and so each $\gamma_i$ is trivial. Middle: When $\gamma_0$ has a positive slope, we trace out a geodesic path clockwise from $y_i$ to $x_i$. Right: When $\gamma_0$ has a negative slope we move counterclockwise instead.}
    \end{center}
    \end{figure}

    Recall that the mapping class group, $\text{Mod}\left(\Sigma\right)$, of a surface, $\Sigma$, is the group of isotopy classes of elements of $\operatorname{Homeo}_{+}(\Sigma, \partial \Sigma)$, which pointwise-fix $\partial \Sigma$. Also recall that the mapping class group of the annulus is $\text{Mod}(A) = \langle \sigma \rangle$ where $\sigma$ is the clockwise Dehn twist along the curve $\varsigma$, parallel to the boundaries.
    
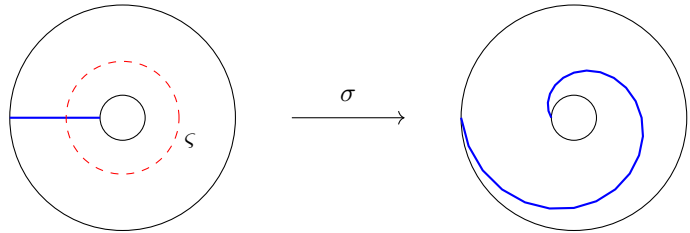
\begin{center}
        \begin{tikzpicture}[scale=1.5]
            \draw (1, 0) circle (1);
            \draw (1, 0) circle (0.2);
            \draw[thick, blue] (0, 0) -- (0.8, 0);
            \draw[dashed, red] (1, 0) circle (0.5);
            \node at (1.6, -0.2) {$\varsigma$};
            \draw[->] (2.5, 0) -- (3.5, 0);
            \node at (3, 0.2) {$\sigma$};
            \draw (5, 0) circle (1);
            \draw (5, 0) circle (0.2);
            \draw[thick, blue, domain=0:1] plot ({5+(1-0.8*\x)*cos(360*\x + 180)}, {(1-0.8*\x)*sin(360*\x + 180)});
        \end{tikzpicture}
    \end{center}
    For $r \in \frac{1}{2}\mathbb{Z}$, notice that there is a unique decomposition, $r = \frac{r_1 + r_2}{2}$, such that $r_1$,$r_2 \in \mathbb{Z}$ and $r_1 - r_2  \in \{0,1\}$. Let $\alpha_1 = \sigma^{r_1} c_1$ and consider this curve in $A$.
    
    Define $B$ to be the filled-in square with corners labeled $\{ v_1, v_2, v_3, v_4 \}$, indexed clockwise, and edges labeled $\{ e_1, e_2, e_3, e_{4} \}$ such that $e_i$ has endpoints $v_i$ and $v_{i+1}$ for $i \mod 4$, and let $z$ be a point on the interior of $e_{4}$.
    
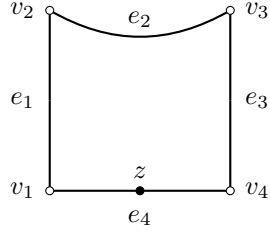
\begin{center}
        \begin{tikzpicture}[scale=1.2]
            \draw[thick] (8, 1) to[out=330, in=210] (10,1);
            \draw[thick] (8,-1) -- (10,-1);
            \draw[thick] (8,1) -- (8,0);
            \draw[thick] (8,0) -- (8,-1);
            \draw[thick] (10,1) -- (10,0);
            \draw[thick] (10,0) -- (10,-1);
            \node[draw, circle, inner sep=0pt, minimum size=3pt, fill=white] at (8, 1) {};
            \node[draw, circle, inner sep=0pt, minimum size=3pt, fill=white] at (10, 1) {};
            \node[draw, circle, inner sep=0pt, minimum size=3pt, fill=white] at (8, -1) {};
            \node[draw, circle, inner sep=0pt, minimum size=3pt, fill=white] at (10, -1) {};
            \node[draw, circle, inner sep=0pt, minimum size=3pt, fill=black] at (9, -1) {};
            \node at (7.7, -1) {$v_1$};
            \node at (7.7, 1) {$v_2$};
            \node at (10.3, 1) {$v_3$};
            \node at (10.3, -1) {$v_4$};
            \node at (7.7, 0) {$e_1$};
            \node at (9, 0.9) {$e_2$};
            \node at (10.3, 0) {$e_3$};
            \node at (9, -1.3) {$e_4$};
            \node at (9, -0.8) {$z$};
        \end{tikzpicture}
    \end{center}
    Take $g_{\alpha_1} : B \to A$ to be the quotient map such that
    \begin{itemize}
        \item $g_{\alpha_1} (e_1) = g_{\alpha_1}(e_3) = \alpha_1$,
        \item $g_{\alpha_1}(v_2) = g_{\alpha_1}(v_3) = x$,
        \item $g_{\alpha_1}(v_1) = g_{\alpha_1}(v_4) = x_1$,
        \item $g_{\alpha_1}(z) = x_2$,
        \item and canonically identifies $e_4$ to $\partial_A$ and $e_2$ to $\partial \left(T^2 \setminus D^2\right)$.
    \end{itemize}
    Using a pullback of the induced metric on $T^2 \setminus D^2$, if $r_2 = r_1 - 1$ define $\beta_2$ to be the geodesic path in $B$ from $z$ to $v_{2}$ and to be the geodesic path from $z$ to $v_{3}$ if $r_2 = r_1$. Finally, let $\alpha_2 := g_{\alpha_1}(\beta_2)$. Then the curve $\alpha := \delta \cup \alpha_1 \cup \alpha_2$ is a simple tangle with endpoints on $x$. We now finally define $\widetilde{f}_{A}(p,q,r) = \alpha$ and $f_{A}(p,q,r) = [\alpha]$.
    
    The rest of this proof is primarily devoted to showing $f_{A}$ is surjective.
    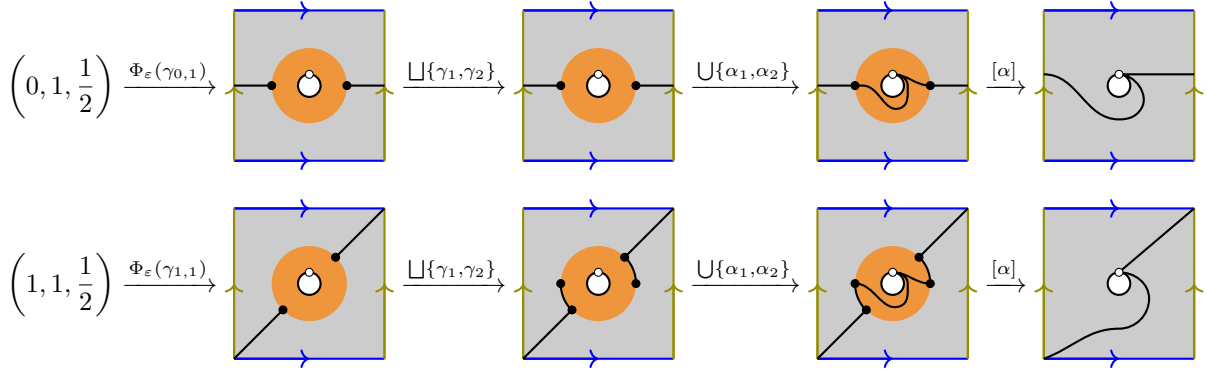
\begin{figure}[h]
        \centering
        $$\left(0,1,\frac{1}{2}\right) \xrightarrow{\Phi_{\varepsilon}(\gamma_{0,1})}
        \begin{tikzpicture}[scale=1, baseline=-3]
            \draw[draw=none, fill=gray!40] (0,1) -- (2,1) -- (2,-1) -- (0,-1) -- (1, -0.15) to[out=0, in=-90] (1.15, 0) to[out=90, in=0] (1, 0.15) to[out=180, in=90] (0.85, 0) to[out=-90, in=180] (1, -0.15) -- (0, -1) -- (0,1);
            \draw[thick, blue, ->] (0,1) -- (1,1);
            \draw[thick, blue] (0,1) -- (2,1);
            \draw[thick, blue, ->] (0,-1) -- (1,-1);
            \draw[thick, blue] (0,-1) -- (2,-1);
            \draw[thick, olive, ->] (0,-1) -- (0,0);
            \draw[thick, olive] (0,-1) -- (0,1);
            \draw[thick, olive, ->] (2,-1) -- (2,0);
            \draw[thick, olive] (2,-1) -- (2,1);
            \draw[draw=none, fill=orange, opacity=0.7] (1,0) circle[radius=0.5cm];
            \draw[thick] (0, 0) -- (0.5, 0);
            \draw[thick] (1.5, 0) -- (2, 0);
            \draw[thick, black, fill=white] (1, 0) circle (0.15);
            \node[draw, circle, inner sep=0pt, minimum size=3pt, fill=white] at (1, 0.15) {};
            \node[draw, circle, inner sep=0pt, minimum size=3pt, fill=black] at (1.5, 0) {};
            \node[draw, circle, inner sep=0pt, minimum size=3pt, fill=black] at (0.5, 0) {};
        \end{tikzpicture} \xrightarrow{\bigsqcup \{ \gamma_1, \gamma_2 \}}
        \begin{tikzpicture}[scale=1, baseline=-3]
            \draw[draw=none, fill=gray!40] (0,1) -- (2,1) -- (2,-1) -- (0,-1) -- (1, -0.15) to[out=0, in=-90] (1.15, 0) to[out=90, in=0] (1, 0.15) to[out=180, in=90] (0.85, 0) to[out=-90, in=180] (1, -0.15) -- (0, -1) -- (0,1);
            \draw[thick, blue, ->] (0,1) -- (1,1);
            \draw[thick, blue] (0,1) -- (2,1);
            \draw[thick, blue, ->] (0,-1) -- (1,-1);
            \draw[thick, blue] (0,-1) -- (2,-1);
            \draw[thick, olive, ->] (0,-1) -- (0,0);
            \draw[thick, olive] (0,-1) -- (0,1);
            \draw[thick, olive, ->] (2,-1) -- (2,0);
            \draw[thick, olive] (2,-1) -- (2,1);
            \draw[draw=none, fill=orange, opacity=0.7] (1,0) circle[radius=0.5cm];
            \draw[thick] (0, 0) -- (0.5, 0);
            \draw[thick] (1.5, 0) -- (2, 0);
            \draw[thick, black, fill=white] (1, 0) circle (0.15);
            \node[draw, circle, inner sep=0pt, minimum size=3pt, fill=white] at (1, 0.15) {};
            \node[draw, circle, inner sep=0pt, minimum size=3pt, fill=black] at (1.5, 0) {};
            \node[draw, circle, inner sep=0pt, minimum size=3pt, fill=black] at (0.5, 0) {};
        \end{tikzpicture} \xrightarrow{\bigcup \{ \alpha_1, \alpha_2 \}}
        \begin{tikzpicture}[scale=1, baseline=-3]
            \draw[draw=none, fill=gray!40] (0,1) -- (2,1) -- (2,-1) -- (0,-1) -- (1, -0.15) to[out=0, in=-90] (1.15, 0) to[out=90, in=0] (1, 0.15) to[out=180, in=90] (0.85, 0) to[out=-90, in=180] (1, -0.15) -- (0, -1) -- (0,1);
            \draw[thick, blue, ->] (0,1) -- (1,1);
            \draw[thick, blue] (0,1) -- (2,1);
            \draw[thick, blue, ->] (0,-1) -- (1,-1);
            \draw[thick, blue] (0,-1) -- (2,-1);
            \draw[thick, olive, ->] (0,-1) -- (0,0);
            \draw[thick, olive] (0,-1) -- (0,1);
            \draw[thick, olive, ->] (2,-1) -- (2,0);
            \draw[thick, olive] (2,-1) -- (2,1);
            \draw[draw=none, fill=orange, opacity=0.7] (1,0) circle[radius=0.5cm];
            \draw[thick] (0, 0) -- (0.6, 0) to[out=0, in=200] (1.1, -0.3) to[out=20, in=270] (1.2, 0) to[out=90, in=0] (1, 0.15) to[out=0, in=180] (1.5, 0) -- (2, 0);
            \draw[thick, black, fill=white] (1, 0) circle (0.15);
            \node[draw, circle, inner sep=0pt, minimum size=3pt, fill=white] at (1, 0.15) {};
            \node[draw, circle, inner sep=0pt, minimum size=3pt, fill=black] at (1.5, 0) {};
            \node[draw, circle, inner sep=0pt, minimum size=3pt, fill=black] at (0.5, 0) {};
        \end{tikzpicture} \xrightarrow{[\alpha]}
        \begin{tikzpicture}[scale=1, baseline=-3]
            \draw[draw=none, fill=gray!40] (0,1) -- (2,1) -- (2,-1) -- (0,-1) -- (1, -0.15) to[out=0, in=-90] (1.15, 0) to[out=90, in=0] (1, 0.15) to[out=180, in=90] (0.85, 0) to[out=-90, in=180] (1, -0.15) -- (0, -1) -- (0,1);
            \draw[thick, blue, ->] (0,1) -- (1,1);
            \draw[thick, blue] (0,1) -- (2,1);
            \draw[thick, blue, ->] (0,-1) -- (1,-1);
            \draw[thick, blue] (0,-1) -- (2,-1);
            \draw[thick, olive, ->] (0,-1) -- (0,0);
            \draw[thick, olive] (0,-1) -- (0,1);
            \draw[thick, olive, ->] (2,-1) -- (2,0);
            \draw[thick, olive] (2,-1) -- (2,1);
            \draw[thick] (2, 0.15) -- (1, 0.15) to[out=0, in=60] (1.3, -0.3) to[out=240, in=0] (1,-0.45) to[out=180, in=0] (0, 0.15);
            \draw[thick, black, fill=white] (1, 0) circle (0.15);
            \node[draw, circle, inner sep=0pt, minimum size=3pt, fill=white] at (1, 0.15) {};
        \end{tikzpicture}$$
        $$\left(1,1,\frac{1}{2}\right) \xrightarrow{\Phi_{\varepsilon}(\gamma_{1,1})}
        \begin{tikzpicture}[scale=1, baseline=-3]
            \draw[draw=none, fill=gray!40] (0,1) -- (2,1) -- (2,-1) -- (0,-1) -- (1, -0.15) to[out=0, in=-90] (1.15, 0) to[out=90, in=0] (1, 0.15) to[out=180, in=90] (0.85, 0) to[out=-90, in=180] (1, -0.15) -- (0, -1) -- (0,1);
            \draw[thick, blue, ->] (0,1) -- (1,1);
            \draw[thick, blue] (0,1) -- (2,1);
            \draw[thick, blue, ->] (0,-1) -- (1,-1);
            \draw[thick, blue] (0,-1) -- (2,-1);
            \draw[thick, olive, ->] (0,-1) -- (0,0);
            \draw[thick, olive] (0,-1) -- (0,1);
            \draw[thick, olive, ->] (2,-1) -- (2,0);
            \draw[thick, olive] (2,-1) -- (2,1);
            \draw[draw=none, fill=orange, opacity=0.7] (1,0) circle[radius=0.5cm];
            \draw[thick] (0, -1) -- (0.65, -0.35);
            \draw[thick] (1.35, 0.35) -- (2, 1);
            \draw[thick, black, fill=white] (1, 0) circle (0.15);
            \node[draw, circle, inner sep=0pt, minimum size=3pt, fill=white] at (1, 0.15) {};
            \node[draw, circle, inner sep=0pt, minimum size=3pt, fill=black] at (1.35, 0.35) {};
            \node[draw, circle, inner sep=0pt, minimum size=3pt, fill=black] at (0.65, -0.35) {};
        \end{tikzpicture} \xrightarrow{\bigsqcup \{ \gamma_1, \gamma_2 \}}
        \begin{tikzpicture}[scale=1, baseline=-3]
            \draw[draw=none, fill=gray!40] (0,1) -- (2,1) -- (2,-1) -- (0,-1) -- (1, -0.15) to[out=0, in=-90] (1.15, 0) to[out=90, in=0] (1, 0.15) to[out=180, in=90] (0.85, 0) to[out=-90, in=180] (1, -0.15) -- (0, -1) -- (0,1);
            \draw[thick, blue, ->] (0,1) -- (1,1);
            \draw[thick, blue] (0,1) -- (2,1);
            \draw[thick, blue, ->] (0,-1) -- (1,-1);
            \draw[thick, blue] (0,-1) -- (2,-1);
            \draw[thick, olive, ->] (0,-1) -- (0,0);
            \draw[thick, olive] (0,-1) -- (0,1);
            \draw[thick, olive, ->] (2,-1) -- (2,0);
            \draw[thick, olive] (2,-1) -- (2,1);
            \draw[draw=none, fill=orange, opacity=0.7] (1,0) circle[radius=0.5cm];
            \draw[thick] (0, -1) -- (0.65, -0.35);
            \draw[thick, domain=180:225] plot ({0.5*cos(\x)+1}, {0.5*sin(\x)});
            \draw[thick, domain=0:45] plot ({0.5*cos(\x)+1}, {0.5*sin(\x)});
            \draw[thick] (1.35, 0.35) -- (2, 1);
            \draw[thick, black, fill=white] (1, 0) circle (0.15);
            \node[draw, circle, inner sep=0pt, minimum size=3pt, fill=white] at (1, 0.15) {};
            \node[draw, circle, inner sep=0pt, minimum size=3pt, fill=black] at (1.35, 0.35) {};
            \node[draw, circle, inner sep=0pt, minimum size=3pt, fill=black] at (0.65, -0.35) {};
            \node[draw, circle, inner sep=0pt, minimum size=3pt, fill=black] at (1.5, 0) {};
            \node[draw, circle, inner sep=0pt, minimum size=3pt, fill=black] at (0.5, 0) {};
        \end{tikzpicture} \xrightarrow{\bigcup \{ \alpha_1, \alpha_2 \}}
        \begin{tikzpicture}[scale=1, baseline=-3]
            \draw[draw=none, fill=gray!40] (0,1) -- (2,1) -- (2,-1) -- (0,-1) -- (1, -0.15) to[out=0, in=-90] (1.15, 0) to[out=90, in=0] (1, 0.15) to[out=180, in=90] (0.85, 0) to[out=-90, in=180] (1, -0.15) -- (0, -1) -- (0,1);
            \draw[thick, blue, ->] (0,1) -- (1,1);
            \draw[thick, blue] (0,1) -- (2,1);
            \draw[thick, blue, ->] (0,-1) -- (1,-1);
            \draw[thick, blue] (0,-1) -- (2,-1);
            \draw[thick, olive, ->] (0,-1) -- (0,0);
            \draw[thick, olive] (0,-1) -- (0,1);
            \draw[thick, olive, ->] (2,-1) -- (2,0);
            \draw[thick, olive] (2,-1) -- (2,1);
            \draw[draw=none, fill=orange, opacity=0.7] (1,0) circle[radius=0.5cm];
            \draw[thick] (0, -1) -- (0.65, -0.35);
            \draw[thick, domain=180:225] plot ({0.5*cos(\x)+1}, {0.5*sin(\x)});
            \draw[thick, domain=0:45] plot ({0.5*cos(\x)+1}, {0.5*sin(\x)});
            \draw[thick] (1.35, 0.35) -- (2, 1);
            \draw[thick] (0.5, 0) to[out=0, in=200] (1.1, -0.3) to[out=20, in=270] (1.2, 0) to[out=90, in=0] (1, 0.15) to[out=0, in=180] (1.5, 0);
            \draw[thick, black, fill=white] (1, 0) circle (0.15);
            \node[draw, circle, inner sep=0pt, minimum size=3pt, fill=white] at (1, 0.15) {};
            \node[draw, circle, inner sep=0pt, minimum size=3pt, fill=black] at (1.35, 0.35) {};
            \node[draw, circle, inner sep=0pt, minimum size=3pt, fill=black] at (0.65, -0.35) {};
            \node[draw, circle, inner sep=0pt, minimum size=3pt, fill=black] at (1.5, 0) {};
            \node[draw, circle, inner sep=0pt, minimum size=3pt, fill=black] at (0.5, 0) {};
        \end{tikzpicture} \xrightarrow{[\alpha]}
        \begin{tikzpicture}[scale=1, baseline=-3]
            \draw[draw=none, fill=gray!40] (0,1) -- (2,1) -- (2,-1) -- (0,-1) -- (1, -0.15) to[out=0, in=-90] (1.15, 0) to[out=90, in=0] (1, 0.15) to[out=180, in=90] (0.85, 0) to[out=-90, in=180] (1, -0.15) -- (0, -1) -- (0,1);
            \draw[thick, blue, ->] (0,1) -- (1,1);
            \draw[thick, blue] (0,1) -- (2,1);
            \draw[thick, blue, ->] (0,-1) -- (1,-1);
            \draw[thick, blue] (0,-1) -- (2,-1);
            \draw[thick, olive, ->] (0,-1) -- (0,0);
            \draw[thick, olive] (0,-1) -- (0,1);
            \draw[thick, olive, ->] (2,-1) -- (2,0);
            \draw[thick, olive] (2,-1) -- (2,1);
            \draw[thick] (0, -1) to[out=20, in=180] (1, -0.6) to[out=0, in=270] (1.4, -0.2) to[out=90, in=0] (1, 0.15) -- (2, 1);
            \draw[thick, black, fill=white] (1, 0) circle (0.15);
            \node[draw, circle, inner sep=0pt, minimum size=3pt, fill=white] at (1, 0.15) {};
        \end{tikzpicture}$$
        \caption[Examples of $f_{A}(p,q,r)$]{Example of $f_{A}\left(0,1,\frac{1}{2}\right)$ and $f_{A}\left(1,1,\frac{1}{2}\right)$}
        \label{fig:fMapExample}
    \end{figure}
    
    Define $A$, $\partial_A$, $x_0$, $x_1$, $x_2$, and $c_1$ as before. Let $\eta$ be a smooth radial vector field on a neighborhood of $A$, denoted $N(A)$, such that
    \begin{itemize}
        \item For all $y \in N(A) \setminus A$, $\eta(y)$ is tangent to the geodesic within $N(A)$ from $y$ to $\partial_A$
        \item $\eta$ is zero outside of $N(A)$,
        \item and for all $y \in \partial_A$, the vector $\eta(y)$ is normal to $\partial_A$ and points inwards.
    \end{itemize}
    
    Define
    $$\Psi: \left(T^2 \setminus D^2 \right) \setminus A \xrightarrow{\quad\sim\quad} T^2 \setminus \{\mathfrak{p}\} \xhookrightarrow{\phantom{phntm}} T^2$$ to be the homeomorphism that flows along $\eta$, and therefore restricts to the identity on $T^2 \setminus N(A)$, followed by its canonical embedding into $T^2$.
    
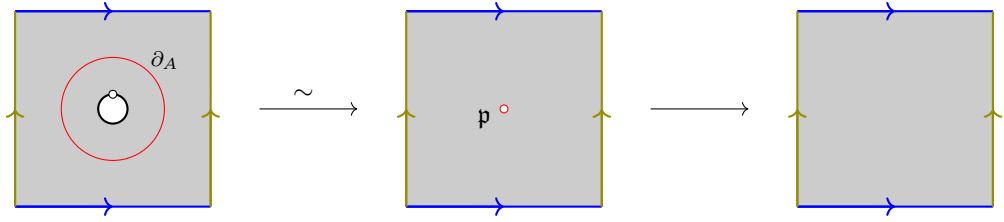
\begin{center}
        \begin{tikzpicture}[scale=1.3]
            \draw[draw=none, fill=gray!40] (-1,1) -- (1,1) -- (1,-1) -- (-1,-1);
            \draw[thick, blue, ->] (-1,1) -- (0,1);
            \draw[thick, blue] (-1,1) -- (1,1);
            \draw[thick, blue, ->] (-1,-1) -- (0,-1);
            \draw[thick, blue] (-1,-1) -- (1,-1);
            \draw[thick, olive, ->] (-1,-1) -- (-1,0);
            \draw[thick, olive] (-1,-1) -- (-1,1);
            \draw[thick, olive, ->] (1,-1) -- (1,0);
            \draw[thick, olive] (1,-1) -- (1,1);
            \draw[draw=red] (0, 0) circle[radius=15pt];
            \node at (0.53, 0.5) {\footnotesize{$\partial_A$}};
            \draw[thick, black, fill=white] (0, 0) circle (0.15);
            \node[draw, circle, inner sep=0pt, minimum size=3pt, fill=white] at (0, 0.15) {};
            \draw[->] (1.5, 0) -- (2.5, 0);
            \node at (1.95, 0.15) {$\sim$};
            \draw[draw=none, fill=gray!40] (3,1) -- (5,1) -- (5,-1) -- (3,-1) -- (3,1);
            \draw[thick, blue, ->] (3,1) -- (4,1);
            \draw[thick, blue] (3,1) -- (5,1);
            \draw[thick, blue, ->] (3,-1) -- (4,-1);
            \draw[thick, blue] (3,-1) -- (5,-1);
            \draw[thick, olive, ->] (3,-1) -- (3,0);
            \draw[thick, olive] (3,-1) -- (3,1);
            \draw[thick, olive, ->] (5,-1) -- (5,0);
            \draw[thick, olive] (5,-1) -- (5,1);
            \node[draw, circle, red, inner sep=0pt, minimum size=3pt, fill=white] at (4, 0) {};
            \node at (3.8, -0.1) {$\mathfrak{p}$};
            \draw[->] (5.5, 0) -- (6.5, 0);
            \draw[draw=none, fill=gray!40] (7,1) -- (9,1) -- (9,-1) -- (7,-1);
            \draw[thick, blue, ->] (7,1) -- (8,1);
            \draw[thick, blue] (7,1) -- (9,1);
            \draw[thick, blue, ->] (7,-1) -- (8,-1);
            \draw[thick, blue] (7,-1) -- (9,-1);
            \draw[thick, olive, ->] (7,-1) -- (7,0);
            \draw[thick, olive] (7,-1) -- (7,1);
            \draw[thick, olive, ->] (9,-1) -- (9,0);
            \draw[thick, olive] (9,-1) -- (9,1);
        \end{tikzpicture}
    \end{center}
    
    Consider any curve in $\Omega_x$ and select a simple representative, $\alpha$, such that $\left|\alpha \cap \partial_A \right| = 2$ and $\Psi\left(\alpha \setminus ( A \cap \alpha )\right)$ has constant slope. Since we required $\alpha$ to be simple, if $\alpha$ lies entirely within $A$, it must either be parallel to the boundary of $T^2 \setminus D^2$ or null-homotopic, and therefore not in $\Omega_x$. As all simple unoriented closed curves on the torus are classified as $(p,q)$-curves, where $(p,q) \sim (-p,-q)$ and $p$ and $q$ are coprime, we can uniquely associate a $(p,q)$ pair to the closed curve $\Psi\left(\alpha \setminus ( A \cap \alpha )\right) \cup \{\mathfrak{p}\}$.
    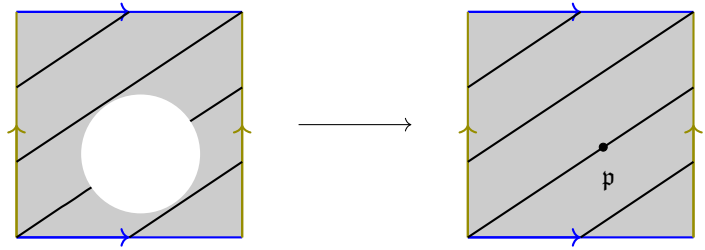
\begin{figure}[h]
    \begin{center}
        \begin{tikzpicture}[scale=1.5]
            \draw[draw=none, fill=gray!40] (3,1) -- (5,1) -- (5,-1) -- (3,-1) -- (4, -0.15) to[out=0, in=-90] (4.15, 0) to[out=90, in=0] (4, 0.15) to[out=180, in=90] (3.85, 0) to[out=-90, in=180] (4, -0.15) -- (3, -1) -- (3,1);
            \draw[thick, blue, ->] (3,1) -- (4,1);
            \draw[thick, blue] (3,1) -- (5,1);
            \draw[thick, blue, ->] (3,-1) -- (4,-1);
            \draw[thick, blue] (3,-1) -- (5,-1);
            \draw[thick, olive, ->] (3,-1) -- (3,0);
            \draw[thick, olive] (3,-1) -- (3,1);
            \draw[thick, olive, ->] (5,-1) -- (5,0);
            \draw[thick, olive] (5,-1) -- (5,1);
            \draw[thick] (3, -1) -- (3.4, -0.73) to[out=30, in=180] (4,-0.45) to[out=0, in=240] (4.3, -0.3) to[out=60, in=0] (4.1, 0.15);
            \draw[thick] (4, 0.15) to[out=90, in=180] (4.54, 0.03) -- (5, 0.33);
            \draw[thick] (3, 0.33) -- (4, 1);
            \draw[thick] (4, -1) -- (5, -0.33);
            \draw[thick] (3, -0.33) -- (5, 1);
            \draw[thick, black] (4, 0) circle (0.15);
            \node[draw, circle, inner sep=0pt, minimum size=3pt, fill=green] at (4, 0.15) {};
            \draw[draw=none, fill=white] (4.1,-0.26) circle[radius=15pt];
            \draw[->] (5.5, 0) -- (6.5, 0);
            \draw[draw=none, fill=gray!40] (7,1) -- (9,1) -- (9,-1) -- (7,-1);
            \draw[thick, blue, ->] (7,1) -- (8,1);
            \draw[thick, blue] (7,1) -- (9,1);
            \draw[thick, blue, ->] (7,-1) -- (8,-1);
            \draw[thick, blue] (7,-1) -- (9,-1);
            \draw[thick, olive, ->] (7,-1) -- (7,0);
            \draw[thick, olive] (7,-1) -- (7,1);
            \draw[thick, olive, ->] (9,-1) -- (9,0);
            \draw[thick, olive] (9,-1) -- (9,1);
            \draw[thick] (7, -1) -- (9, 0.33);
            \draw[thick] (7, 0.33) -- (8, 1);
            \draw[thick] (8, -1) -- (9, -0.33);
            \draw[thick] (7, -0.33) -- (9, 1);
            \node[draw, circle, inner sep=0pt, minimum size=3pt, fill=black] at (8.2, -0.2) {};
            \node at (8.25, -0.5) {$\mathfrak{p}$};
        \end{tikzpicture}\caption[Turning a tangle into a closed curve on the torus]{Turning the tangle, $\alpha$, in the complement of $A$ into a closed curve on the torus via $\Psi$. This example becomes a curve in $T^2$ with classification $(2,3)$}\label{fig:PsiClosure}
    \end{center}
    \end{figure}

    Identify $\alpha$ with a map, $\alpha : [0,1] \to T^2 \setminus D^2$, such that $\alpha(0) = \alpha(1) = x$. Label the points $\{y_1, y_2\} = \alpha \cap \partial_A$ and their preimages as $t_i^\prime := \alpha^{-1}(y_i)$, assuming $t_1^\prime < t_2^\prime$. If $(p,q)$ is a positive slope, then let $\gamma_i$ be the geodesics on $\partial_A$ from $y_i$ to $x_i$, rotating clockwise. If these geodesics intersect each other then we can swap the labels $y_i$ by changing the orientation of $\alpha$ to avoid this. If $(p,q)$ is a negative slope then let $\gamma_i$ be the geodesics on $\partial_A$ from $y_i$ to $x_i$, rotating counterclockwise, once again changing the orientation of $\alpha$ if needed.
    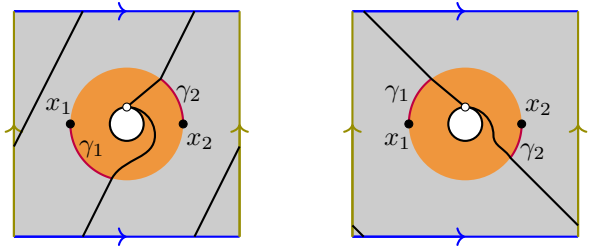
\begin{figure}
    \begin{center}
        \begin{tikzpicture}[scale=1.5]
            \draw[draw=none, fill=gray!40] (3,1) -- (5,1) -- (5,-1) -- (3,-1) -- (4, -0.15) to[out=0, in=-90] (4.15, 0) to[out=90, in=0] (4, 0.15) to[out=180, in=90] (3.85, 0) to[out=-90, in=180] (4, -0.15) -- (3, -1) -- (3,1);
            \draw[thick, blue, ->] (3,1) -- (4,1);
            \draw[thick, blue] (3,1) -- (5,1);
            \draw[thick, blue, ->] (3,-1) -- (4,-1);
            \draw[thick, blue] (3,-1) -- (5,-1);
            \draw[thick, olive, ->] (3,-1) -- (3,0);
            \draw[thick, olive] (3,-1) -- (3,1);
            \draw[thick, olive, ->] (5,-1) -- (5,0);
            \draw[thick, olive] (5,-1) -- (5,1);
            \draw[draw=none, fill=orange, opacity=0.7] (4,0) circle[radius=0.5cm];
            \draw[thick, purple, domain=180:255] plot ({0.5*cos(\x)+4}, {0.5*sin(\x)});
            \draw[thick, purple, domain=0:53] plot ({0.5*cos(\x)+4}, {0.5*sin(\x)});
            \draw[thick] (3.61, -1) -- (3.87, -0.48) to[out=60, in=270] (4.25, -0.1) to[out=90, in=0] (4, 0.15) -- (4.3, 0.4) -- (4.6, 1);
            \draw[thick] (4.6, -1) -- (5, -0.2);
            \draw[thick] (3, -0.2) -- (3.61, 1);
            \draw[thick, black, fill=white] (4, 0) circle (0.15);
            \node[draw, circle, inner sep=0pt, minimum size=3pt, fill=white] at (4, 0.15) {};
            \node[draw, circle, inner sep=0pt, minimum size=3pt, fill=black] at (4.5, 0) {};
            \node[draw, circle, inner sep=0pt, minimum size=3pt, fill=black] at (3.5, 0) {};
            \node at (3.4, 0.16) {$x_1$};
            \node at (4.65, -0.15) {$x_2$};
            \node at (3.69, -0.22) {$\gamma_1$};
            \node at (4.55, 0.27) {$\gamma_2$};
            \draw[draw=none, fill=gray!40] (6,1) -- (8,1) -- (8,-1) -- (6,-1) -- (7, -0.15) to[out=0, in=-90] (7.15, 0) to[out=90, in=0] (7, 0.15) to[out=180, in=90] (6.85, 0) to[out=-90, in=180] (7, -0.15) -- (6, -1) -- (6, 1);
            \draw[thick, blue, ->] (6,1) -- (7,1);
            \draw[thick, blue] (6,1) -- (8,1);
            \draw[thick, blue, ->] (6,-1) -- (7,-1);
            \draw[thick, blue] (6,-1) -- (8,-1);
            \draw[thick, olive, ->] (6,-1) -- (6,0);
            \draw[thick, olive] (6,-1) -- (6,1);
            \draw[thick, olive, ->] (8,-1) -- (8,0);
            \draw[thick, olive] (8,-1) -- (8,1);
            \draw[draw=none, fill=orange, opacity=0.7] (7,0) circle[radius=0.5cm];
            \draw[thick, purple, domain=127:180] plot ({0.5*cos(\x)+7}, {0.5*sin(\x)});
            \draw[thick, purple, domain=-37:0] plot ({0.5*cos(\x)+7}, {0.5*sin(\x)});
            \draw[thick] (6.1, 1) -- (6.7, 0.4) -- (7, 0.15) to[out=0, in=90] (7.25, -0.1) to[out=270, in=120] (7.4, -0.3) -- (8, -0.9);
            \draw[thick] (6, -0.9) -- (6.1, -1);
            \draw[thick, black, fill=white] (7, 0) circle (0.15);
            \node[draw, circle, inner sep=0pt, minimum size=3pt, fill=white] at (7, 0.15) {};
            \node[draw, circle, inner sep=0pt, minimum size=3pt, fill=black] at (7.5, 0) {};
            \node[draw, circle, inner sep=0pt, minimum size=3pt, fill=black] at (6.5, 0) {};
            \node at (6.4, -0.15) {$x_1$};
            \node at (7.65, 0.15) {$x_2$};
            \node at (6.4, 0.27) {$\gamma_1$};
            \node at (7.6, -0.25) {$\gamma_2$};
        \end{tikzpicture}
        \caption[Constructing geodesic paths, $\gamma_i$]{Left: Example of a curve with a positive slope. Right: Example with a negative slope.}
    \end{center}
    \end{figure}
    If $(p,q) = (1,0)$, the same slope as the longitude, then let $\gamma_i$ be the geodesics on $\partial_A$ from $y_i$ to $x_i$, rotating clockwise. Finally, if $(p,q) = (0,1)$ then $y_i = x_i$ and $\gamma_i$ is trivial for each $i$.
    
    Isotope $\alpha$ inside $A$ so that for some $0 < t_1 < t_1^\prime$ and $1 > t_2 > t_2^\prime$, $\alpha$ satisfies:
    \begin{itemize}
        \item $\alpha(t_i) = x_i$,
        \item $\alpha([t_1, t_1^\prime]) = \gamma_1$,
        \item $\alpha([t_2^\prime, t_2]) = \gamma_2$,
        \item $\alpha([t_1^\prime, t_2^\prime]\bigsqcup \{0,1\})$ remains unchanged,
        \item $\alpha$ is transverse with itself at $\alpha(0)$ and $\alpha(1)$.
    \end{itemize}
    
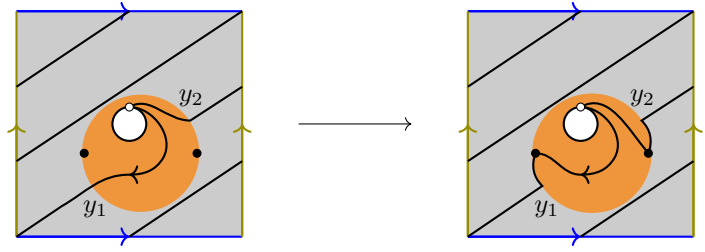
\begin{center}
        \begin{tikzpicture}[scale=1.5]
            \draw[draw=none, fill=gray!40] (4,1) -- (6,1) -- (6,-1) -- (4,-1) -- (5, -0.15) to[out=0, in=-90] (5.15, 0) to[out=90, in=0] (5, 0.15) to[out=180, in=90] (4.85, 0) to[out=-90, in=180] (5, -0.15) -- (4, -1) -- (4,1);
            \draw[thick, blue, ->] (4,1) -- (5,1);
            \draw[thick, blue] (4,1) -- (6,1);
            \draw[thick, blue, ->] (4,-1) -- (5,-1);
            \draw[thick, blue] (4,-1) -- (6,-1);
            \draw[thick, olive, ->] (4,-1) -- (4,0);
            \draw[thick, olive] (4,-1) -- (4,1);
            \draw[thick, olive, ->] (6,-1) -- (6,0);
            \draw[thick, olive] (6,-1) -- (6,1);
            \draw[draw=none, fill=orange, opacity=0.7] (5.1,-0.26) circle[radius=0.52cm];
            \draw[thick] (4, -1) -- (4.2, -0.865) -- (4.4, -0.73) to[out=30, in=180] (5,-0.45) to[out=0, in=240] (5.3, -0.3) to[out=60, in=0] (5, 0.15);
            \draw[thick] (5, 0.15) to[out=30, in=180] (5.54, 0.03) -- (6, 0.33);
            \draw[thick] (4, 0.33) -- (5, 1);
            \draw[thick] (5, -1) -- (6, -0.33);
            \draw[thick] (4, -0.33) -- (6, 1);
            \path[thick, tips, ->] (5.45,0) arc (0:-90:0.45);
            \draw[thick, black, fill=white] (5, 0) circle (0.15);
            \node[draw, circle, inner sep=0pt, minimum size=3pt, fill=white] at (5, 0.15) {};
            \node[draw, circle, inner sep=0pt, minimum size=3pt, fill=black] at (4.6, -0.26) {};
            \node[draw, circle, inner sep=0pt, minimum size=3pt, fill=black] at (5.6, -0.26) {};
            \node at (4.7, -0.75) {$y_1$};
            \node at (5.55, 0.23) {$y_2$};
            \draw[->] (6.5, 0) -- (7.5, 0);
            \draw[draw=none, fill=gray!40] (8,1) -- (10,1) -- (10,-1) -- (8,-1) -- (9, -0.15) to[out=0, in=-90] (9.15, 0) to[out=90, in=0] (9, 0.15) to[out=180, in=90] (8.85, 0) to[out=-90, in=180] (9, -0.15) -- (8, -1) -- (8,1);
            \draw[thick, blue, ->] (8,1) -- (9,1);
            \draw[thick, blue] (8,1) -- (10,1);
            \draw[thick, blue, ->] (8,-1) -- (9,-1);
            \draw[thick, blue] (8,-1) -- (10,-1);
            \draw[thick, olive, ->] (8,-1) -- (8,0);
            \draw[thick, olive] (8,-1) -- (8,1);
            \draw[thick, olive, ->] (10,-1) -- (10,0);
            \draw[thick, olive] (10,-1) -- (10,1);
            \draw[draw=none, fill=orange, opacity=0.7] (9.1,-0.26) circle[radius=0.53cm];
            \draw[thick] (8, -1) -- (8.66, -0.55) to[out=135, in=250] (8.6, -0.26) to[out=0, in=180] (9,-0.45) to[out=0, in=240] (9.3, -0.3) to[out=60, in=0] (9, 0.15);
            \draw[thick] (9, 0.15) to[out=30, in=130] (9.6, -0.26) to[out=70, in=-45] (9.54, 0.03) -- (10, 0.33);
            \draw[thick] (8, 0.33) -- (9, 1);
            \draw[thick] (9, -1) -- (10, -0.33);
            \draw[thick] (8, -0.33) -- (10, 1);
            \path[thick, tips, ->] (9.45,0) arc (0:-90:0.45);
            \draw[thick, black, fill=white] (9, 0) circle (0.15);
            \node[draw, circle, inner sep=0pt, minimum size=3pt, fill=white] at (9, 0.15) {};
            \node[draw, circle, inner sep=0pt, minimum size=3pt, fill=black] at (8.6, -0.26) {};
            \node[draw, circle, inner sep=0pt, minimum size=3pt, fill=black] at (9.6, -0.26) {};
            \node at (8.7, -0.75) {$y_1$};
            \node at (9.55, 0.23) {$y_2$};
        \end{tikzpicture}
    \end{center}
    We can now focus on the region within the annulus, $A$, and just $\alpha([0,t_1] \sqcup [t_2, 1])$.
    
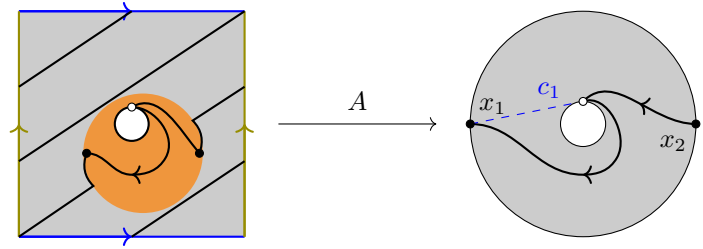
\begin{center}
        \begin{tikzpicture}[scale=1.5]
            \draw[draw=none, fill=gray!40] (0,1) -- (2,1) -- (2,-1) -- (0,-1) -- (1, -0.15) to[out=0, in=250] (1.15, 0) to[out=90, in=0] (1, 0.15) to[out=180, in=90] (0.85, 0) to[out=-90, in=180] (1, -0.15) -- (0, -1) -- (0,1);
            \draw[thick, blue, ->] (0,1) -- (1,1);
            \draw[thick, blue] (0,1) -- (2,1);
            \draw[thick, blue, ->] (0,-1) -- (1,-1);
            \draw[thick, blue] (0,-1) -- (2,-1);
            \draw[thick, olive, ->] (0,-1) -- (0,0);
            \draw[thick, olive] (0,-1) -- (0,1);
            \draw[thick, olive, ->] (2,-1) -- (2,0);
            \draw[thick, olive] (2,-1) -- (2,1);
            \draw[draw=none, fill=orange, opacity=0.7] (1.1,-0.26) circle[radius=0.53cm];
            \draw[thick] (0, -1) -- (0.66, -0.55) to[out=135, in=250] (0.6, -0.26) to[out=0, in=180] (1,-0.45) to[out=0, in=240] (1.3, -0.3) to[out=60, in=0] (1, 0.15);
            \draw[thick] (1, 0.15) to[out=30, in=130] (1.6, -0.26) to[out=70, in=-45] (1.54, 0.03) -- (2, 0.33);
            \draw[thick] (0, 0.33) -- (1, 1);
            \draw[thick] (1, -1) -- (2, -0.33);
            \draw[thick] (0, -0.33) -- (2, 1);
            \path[thick, tips, ->] (1.45,0) arc (0:-90:0.45);
            \draw[thick, black, fill=white] (1, 0) circle (0.15);
            \node[draw, circle, inner sep=0pt, minimum size=3pt, fill=white] at (1, 0.15) {};
            \node[draw, circle, inner sep=0pt, minimum size=3pt, fill=black] at (0.6, -0.26) {};
            \node[draw, circle, inner sep=0pt, minimum size=3pt, fill=black] at (1.6, -0.26) {};
            \draw[fill=gray!40] (5, 0) circle (1);
            \draw[fill=white] (5, 0) circle (0.2);
            \draw[thick] (4, 0) to[out=0, in=180] (5, -0.45) to[out=0, in=240] (5.3, -0.3) to[out=60, in=20] (5, 0.2);
            \draw[thick] (5, 0.2) to[out=45, in=180] (6, 0);
            \path[thick, tips, ->] (5.45,0) arc (0:-90:0.45);
            \path[thick, tips, ->] (6, 0) -- (5.5, 0.18);
            \draw[dashed, blue] (4, 0) -- (5, 0.2);
            \node[draw, circle, inner sep=0pt, minimum size=3pt, fill=black] at (4, 0) {};
            \node[draw, circle, inner sep=0pt, minimum size=3pt, fill=black] at (6, 0) {};
            \node[draw, circle, inner sep=0pt, minimum size=3pt, fill=white] at (5, 0.2) {};
            \node at (4.2, 0.16) {$x_1$};
            \node at (5.8, -0.17) {$x_2$};
            \node[blue] at (4.7, 0.3) {$c_1$};
            \draw[->] (2.3, 0) -- (3.7, 0);
            \node at (3, 0.2) {$A$};
        \end{tikzpicture}
    \end{center}

    Consider the corresponding restriction, $\hat{\alpha}: [0,t_1] \cup [t_2, 1] \to A$, of $\alpha$ and further restrict this into its two components, $\hat{\alpha}_1: [0,t_1] \to A$, and $\hat{\alpha}_2: [t_2,1] \to A$. Since the mapping class group of the annulus is isomorphic to $\mathbb{Z}$, $\hat{\alpha}_1$ must be isotopic to $\sigma^{r_{1,\alpha}} c_1$ for some $r_{1,\alpha} \in \mathbb{Z}$.
    
    If we cut $A$ along $\hat{\alpha}_1$, the resulting picture is a square. Since $\hat{\alpha}_2$ is a path from $\hat{\alpha}(t_2)$ to $x$ and since $\hat{\alpha}(t_2)$ lies on $\partial_A$ and is distinct from $\hat{\alpha}(t_1)$, there are exactly two possible ways $\hat{\alpha}$ can be simple up to isotopy.
    
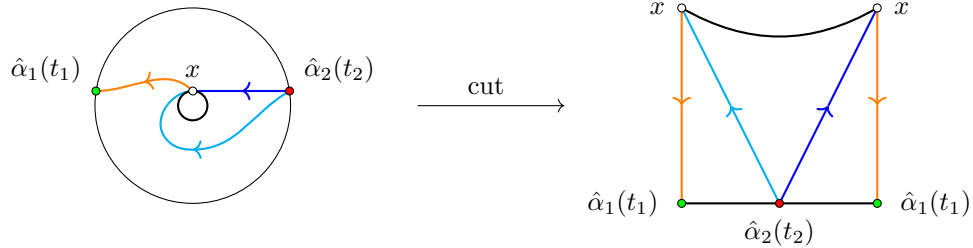
\begin{center}
        \begin{tikzpicture}[scale=1.3]
            \draw (3,0) circle (1);
            \draw[thick, orange] (2.01, 0.15) to[out=0, in=135] (3, 0.15);
            \draw[thick, cyan] (3, 0.15) to[out=180, in=120] (2.7, -0.3) to[out=300, in=180] (3,-0.45) to[out=0, in=210] (3.99, 0.15);
            \draw[thick, blue] (3, 0.15) -- (3.99, 0.15);
            \draw[thick, black] (3, 0) circle (0.15);
            \node[draw, circle, inner sep=0pt, minimum size=3pt, fill=white] at (3, 0.15) {};
            \node[draw, circle, inner sep=0pt, minimum size=3pt, fill=green] at (2.01, 0.15) {};
            \node[draw, circle, inner sep=0pt, minimum size=3pt, fill=red] at (3.99, 0.15) {};
            \path[thick, tips, orange, ->] (3, 0.35) -- (2.5, 0.24);
            \path[thick, tips, blue, ->] (4, 0.15) -- (3.5, 0.15);
            \path[thick, tips, cyan, ->] (3.45,0) arc (0:-90:0.45);
            \node at (1.5, 0.37) {$\hat{\alpha}_1(t_1)$};
            \node at (3, 0.35) {$x$};
            \node at (4.5, 0.37) {$\hat{\alpha}_2(t_2)$};
            \draw[->] (5.3, 0) -- (6.8, 0);
            \node at (6, 0.2) {cut};
            \draw[thick] (8, 1) to[out=330, in=210] (10,1);
            \draw[thick] (8,-1) -- (10,-1);
            \draw[thick, orange, ->] (8,1) -- (8,0);
            \draw[thick, orange] (8,0) -- (8,-1);
            \draw[thick, orange, ->] (10,1) -- (10,0);
            \draw[thick, orange] (10,0) -- (10,-1);
            \draw[thick, cyan, ->] (9, -1) -- (8.5, 0);
            \draw[thick, cyan] (8.5, 0) -- (8, 1);
            \draw[thick, blue, ->] (9, -1) -- (9.5, 0);
            \draw[thick, blue] (9.5, 0) -- (10, 1);
            \node[draw, circle, inner sep=0pt, minimum size=3pt, fill=white] at (8, 1) {};
            \node[draw, circle, inner sep=0pt, minimum size=3pt, fill=white] at (10, 1) {};
            \node[draw, circle, inner sep=0pt, minimum size=3pt, fill=green] at (8, -1) {};
            \node[draw, circle, inner sep=0pt, minimum size=3pt, fill=green] at (10, -1) {};
            \node[draw, circle, inner sep=0pt, minimum size=3pt, fill=red] at (9, -1) {};
            \node at (7.4, -1) {$\hat{\alpha}_1(t_1)$};
            \node at (10.6, -1) {$\hat{\alpha}_1(t_1)$};
            \node at (9, -1.3) {$\hat{\alpha}_2(t_2)$};
            \node at (7.75, 1) {$x$};
            \node at (10.25, 1) {$x$};
        \end{tikzpicture}
    \end{center}
    
    Let $v_1 = \hat{\alpha}_{1}'(0)$ and $v_2 = \hat{\alpha}_{2}'(1)$, the corresponding tangent vectors, using the induced orientation from $\alpha$'s domain. Since $\alpha$ is transverse with itself at $\alpha(0)$ and $\alpha(1)$, we must have that $\operatorname{det}(v_1 v_2) \neq 0$.
    Finally, define 
    \begin{align}\label{eq:r2}
        r_{2,\alpha} = \begin{cases}
            r_{1,\alpha} & \text{if } \operatorname{det}(v_1 v_2) > 0\\
            r_{1,\alpha} - 1 & \text{if } \operatorname{det}(v_1 v_2) < 0\\
        \end{cases}
    \end{align}
    and let $r = \frac{r_{1,\alpha} + r_{2,\alpha}}{2}$ which is clearly an element of $\frac{1}{2}\mathbb{Z}$, providing us with the last entry in the tuple.\footnote{If $x_1$ were located somewhere else on $\partial_A$, it might seem more natural to define $r_{2, \alpha}$ as $r_{1,\alpha} + 1$ when the determinant is positive and as $r_{1,\alpha}$ when it's negative. However, using equation (\ref{eq:r2}) instead of this for the definition of $r_{2, \alpha}$, merely corresponds to a $\frac{1}{2}$ shift in this $\frac{1}{2}\mathbb{Z}$-torsor and does not violate any assumptions. Furthermore, if we had chosen $\sigma$ to be the counterclockwise Dehn twist instead, we would, among other things, want to define $r_{2, \alpha}$ as $r_{1,\alpha}$ when the determinant is positive and as $r_{1,\alpha}+1$ when it's negative.} Notice that this representative of $\alpha$ is clearly equivalent to $\widetilde{f}_{A}(p,q,r)$ outside of $A$ and isotopic inside of $A$ as the decomposition of $r$ into $r_{1,\alpha}$ and $r_{2,\alpha}$ is unique. Therefore, $\left(\pi_{\text{iso}} \circ \widetilde{f}_{A}\right)(p,q,r) = [\alpha]$ and so this map is surjective.

    Finally, suppose $f_{A}(p,q,r) = f_{A}(p^\prime,q^\prime,r^\prime) = [\alpha]$. First notice that $\alpha$'s geometric intersection number with a fixed meridian and fixed longitude away from $x$ are invariant under isotopy. As these intersection numbers correspond to $p$ and $q$ respectively, $p = p^\prime$ and $q = q^\prime$. Notice that $\widetilde{f}_{A}(p,q,r) \cap\partial_A$ must completely agree with $\widetilde{f}_{A}(p^\prime,q^\prime,r^\prime) \cap\partial_A$ as well as they are equivalent on $T^2 \setminus A$ and therefore $[\widetilde{f}_{A}(p,q,r)] = [\widetilde{f}_{A}(p^\prime,q^\prime,r^\prime)]$ on $A$. Because they agree up to isotopy and since $r = \frac{r_{1,\alpha} + r_{2,\alpha}}{2}$ and $r^\prime = \frac{r_{1,\alpha}^\prime + r_{2,\alpha}^\prime}{2}$ have unique decompositions, $r = r^\prime$. Thus $(p,q,r) = (p^\prime, q^\prime, r^\prime)$ and so $f_{A}$ is bijective.
\end{proof}

\begin{lemma}\label{lemma:IntersectionDehnTwist}
    If $\alpha \in \mathscr{S}(\Sigma)$ is a tangle and $\gamma \in \mathscr{S}(\Sigma)$ is a closed curve such that the geometric intersection number of $\alpha$ and $\gamma$ is $1$, then $\frac{1}{q^2 - q^{-2}}[\alpha, \gamma]_q$ resolves to a Dehn twist of $\alpha$ along $\gamma$ and $\frac{-1}{q^2 - q^{-2}}[\alpha, \gamma]_{q^{-1}} = \frac{1}{q^2 - q^{-2}}[\gamma, \alpha]_q$ is the corresponding Dehn twist in the opposite direction.
\end{lemma}
\begin{proof}
    Suppose $\alpha$ and $\gamma$ intersect exactly once.
    Then locally, we have
    \begin{align*}
        [\alpha, \gamma]_q &= q\alpha \gamma - q^{-1}\gamma \alpha\\
        &= q
        \begin{tikzpicture}[baseline=0]
            \draw[dashed, fill=gray!40] (0,0) circle (1);
            \draw[thick] (-0.71, 0.71) -- (0.71, -0.71);
            \draw[line width=3mm, gray!40] (-0.5, -0.5) -- (0.5, 0.5);
            \draw[thick] (-0.71, -0.71) -- (0.71, 0.71);
            \node at (-0.3,-0.6) {$\alpha$};
            \node at (-0.3,0.6) {$\gamma$};
            \node (space) at (-1, 0) {};
        \end{tikzpicture} - q^{-1}
        \begin{tikzpicture}[baseline=0]
            \draw[dashed, fill=gray!40] (0,0) circle (1);
            \draw[thick] (-0.71, -0.71) -- (0.71, 0.71);
            \draw[line width=3mm, gray!40] (-0.5, 0.5) -- (0.5, -0.5);
            \draw[thick] (-0.71, 0.71) -- (0.71, -0.71);
            \node at (-0.3,-0.6) {$\alpha$};
            \node at (-0.3,0.6) {$\gamma$};
            \node (space) at (-1, 0) {};
        \end{tikzpicture} \\
        &= q^2
        \begin{tikzpicture}[baseline=0]
            \draw[dashed, fill=gray!40] (0,0) circle (1);
            \draw[thick] (-0.71, -0.71) to[out=45, in=-45] (-0.71, 0.71);
            \draw[thick] (0.71, -0.71) to[out=135, in=-135] (0.71, 0.71);
            \node (space) at (-1, 0) {};
        \end{tikzpicture} +
        \begin{tikzpicture}[baseline=0]
            \draw[dashed, fill=gray!40] (0,0) circle (1);
            \draw[thick] (-0.71, -0.71) to[out=45, in=135] (0.71, -0.71);
            \draw[thick] (-0.71, 0.71) to[out=-45, in=-135] (0.71, 0.71);
            \node (space) at (-1, 0) {};
        \end{tikzpicture} -
        \begin{tikzpicture}[baseline=0]
            \draw[dashed, fill=gray!40] (0,0) circle (1);
            \draw[thick] (-0.71, -0.71) to[out=45, in=135] (0.71, -0.71);
            \draw[thick] (-0.71, 0.71) to[out=-45, in=-135] (0.71, 0.71);
            \node (space) at (-1, 0) {};
        \end{tikzpicture} - q^{-2}
        \begin{tikzpicture}[baseline=0]
            \draw[dashed, fill=gray!40] (0,0) circle (1);
            \draw[thick] (-0.71, -0.71) to[out=45, in=-45] (-0.71, 0.71);
            \draw[thick] (0.71, -0.71) to[out=135, in=-135] (0.71, 0.71);
            \node (space) at (-1, 0) {};
        \end{tikzpicture}\\
        &= \left( q^2 - q^{-2} \right)
        \begin{tikzpicture}[baseline=0]
            \draw[dashed, fill=gray!40] (0,0) circle (1);
            \draw[thick] (-0.71, -0.71) to[out=45, in=-45] (-0.71, 0.71);
            \draw[thick] (0.71, -0.71) to[out=135, in=-135] (0.71, 0.71);
            \node (space) at (-1, 0) {};
        \end{tikzpicture}\\
    \end{align*}
    giving us our result.
    Noticing that $[b,a]_q = -[a,b]_{q^{-1}}$, we get $\frac{1}{q^2 - q^{-2}}[\gamma, \alpha]_{q} = \frac{-1}{q^2 - q^{-2}}[\alpha,\gamma]_{q^{-1}}$ is the Dehn twist in the opposite direction.
\end{proof}

\begin{lemma}\label{lemma:rationalcurves}
    If the ground ring $R$ contains $\mathbb{Q}[q^{\pm 1},\frac{1}{q^4-1}]$, then the algebra generated by the set $\{ X_{1,0}(\mu_{1}, \nu_{1}), \allowbreak X_{2,0}(\mu_{2}, \nu_{2}), \allowbreak X_{3,0}(\mu_{3}, \nu_{3}) \}$ for all $\mu_i, \nu_j \in \{ \pm \}$ contains any $\left\{ \mathbb{Q} \cup \frac{1}{0} \right\}$-sloped tangle in $\mathscr{S}\left( T^2 \setminus D^2 \right)$ with $0 \in \frac{1}{2}\mathbb{Z}$ twists.
\end{lemma}
\begin{proof}
    Recall that $\left| \operatorname{det} \begin{pmatrix} a & c\\ b & d \end{pmatrix} \right| = n$ if and only if the $(a,b)$-curve and the $(c,d)$-curve (or $(c,d)$-tangle) have a geometric intersection number of $n$.
    Consider the map of sets, $\sigma : GL_2(\mathbb{Z}) \to \mathbb{Z}^2$ defined by $A \mapsto A\cdot \left[ \begin{matrix} 1\\1\end{matrix} \right]$.
    Suppose $p$ and $q$ are coprime.
    If they are not coprime then the $(p,q,0)$-tangle intersects itself and can be resolved into curves and $\partial (T^2 \setminus D^2)$-tangles with $\mathbb{Q} \bigcup \frac{1}{0}$ slopes.
    We'll also assume that $0 < q < p$ (the proof is symmetric for negative slopes and when $q > p$).
    
    By Lemma 1 in \cite{FG00}, we can decompose $p$ and $q$ into $u + w = p$ and $v + z = q$ such that $\operatorname{det} \begin{pmatrix} u & w\\ v & z \end{pmatrix} = \pm 1$ with $0 < w < p$, $0 < u < p-1$, and $v, z > 0$, for $p \geq 3$.
    Thus for $\begin{pmatrix} p\\ q \end{pmatrix} \in \mathbb{Z}^2$, there exists an inverse, $\sigma^{-1} \begin{pmatrix} p\\ q \end{pmatrix} = \begin{pmatrix} u & w\\ v & z \end{pmatrix}$.
    We then find the inverse of the second column, $\sigma^{-1} \begin{pmatrix} w\\ z \end{pmatrix}$, and repeat until we get $\begin{pmatrix} p'\\ q' \end{pmatrix}$ for $q' < p' \leq 2$.
    Using Lemma~\ref{lemma:IntersectionDehnTwist}, each step in the reverse process of this algorithm corresponds to a Dehn twist, $\frac{\pm 1}{q^2 - q^{-2}}\left[ Y_j, - \right]_{q^{\pm 1}}$, along some $Y_j$-curve corresponding to the first column.
    
    Finally, note that the $(2,1)$-tangle is equal to $\frac{1}{q^2 - q^{-2}} \left[ Y_1, X_{3,0}(\mu, \nu) \right]_q$.
\end{proof}

\begin{example}
    To illustrate the application of this algorithm, let's examine the (5,3)-tangle with states $\mu$ and $\nu$, which we'll denote as $\tilde{X}(\mu, \nu)$.
    The matrices of interest in this example are
    $$\sigma\begin{pmatrix} 0 & 1\\ 1 & 0 \end{pmatrix} = \begin{pmatrix} 1\\ 1 \end{pmatrix}, \quad\quad
    \sigma\begin{pmatrix} 1 & 1\\ 0 & 1 \end{pmatrix} = \begin{pmatrix} 2\\ 1 \end{pmatrix}, \quad\quad
    \sigma\begin{pmatrix} 2 & 1\\ 1 & 1 \end{pmatrix} = \begin{pmatrix} 3\\ 2 \end{pmatrix}, \quad\quad
    \sigma\begin{pmatrix} 3 & 2\\ 2 & 1 \end{pmatrix} = \begin{pmatrix} 5\\ 3 \end{pmatrix}.$$
    Thus, we get the following series of Dehn twists
    $$ \frac{-1}{\left( q^2 - q^{-2} \right)^{6}}\left[ \left[ Y_3, \left[ Y_1, \left[ Y_2, Y_1 \right]_{q^{-1}} \right]_q \right]_{q^{-1}}, \left[ Y_1, \left[ Y_2, X_{1,0}(\mu, \nu) \right]_{q^{-1}} \right]_q \right]_{q} = \tilde{X}(\mu, \nu).$$
\end{example}

\begin{theorem}\label{theorem:Generators}
    If the ground ring $R$ contains $\mathbb{Q}[q^{\pm 1},\frac{1}{q^4-1}]$, then the set $B$ generates $\mathscr{S}\left( T^2 \setminus D^2 \right)$ as an algebra.
\end{theorem}
\begin{proof}
    A short calculation shows that
    \begin{align}
        X_{i,k+1}(\mu, \nu) &= \frac{1}{(q^2 - q^{-2})^3} \left[Y_{i+1}, \left[ Y_i, \left[ Y_{i-1}, X_{i,k}(\mu, \nu) \right]_q \right]_q \right]_q \label{eq:shiftup}\\
        X_{i,k-1}(\mu, \nu) &= \frac{1}{(q^2 - q^{-2})^3} \left[ \left[ \left[ X_{i,k}(\mu, \nu), Y_{i+1} \right]_q, Y_{i} \right]_q, Y_{i-1} \right]_q \label{eq:shiftdown}
    \end{align}
    for $i \mod 3$ and $\mu, \nu \in \{ \pm \}$.
    Using this result and Lemma~\ref{lemma:rationalcurves}, we can construct any $(p,q,r)$-tangle.
    By \cite{Le18}, the set of all isotopy classes of increasingly stated, simple $\partial(T^2 \setminus D^2)$-tangle diagrams forms a basis for $\mathscr{S}\left( T^2 \setminus D^2 \right)$.
    Since every simple non-parallel tangle can be expressed as a $(p,q,r)$-tangle, every simple $\partial\left( T^2 \setminus D^2 \right)$-tangle diagram can be written as a linear combination of products of these $(p,q,r)$-tangles and powers of our single parallel tangle.
    Thus, the stated skein algebra generated by $B$ spans the entire space.\hfill
\end{proof}

\begin{remark}
    Since we have the relation
    $$X_{3,0}(\mu, \nu) = \frac{1}{q^2 - q^{-2}}\left[ X_{1,0}(\mu, \nu), q^{1/2}X_{2,0}(+,-) - q^{5/2}X_{2,0}(-,+) \right]_q,$$
    we don't technically need to include $X_{3,0}(\mu, \nu)$ to generate $\mathscr{S}(T^2 \setminus D^2)$. We only need the eight elements $\left\{ X_{i,0}(\mu, \nu) \, \mid \mu, \nu \in \{ \pm \}, i \in \{1,2\} \right\}$. However, as with $K_q(T^2 \setminus D^2)$, it is often more notationally convenient to include $X_{3,0}(\mu, \nu)$ as well.
\end{remark}
\section{DAHA modules from quantum tori}\label{section:modsQTori}

\subsection{General surfaces}

Given an antisymmetric integral $n \times n$ matrix $Q$, the associated \textit{quantum torus} is defined as
$$\mathbb{T}^{n}(Q) := \frac{\mathbb{C} \left[ x_1^{\pm 1}, \cdots, x_n^{\pm 1} \right]}{\left( x_i x_j = q^{Q_{ij}} x_j x_i \right)}$$
and the corresponding \textit{quantum plane} is
$$\mathbb{T}_{+}^{n}(Q) := \frac{\mathbb{C} \left[ x_1, \cdots, x_n \right]}{\left( x_i x_j = q^{Q_{ij}} x_j x_i \right)}.$$

\begin{prop}[Proposition 2.2 in \cite{LY22}]\label{prop:embedIntoQT}
    Let $Q$ be an antisymmetric integral $r \times r$ matrix, and let $A$ be an $\mathcal{R}$-algebra containing the quantum plane $\mathbb{T}_+^r(Q)$ as a subalgebra. If $A$ is a domain and for every $a \in A$ there exists a $\mathbf{k} \in \mathbb{N}^r$ such that $x^\mathbf{k}a \in \mathbb{T}_+^r(Q) \subset A$, where $x^\mathbf{k} := x_1^{k_1} x_2^{k_2} \cdots x_r^{k_r}$, then $A$ is an Ore domain and the embedding $\mathbb{T}_+^r(Q) \hookrightarrow \mathbb{T}^r(Q)$ uniquely extends to an algebra embedding $\varphi_{\mathcal{E}} : A \hookrightarrow \mathbb{T}^r(Q)$.
\end{prop}

The map $\varphi_{\mathcal{E}}$ is associated to a quasitriangulation of $\Sigma$, denoted $\mathcal{E}$, and is discussed extensively in \cite{LY22}. Bonahon and Wong constructed a similar map in \cite{BW11} where $\mathbb{T}^r(Q)$ would instead be rational functions in skew-commuting variables associated to the square roots of the shear coordinates of Chekhov and Fock's enhanced quantum Teichm\"{u}ller space.
Furthermore, Bonahon and Wong demonstrated that a change in quasitriangulation induces an algebra isomorphism between the respective Teichm\"{u}ller spaces, which can also be understood as a change of shear coordinates.

We define an \textit{ideal multiarc} in $\Sigma$ as a finite collection of disjoint ideal arcs. A \textit{quasitriangulation} of $\Sigma$, denoted $\mathcal{E}$, is then defined as a maximal collection of nontrivial non-isotopic ideal arcs. Notice that $\mathcal{E}$ can be decomposed as $\mathcal{E} = \mathring{\mathcal{E}} \bigsqcup \mathcal{E}_\partial$, where $\mathring{\mathcal{E}}$ are the ideal arcs not parallel to any boundary components, and $\mathcal{E}_\partial$ are the ideal arcs that are parallel to boundary components. Let $\hat{\mathcal{E}}_\partial$ be a copy of $\mathcal{E}_\partial$, and define $\overline{\mathcal{E}} = \mathcal{E} \bigsqcup \hat{\mathcal{E}}_\partial = \left( \mathring{\mathcal{E}} \bigsqcup \mathcal{E}_\partial \right) \bigsqcup \hat{\mathcal{E}}_\partial$. Using $\overline{\mathcal{E}}$, L\^{e} and Yu demonstrated that stated skein algebras meet the conditions for Proposition~\ref{prop:embedIntoQT} and thus always lie between a quantum plane and a quantum torus (see Theorem~4.2 in \cite{LY22}).

In particular, we can correspond each $e \in \overline{\mathcal{E}}$ to a generator, $x_e$, of some quantum plane $\mathbb{T}_{+}^{r}(Q)$. The matrix $Q$ is defined using the anti-symmetric function
\begin{align*}
    Q(a,b) &= \# \locIdeal{b}{a} - \# \locIdeal{a}{b} & \text{for } a,b \in \mathcal{E}\\
    Q(a,\hat{b}) &= - \# \locIdeal{b}{a} - \# \locIdeal{a}{b} & \text{for } a \in \mathcal{E}, b \in \mathcal{E}_\partial \\
    Q(\hat{a}, \hat{b}) &= -Q(a,b) & \text{for } a,b \in \mathcal{E}_\partial.
\end{align*}
and canonically identifying the entries of $Q$ as $Q_{a,b} = Q(a,b)$.
By $\# \resizebox{0.8\width}{!}{\locIdeal{i}{j}}$ we mean the number of times a half edge of ideal arc $i$ and a half edge of ideal arc $j$ meet at the same ideal point with $i$ following $j$ in the clockwise order.

\begin{remark}
In geometry, ideal points are actually points that are removed from the boundary of a surface, making it non-compact. In our skein-theoretic constructions, the ideal points are filled in and become (vertical) marked intervals that our stated tangles end on. In the diagrams below, these marked intervals are indicated by white dots.
\end{remark}

Define $\psi_{\mathcal{E}} : \mathbb{T}_+^r(Q) \to \mathscr{S}(\Sigma)$ to be the map that sends a generator $x_e$ to a rescaling of the  diagram representing the edge $e \in \overline{\mathcal{E}}$, where the diagram has positive states if $e \in \mathcal{E}$ and has the following mixed states if $e \in \hat{\mathcal{E}}_\partial$.
$$\locIdeal{}{} \mapsto
\begin{cases}
    q^{-1/2} \text{ } \locTangle{+}{+} & \text{ if } e \in \mathcal{E}\\
    q^{1/2} \text{\phantom{aa}} \locTangle{+}{-} & \text{ if } e \in \hat{\mathcal{E}}_\partial.
\end{cases}$$
The rescaling constants $q^{\pm 1/2}$ in the above equation are introduced so that $x_e$ is reflection invariant, which follows by the $(R_6)$ relation.

\begin{lemma}[Lemma~4.9 in \cite{LY22}]
    For any $\alpha \in \mathscr{S}(\Sigma)$, there exists a monomial $m(x_1, \cdots, x_r) \in \mathbb{T}_{+}^{r}(Q)$ such that the product $\psi_{\mathcal{E}}\left(m(x_1, \cdots, x_r)\right) \alpha$ lies in the image of $\psi_{\mathcal{E}}$.
\end{lemma}

Since $\varphi_{\mathcal{E}}$ and $\psi_{\mathcal{E}}$ are algebra homomorphisms and $(\varphi_{\mathcal{E}} \circ \psi_{\mathcal{E}})(x_e) = x_e$, we can explicitly find where any element in our skein algebra is mapped to using the following trick
$$\varphi_{\mathcal{E}}(\alpha) = m^{-1}(x_1, \cdots, x_r) m(x_1, \cdots, x_r) \varphi_{\mathcal{E}}(\alpha) = m^{-1}(x_1, \cdots, x_r) \varphi_{\mathcal{E}}(\psi_{\mathcal{E}}(m(x_1, \cdots, x_r)) \alpha),$$
where $\mathscr{S}(\Sigma)$ is viewed as the canonical $T_{+}^{r}(Q)$-module induced by $\psi_{\mathcal{E}}$, i.e. $m \cdot \alpha := \psi_{\mathcal{E}}(m)\alpha$. This is a slight generalization of Lemma~6.9 in \cite{Mul16}.

\subsection{The punctured torus embedding}

We will now perform this calculation for when $\mathfrak{S} = T^2 \setminus D^2$, with a single marked point on the boundary. Let $\mathcal{E}$ be the quasitriangulation of $T^2 \setminus D^2$ shown in Figure~\ref{fig:quasitriangulation}.
\begin{figure}[h]
    \centering
    \resizebox{2\width}{!}{\begin{tikzpicture}
        \draw[draw=none, fill=gray!40] (0,1) -- (2,1) -- (2,-1) -- (0,-1) -- (1, -0.15) to[out=0, in=-90] (1.15, 0) to[out=90, in=0] (1, 0.15) to[out=180, in=90] (0.85, 0) to[out=-90, in=180] (1, -0.15) -- (0, -1) -- (0,1);
        \draw[thick, yellow] (1, 0) circle (0.15);
        \draw[thick, red] (1.3, -1) to[out=90, in=270] (1.3, 0) to[out=90, in=0] (1, 0.15) to[out=60, in=270] (1.3, 1);
        \draw[thick, blue] (0.4, -1) -- (0.4, -0.2) to[out=90, in=180] (1, 0.15) to[out=45, in=180] (2, 0.6);
        \draw[thick, blue] (0, 0.6) to[out=0, in=270] (0.4, 1);
        \draw[thick, green] (0.7, -1) to[out=90, in=270] (0.7, 0) to[out=90, in=180] (1, 0.15) to[out=120, in=270] (0.7, 1);
        \draw[thick, orange] (0, 0.15) -- (1, 0.15) -- (2, 0.15);
        \draw[thick, ->] (0,1) -- (0.9,1);
        \draw[thick, ->] (0.85,1) -- (1.2,1);
        \draw[thick] (1.2,1) -- (2,1);
        \draw[thick, ->] (0,-1) -- (0.9,-1);
        \draw[thick, ->] (0.85,-1) -- (1.2,-1);
        \draw[thick] (1.2,-1) -- (2,-1);
        \draw[thick, ->] (0,-1) -- (0,0);
        \draw[thick] (0,-1) -- (0,1);
        \draw[thick, ->] (2,-1) -- (2,0);
        \draw[thick] (2,-1) -- (2,1);
        \node[draw, circle, inner sep=0pt, minimum size=3pt, fill=white] at (1, 0.15) {};
    \end{tikzpicture}}
    \caption{A quasitriangulation of $T^2 \setminus D^2$}
    \label{fig:quasitriangulation}
\end{figure}
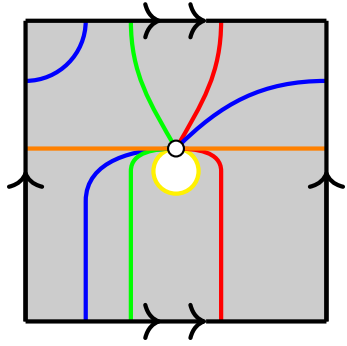
Note that as $T^2 \setminus D^2$ does not contain any punctures, $\mathcal{E}$ is a full triangulation of $T^2 \setminus D^2$. We will correspond our variables, $x_i$, to edges of $\mathcal{E}$ in the following way.
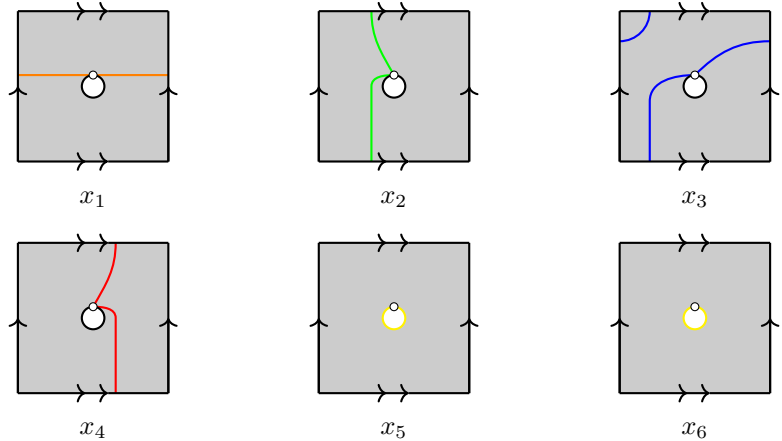
\begin{center}
    \begin{tikzpicture}
        \draw[draw=none, fill=gray!40] (0,1) -- (2,1) -- (2,-1) -- (0,-1) -- (1, -0.15) to[out=0, in=-90] (1.15, 0) to[out=90, in=0] (1, 0.15) to[out=180, in=90] (0.85, 0) to[out=-90, in=180] (1, -0.15) -- (0, -1) -- (0,1);
        \draw[thick, orange] (0, 0.15) -- (1, 0.15) -- (2, 0.15);
        \draw[thick, ->] (0,1) -- (0.9,1);
        \draw[thick, ->] (0.85,1) -- (1.2,1);
        \draw[thick] (1.2,1) -- (2,1);
        \draw[thick, ->] (0,-1) -- (0.9,-1);
        \draw[thick, ->] (0.85,-1) -- (1.2,-1);
        \draw[thick] (1.2,-1) -- (2,-1);
        \draw[thick, ->] (0,-1) -- (0,0);
        \draw[thick] (0,-1) -- (0,1);
        \draw[thick, ->] (2,-1) -- (2,0);
        \draw[thick] (2,-1) -- (2,1);
        \draw[thick, black] (1, 0) circle (0.15);
        \node[draw, circle, inner sep=0pt, minimum size=3pt, fill=white] at (1, 0.15) {};
        \node at (1,-1.5) {$x_1$};
        \draw[draw=none, fill=gray!40] (4,1) -- (6,1) -- (6,-1) -- (4,-1) -- (5, -0.15) to[out=0, in=-90] (5.15, 0) to[out=90, in=0] (5, 0.15) to[out=180, in=90] (4.85, 0) to[out=-90, in=180] (5, -0.15) -- (4, -1) -- (4,1);
        \draw[thick, green] (4.7, -1) to[out=90, in=270] (4.7, 0) to[out=90, in=180] (5, 0.15) to[out=120, in=270] (4.7, 1);
        \draw[thick, ->] (4,1) -- (4.9,1);
        \draw[thick, ->] (4.85,1) -- (5.2,1);
        \draw[thick] (5.2,1) -- (6,1);
        \draw[thick, ->] (4,-1) -- (4.9,-1);
        \draw[thick, ->] (4.85,-1) -- (5.2,-1);
        \draw[thick] (5.2,-1) -- (6,-1);
        \draw[thick, ->] (4,-1) -- (4,0);
        \draw[thick] (4,-1) -- (4,1);
        \draw[thick, ->] (6,-1) -- (6,0);
        \draw[thick] (6,-1) -- (6,1);
        \draw[thick, black] (5, 0) circle (0.15);
        \node[draw, circle, inner sep=0pt, minimum size=3pt, fill=white] at (5, 0.15) {};
        \node at (5,-1.5) {$x_2$};
        \draw[draw=none, fill=gray!40] (8,1) -- (10,1) -- (10,-1) -- (8,-1) -- (9, -0.15) to[out=0, in=-90] (9.15, 0) to[out=90, in=0] (9, 0.15) to[out=180, in=90] (8.85, 0) to[out=-90, in=180] (9, -0.15) -- (8, -1) -- (8,1);
        \draw[thick, blue] (8.4, -1) -- (8.4, -0.2) to[out=90, in=180] (9, 0.15) to[out=45, in=180] (10, 0.6);
        \draw[thick, blue] (8, 0.6) to[out=0, in=270] (8.4, 1);
        \draw[thick, ->] (8,1) -- (8.9,1);
        \draw[thick, ->] (8.85,1) -- (9.2,1);
        \draw[thick] (9.2,1) -- (10,1);
        \draw[thick, ->] (8,-1) -- (8.9,-1);
        \draw[thick, ->] (8.85,-1) -- (9.2,-1);
        \draw[thick] (9.2,-1) -- (10,-1);
        \draw[thick, ->] (8,-1) -- (8,0);
        \draw[thick] (8,-1) -- (8,1);
        \draw[thick, ->] (10,-1) -- (10,0);
        \draw[thick] (10,-1) -- (10,1);
        \draw[thick, black] (9, 0) circle (0.15);
        \node[draw, circle, inner sep=0pt, minimum size=3pt, fill=white] at (9, 0.15) {};
        \node at (9,-1.5) {$x_3$};
    \end{tikzpicture}
\end{center}
\begin{center}
    \begin{tikzpicture}
        \draw[draw=none, fill=gray!40] (0,1) -- (2,1) -- (2,-1) -- (0,-1) -- (1, -0.15) to[out=0, in=-90] (1.15, 0) to[out=90, in=0] (1, 0.15) to[out=180, in=90] (0.85, 0) to[out=-90, in=180] (1, -0.15) -- (0, -1) -- (0,1);
        \draw[thick, red] (1.3, -1) to[out=90, in=270] (1.3, 0) to[out=90, in=0] (1, 0.15) to[out=60, in=270] (1.3, 1);
        \draw[thick, ->] (0,1) -- (0.9,1);
        \draw[thick, ->] (0.85,1) -- (1.2,1);
        \draw[thick] (1.2,1) -- (2,1);
        \draw[thick, ->] (0,-1) -- (0.9,-1);
        \draw[thick, ->] (0.85,-1) -- (1.2,-1);
        \draw[thick] (1.2,-1) -- (2,-1);
        \draw[thick, ->] (0,-1) -- (0,0);
        \draw[thick] (0,-1) -- (0,1);
        \draw[thick, ->] (2,-1) -- (2,0);
        \draw[thick] (2,-1) -- (2,1);
        \draw[thick, black] (1, 0) circle (0.15);
        \node[draw, circle, inner sep=0pt, minimum size=3pt, fill=white] at (1, 0.15) {};
        \node at (1,-1.5) {$x_4$};
        \draw[draw=none, fill=gray!40] (4,1) -- (6,1) -- (6,-1) -- (4,-1) -- (5, -0.15) to[out=0, in=-90] (5.15, 0) to[out=90, in=0] (5, 0.15) to[out=180, in=90] (4.85, 0) to[out=-90, in=180] (5, -0.15) -- (4, -1) -- (4,1);
        \draw[thick, ->] (4,1) -- (4.9,1);
        \draw[thick, ->] (4.85,1) -- (5.2,1);
        \draw[thick] (5.2,1) -- (6,1);
        \draw[thick, ->] (4,-1) -- (4.9,-1);
        \draw[thick, ->] (4.85,-1) -- (5.2,-1);
        \draw[thick] (5.2,-1) -- (6,-1);
        \draw[thick, ->] (4,-1) -- (4,0);
        \draw[thick] (4,-1) -- (4,1);
        \draw[thick, ->] (6,-1) -- (6,0);
        \draw[thick] (6,-1) -- (6,1);
        \draw[thick, yellow] (5, 0) circle (0.15);
        \node[draw, circle, inner sep=0pt, minimum size=3pt, fill=white] at (5, 0.15) {};
        \node at (5,-1.5) {$x_5$};
        \draw[draw=none, fill=gray!40] (8,1) -- (10,1) -- (10,-1) -- (8,-1) -- (9, -0.15) to[out=0, in=-90] (9.15, 0) to[out=90, in=0] (9, 0.15) to[out=180, in=90] (8.85, 0) to[out=-90, in=180] (9, -0.15) -- (8, -1) -- (8,1);
        \draw[thick, ->] (8,1) -- (8.9,1);
        \draw[thick, ->] (8.85,1) -- (9.2,1);
        \draw[thick] (9.2,1) -- (10,1);
        \draw[thick, ->] (8,-1) -- (8.9,-1);
        \draw[thick, ->] (8.85,-1) -- (9.2,-1);
        \draw[thick] (9.2,-1) -- (10,-1);
        \draw[thick, ->] (8,-1) -- (8,0);
        \draw[thick] (8,-1) -- (8,1);
        \draw[thick, ->] (10,-1) -- (10,0);
        \draw[thick] (10,-1) -- (10,1);
        \draw[thick, yellow] (9, 0) circle (0.15);
        \node[draw, circle, inner sep=0pt, minimum size=3pt, fill=white] at (9, 0.15) {};
        \node at (9,-1.5) {$x_6$};
    \end{tikzpicture}
\end{center}
Here $x_5$ corresponds to the single edge in $\mathcal{E}_\partial$ and $x_6$ corresponds to our extra copy of $x_5$ in $\hat{\mathcal{E}}_\partial$. Using the anti-symmetric function, we find our antisymmetric matrix to be
$$Q =
{\begin{pmatrix}
    0 & 2 & 2 & -2 & 0 & -4\\
    -2 & 0 & -2 & -4 & 0 & -4\\
    -2 & 2 & 0 & -2 & 0 & -4\\
    2 & 4 & 2 & 0 & 0 & -4\\
    0 & 0 & 0 & 0 & 0 & 0\\
    4 & 4 & 4 & 4 & 0 & 0\\
\end{pmatrix}}.$$

Therefore, we are embedding our stated skein algebra, $\mathscr{S}\left( T^2 \setminus D^2 \right)$, into a quantum $6$-torus. This quantum torus is not simple as it has nontrivial center, however, we still need to consider the entire quantum $6$-torus as $\mathscr{S}(T^2 \setminus D^2)$ has a GK-dimension of $6$.

\begin{prop}\label{prop:embedding}
    Suppose $\varphi_{\mathcal{E}} : \mathscr{S}(T^2 \setminus D^2) \hookrightarrow \mathbb{T}^{6}(Q)$ is the algebra homomorphism defined as above. Let $y_1, y_2, y_3 \in \mathscr{S}(T^2 \setminus D^2)$ be the meridian, longitude, and $(1,1)$-curve respectively and $\partial =$
    $\begin{tikzpicture}[scale=0.6, baseline=-3]
        \MarkedTorusBackground
        \draw[thick] (1, 0) circle (0.4);
        \node[draw, circle, inner sep=0pt, minimum size=2pt, fill=white] at (1, 0.15) {};
    \end{tikzpicture}$ the boundary curve. Then
    \begin{align*}
        y_1 &\mapsto q^{-1}x_2x_3^{-1} + q^{-1}x_2^{-1}x_3 + qx_1^{2}x_3^{-1}x_4^{-1} + q^{2}x_1x_2^{-1}x_4^{-1}x_5\\
        y_2 &\mapsto qx_1x_3^{-1} + qx_1^{-1}x_3 + q^{-1}x_1^{-1}x_2x_3^{-1}x_4\\
        y_3 &\mapsto q^{-1}x_1x_4^{-1} + q^{-1}x_1^{-1}x_4 + q^{-1}x_1^{-1}x_2^{-1}x_3^{2} + x_2^{-1}x_3x_4^{-1}x_5 \\
        \partial &\mapsto q^{-2}x_2^{-1}x_4 + q^{-2}x_2x_4^{-1} + qx_1^{-1}x_3x_4^{-1}x_5 + qx_1x_3^{-1}x_4^{-1}x_5 + q^{-3}x_1^{-1}x_3^{-1}x_4x_5 \\
        &\phantom{\mapsto} + q^{3}x_1x_2^{-1}x_3^{-1}x_5 + q^{-1}x_1^{-1}x_2x_3^{-1}x_5 + q^{-1}x_1^{-1}x_2^{-1}x_3x_5 + q^{2}x_2^{-1}x_4^{-1}x_5^{2}
    \end{align*}
    In particular, the images of $y_i$ satisfy the commutation relations
    $$\left(q^2 - q^{-2}\right)^{-1}[\varphi_{\mathcal{E}}(y_i), \varphi_{\mathcal{E}}(y_{i+1})]_q = \varphi_{\mathcal{E}}(y_{i+2})$$
    for $i\mod 3$ and
    $$\varphi_{\mathcal{E}}(\partial) = q \varphi_{\mathcal{E}}(y_1) \varphi_{\mathcal{E}}(y_2) \varphi_{\mathcal{E}}(y_3) - q^2 \varphi_{\mathcal{E}}(y_1)^2 - q^{-2} \varphi_{\mathcal{E}}(y_2)^2 - q^2 \varphi_{\mathcal{E}}(y_3)^2 + q^2 + q^{-2}.$$
\end{prop}

\begin{proof}
    To track the relative heights, we use the notation described at the beginning of Appendix \ref{appendix:DiagCalc} for the following computations.
\begin{align*}
    \psi_{\mathcal{E}} (x_2 x_4 x_3) \cdot y_1 &= \left( q^{-3/2}
    \begin{tikzpicture}[baseline=-1]
        \MarkedTorusBackground[2][1]
        \draw[thick] (0.7, -1) to[out=90, in=270] (0.7, 0) to[out=90, in=180] (1, 0.15) to[out=120, in=270] (0.7, 1);
        \node[draw, circle, inner sep=0pt, minimum size=3pt, fill=white] at (1, 0.15) {};
    \end{tikzpicture}
    \begin{tikzpicture}[baseline=-1]
        \MarkedTorusBackground[2][1]
        \draw[thick] (1.3, -1) to[out=90, in=270] (1.3, 0) to[out=90, in=0] (1, 0.15) to[out=60, in=270] (1.3, 1);
        \node[draw, circle, inner sep=0pt, minimum size=3pt, fill=white] at (1, 0.15) {};
    \end{tikzpicture}
    \begin{tikzpicture}[baseline=-1]
        \MarkedTorusBackground[2][1]
        \draw[thick] (0.4, -1) -- (0.4, -0.2) to[out=90, in=180] (1, 0.15) to[out=45, in=180] (2, 0.5);
        \draw[thick] (0, 0.5) to[out=0, in=270] (0.4, 1);
        \node[draw, circle, inner sep=0pt, minimum size=3pt, fill=white] at (1, 0.15) {};
    \end{tikzpicture}
    \right)
    \begin{tikzpicture}[baseline=-1]
        \MarkedTorusBackground
        \draw[thick] (0, 0.7) -- (2, 0.7);
        \node[draw, circle, inner sep=0pt, minimum size=3pt, fill=white] at (1, 0.15) {};
    \end{tikzpicture} \\
    &= \left( q^{-3/2}
    \begin{tikzpicture}[baseline=-1]
        \MarkedTorusBackground[2][1]
        \draw[thick] (0.7, -1) to[out=90, in=270] (0.7, 0) to[out=90, in=180] (1, 0.15) to[out=120, in=270] (0.7, 1);
        \node[draw, circle, inner sep=0pt, minimum size=3pt, fill=white] at (1, 0.15) {};
    \end{tikzpicture}
    \begin{tikzpicture}[baseline=-1]
        \MarkedTorusBackground[2][1]
        \draw[thick] (1.3, -1) to[out=90, in=270] (1.3, 0) to[out=90, in=0] (1, 0.15) to[out=60, in=270] (1.3, 1);
        \node[draw, circle, inner sep=0pt, minimum size=3pt, fill=white] at (1, 0.15) {};
    \end{tikzpicture}
    \right)
    \begin{tikzpicture}[baseline=-1]
        \MarkedTorusBackground[2][1]
        \draw[thick] (0, 0.7) -- (2, 0.7);
        \draw[thick] (0.4, -1) -- (0.4, -0.2) to[out=90, in=180] (1, 0.15) to[out=45, in=180] (2, 0.5);
        \draw[line width=3mm, gray!40] (0.1, 0.5) to[out=0, in=270] (0.35, 0.9);
        \draw[thick] (0, 0.5) to[out=0, in=270] (0.4, 1);
        \node[draw, circle, inner sep=0pt, minimum size=3pt, fill=white] at (1, 0.15) {};
    \end{tikzpicture} \\
    &= \left( q^{-3/2}
    \begin{tikzpicture}[baseline=-1]
        \MarkedTorusBackground[2][1]
        \draw[thick] (0.7, -1) to[out=90, in=270] (0.7, 0) to[out=90, in=180] (1, 0.15) to[out=120, in=270] (0.7, 1);
        \node[draw, circle, inner sep=0pt, minimum size=3pt, fill=white] at (1, 0.15) {};
    \end{tikzpicture}
    \begin{tikzpicture}[baseline=-1]
        \MarkedTorusBackground[2][1]
        \draw[thick] (1.3, -1) to[out=90, in=270] (1.3, 0) to[out=90, in=0] (1, 0.15) to[out=60, in=270] (1.3, 1);
        \node[draw, circle, inner sep=0pt, minimum size=3pt, fill=white] at (1, 0.15) {};
    \end{tikzpicture}
    \right)\left(q 
    \begin{tikzpicture}[baseline=-1]
        \MarkedTorusBackground[2][1]
        \draw[thick] (0.7, -1) to[out=90, in=270] (0.7, 0) to[out=90, in=180] (1, 0.15) to[out=120, in=270] (0.7, 1);
        \node[draw, circle, inner sep=0pt, minimum size=3pt, fill=white] at (1, 0.15) {};
    \end{tikzpicture}
    + q^{-1}
    \begin{tikzpicture}[baseline=-1]
        \MarkedTorusBackground[2][1]
        \draw[thick] (0, 0.7) to[out=0, in=270] (0.4, 1);
        \draw[thick] (0.4, -1) -- (0.4, -0.2) to[out=90, in=180] (1, 0.15) to[out=45, in=180] (2, 0.5);
        \draw[thick] (0, 0.5) to[out=0, in=180] (2, 0.7);
        \node[draw, circle, inner sep=0pt, minimum size=3pt, fill=white] at (1, 0.15) {};
    \end{tikzpicture}
    \right) \\
    &= \left( q^{-3/2}
    \begin{tikzpicture}[baseline=-1]
        \MarkedTorusBackground[2][1]
        \draw[thick] (0.7, -1) to[out=90, in=270] (0.7, 0) to[out=90, in=180] (1, 0.15) to[out=120, in=270] (0.7, 1);
        \node[draw, circle, inner sep=0pt, minimum size=3pt, fill=white] at (1, 0.15) {};
    \end{tikzpicture}
    \right)\left(q
    \begin{tikzpicture}[baseline=-1]
        \MarkedTorusBackground[4][3][][][2][1]
        \draw[thick] (1.3, -1) to[out=90, in=270] (1.3, 0) to[out=90, in=0] (1, 0.15) to[out=60, in=270] (1.3, 1);
        \draw[thick] (0.7, -1) to[out=90, in=270] (0.7, 0) to[out=90, in=180] (1, 0.15) to[out=120, in=270] (0.7, 1);
        \node[draw, circle, inner sep=0pt, minimum size=3pt, fill=white] at (1, 0.15) {};
    \end{tikzpicture}
    + q^{-1}
    \begin{tikzpicture}[baseline=-1]
        \MarkedTorusBackground[4][2][][][3][1]
        \draw[thick] (0, 0.5) to[out=0, in=180] (2, 0.7);
        \draw[line width=3mm, gray!40] (1.3, -0.9) to[out=90, in=270] (1.3, 0) to[out=90, in=0] (1, 0.15) to[out=60, in=270] (1.3, 0.9);
        \draw[thick] (1.3, -1) to[out=90, in=270] (1.3, 0) to[out=90, in=0] (1, 0.15) to[out=60, in=270] (1.3, 1);
        \draw[thick] (0, 0.7) to[out=0, in=270] (0.4, 1);
        \draw[thick] (0.4, -1) -- (0.4, -0.2) to[out=90, in=180] (1, 0.15) to[out=45, in=180] (2, 0.5);
        \draw[thick, black, fill=white] (1, 0) circle (0.15);
        \node[draw, circle, inner sep=0pt, minimum size=3pt, fill=white] at (1, 0.15) {};
    \end{tikzpicture}
    \right) \\
    &= \left( q^{-3/2}
    \begin{tikzpicture}[baseline=-1]
        \MarkedTorusBackground[2][1]
        \draw[thick] (0.7, -1) to[out=90, in=270] (0.7, 0) to[out=90, in=180] (1, 0.15) to[out=120, in=270] (0.7, 1);
        \node[draw, circle, inner sep=0pt, minimum size=3pt, fill=white] at (1, 0.15) {};
    \end{tikzpicture}
    \right)\left( q
    \begin{tikzpicture}[baseline=-1]
        \MarkedTorusBackground[4][3][][][2][1]
        \draw[thick] (1.3, -1) to[out=90, in=270] (1.3, 0) to[out=90, in=0] (1, 0.15) to[out=60, in=270] (1.3, 1);
        \draw[thick] (0.7, -1) to[out=90, in=270] (0.7, 0) to[out=90, in=180] (1, 0.15) to[out=120, in=270] (0.7, 1);
        \node[draw, circle, inner sep=0pt, minimum size=3pt, fill=white] at (1, 0.15) {};
    \end{tikzpicture}
    +
    \begin{tikzpicture}[baseline=-1]
        \MarkedTorusBackground[4][2][][][3][1]
        \draw[thick] (0, 0.15) -- (1, 0.15) -- (2, 0.15);
        \draw[thick] (0, 0.5) to[out=0, in=120] (1, 0.15) to[out=60, in=180] (2, 0.5);
        \node[draw, circle, inner sep=0pt, minimum size=3pt, fill=white] at (1, 0.15) {};
    \end{tikzpicture}
    + q^{-2}
    \begin{tikzpicture}[baseline=-1]
        \MarkedTorusBackground[4][2][][][3][1]
        \draw[thick] (0, 0.5) to[out=0, in=270] (0.4, 1);
        \draw[thick] (0.4, -1) to[out=90, in=270] (1.4, -0.1) to[out=90, in=0] (1, 0.15) to[out=45, in=180] (2, 0.5);
        \draw[thick] (0, -1) to[out=45, in=180] (1, 0.15) to[out=60, in=225] (2, 1);
        \node[draw, circle, inner sep=0pt, minimum size=3pt, fill=white] at (1, 0.15) {};
    \end{tikzpicture}
    \right) \\
    &= q^{-3/2}
    \left( q
    \begin{tikzpicture}[baseline=-1]
        \MarkedTorusBackground[6][2][1][5][4][3]
        \draw[thick] (1.3, -1) to[out=90, in=270] (1.3, 0) to[out=90, in=0] (1, 0.15) to[out=60, in=270] (1.3, 1);
        \draw[thick] (0.7, -1) to[out=90, in=270] (0.7, 0) to[out=90, in=180] (1, 0.15) to[out=120, in=270] (0.7, 1);
        \draw[thick] (0.4, -1) to[out=90, in=270] (0.4, 0) to[out=90, in=180] (1, 0.15) to[out=120, in=270] (0.4, 1);
        \node[draw, circle, inner sep=0pt, minimum size=3pt, fill=white] at (1, 0.15) {};
    \end{tikzpicture}
    +
    \begin{tikzpicture}[baseline=-1]
        \MarkedTorusBackground[2][6][4][1][5][3]
        \draw[thick] (0, 0.15) -- (1, 0.15) -- (2, 0.15);
        \draw[thick] (0, 0.5) to[out=0, in=120] (1, 0.15) to[out=60, in=180] (2, 0.5);
        \draw[thick] (0.7, -1) to[out=90, in=270] (0.7, 0) to[out=90, in=180] (1, 0.15) to[out=120, in=270] (0.7, 1);
        \node[draw, circle, inner sep=0pt, minimum size=3pt, fill=white] at (1, 0.15) {};
    \end{tikzpicture}
    + q^{-2}
    \begin{tikzpicture}[baseline=-1]
        \MarkedTorusBackground[2][6][1][4][5][3]
        \draw[thick] (0, 0.5) to[out=0, in=270] (0.4, 1);
        \draw[thick] (0.4, -1) to[out=90, in=270] (1.4, -0.1) to[out=90, in=0] (1, 0.15) to[out=45, in=180] (2, 0.5);
        \draw[thick] (0, -1) to[out=45, in=180] (1, 0.15) to[out=60, in=225] (2, 1);
        \draw[line width=3mm, gray!40] (0.7, -0.8) -- (0.7, -0.5);
        \draw[thick] (0.7, -1) to[out=90, in=270] (0.7, -0.3) to[out=90, in=180] (1, 0.15) to[out=120, in=270] (0.7, 1);
        \draw[thick, black, fill=white] (1, 0) circle (0.15);
        \node[draw, circle, inner sep=0pt, minimum size=3pt, fill=white] at (1, 0.15) {};
    \end{tikzpicture}
    \right) \\
    &= q^{-3/2}
    \left( q
    \begin{tikzpicture}[baseline=-1]
        \MarkedTorusBackground[6][2][1][5][4][3]
        \draw[thick] (1.3, -1) to[out=90, in=270] (1.3, 0) to[out=90, in=0] (1, 0.15) to[out=60, in=270] (1.3, 1);
        \draw[thick] (0.7, -1) to[out=90, in=270] (0.7, 0) to[out=90, in=180] (1, 0.15) to[out=120, in=270] (0.7, 1);
        \draw[thick] (0.4, -1) to[out=90, in=270] (0.4, 0) to[out=90, in=180] (1, 0.15) to[out=120, in=270] (0.4, 1);
        \node[draw, circle, inner sep=0pt, minimum size=3pt, fill=white] at (1, 0.15) {};
    \end{tikzpicture}
    +
    \begin{tikzpicture}[baseline=-1]
        \MarkedTorusBackground[2][6][4][1][5][3]
        \draw[thick] (0, 0.15) -- (1, 0.15) -- (2, 0.15);
        \draw[thick] (0, 0.5) to[out=0, in=120] (1, 0.15) to[out=60, in=180] (2, 0.5);
        \draw[thick] (0.7, -1) to[out=90, in=270] (0.7, 0) to[out=90, in=180] (1, 0.15) to[out=120, in=270] (0.7, 1);
        \node[draw, circle, inner sep=0pt, minimum size=3pt, fill=white] at (1, 0.15) {};
    \end{tikzpicture}
    + q^{-1}
    \begin{tikzpicture}[baseline=-1]
        \MarkedTorusBackground[2][6][1][4][5][3]
        \draw[thick] (0, 0.15) -- (1, 0.15) -- (2, 0.15);
        \draw[thick] (0.4, -1) -- (0.4, -0.2) to[out=90, in=180] (1, 0.15) to[out=45, in=180] (2, 0.7);
        \draw[thick] (0, 0.7) to[out=0, in=270] (0.4, 1);
        \draw[thick] (1, 0.15) to[out=0, in=90] (1.3, -0.2) to[out=270, in=0] (1, -0.5) to[out=180, in=270] (0.7, -0.2) to[out=90, in=180] (1, 0.15);
        \node[draw, circle, inner sep=0pt, minimum size=3pt, fill=white] at (1, 0.15) {};
    \end{tikzpicture}
    +q^{-3}
    \begin{tikzpicture}[baseline=-1]
        \MarkedTorusBackground[2][6][1][4][5][3]
        \draw[thick] (1.3, -1) to[out=90, in=270] (1.3, 0) to[out=90, in=0] (1, 0.15) to[out=60, in=270] (1.3, 1);
        \draw[thick] (0, -1) to[out=45, in=180] (1, 0.15) to[out=60, in=225] (2, 1);
        \draw[thick] (0.7, -1) -- (0.7, -0.2) to[out=90, in=180] (1, 0.15) to[out=45, in=180] (2, 0.5);
        \draw[thick] (0, 0.5) to[out=0, in=270] (0.7, 1);
        \node[draw, circle, inner sep=0pt, minimum size=3pt, fill=white] at (1, 0.15) {};
    \end{tikzpicture}
    \right) \\
    &= q^{-3/2}
    \left( q
    \begin{tikzpicture}[baseline=-1]
        \MarkedTorusBackground[6][2][1][5][4][3]
        \draw[thick] (1.3, -1) to[out=90, in=270] (1.3, 0) to[out=90, in=0] (1, 0.15) to[out=60, in=270] (1.3, 1);
        \draw[thick] (0.7, -1) to[out=90, in=270] (0.7, 0) to[out=90, in=180] (1, 0.15) to[out=120, in=270] (0.7, 1);
        \draw[thick] (0.4, -1) to[out=90, in=270] (0.4, 0) to[out=90, in=180] (1, 0.15) to[out=120, in=270] (0.4, 1);
        \node[draw, circle, inner sep=0pt, minimum size=3pt, fill=white] at (1, 0.15) {};
    \end{tikzpicture}
    + q
    \begin{tikzpicture}[baseline=-1]
        \MarkedTorusBackground[2][6][4][1][3][5]
        \draw[thick] (0, 0.15) -- (1, 0.15) -- (2, 0.15);
        \draw[thick] (0, 0.5) to[out=0, in=120] (1, 0.15) to[out=60, in=180] (2, 0.5);
        \draw[thick] (0.7, -1) to[out=90, in=270] (0.7, 0) to[out=90, in=180] (1, 0.15) to[out=120, in=270] (0.7, 1);
        \node[draw, circle, inner sep=0pt, minimum size=3pt, fill=white] at (1, 0.15) {};
    \end{tikzpicture}
    + q^{-4}
    \begin{tikzpicture}[baseline=-1]
        \MarkedTorusBackground[6][4][2][3][1][5]
        \draw[thick] (0, 0.15) -- (1, 0.15) -- (2, 0.15);
        \draw[thick] (0.4, -1) -- (0.4, -0.2) to[out=90, in=180] (1, 0.15) to[out=45, in=180] (2, 0.7);
        \draw[thick] (0, 0.7) to[out=0, in=270] (0.4, 1);
        \draw[thick] (1, 0.15) to[out=0, in=90] (1.3, -0.2) to[out=270, in=0] (1, -0.5) to[out=180, in=270] (0.7, -0.2) to[out=90, in=180] (1, 0.15);
        \node[draw, circle, inner sep=0pt, minimum size=3pt, fill=white] at (1, 0.15) {};
    \end{tikzpicture}
    +q^{-3}
    \begin{tikzpicture}[baseline=-1]
        \MarkedTorusBackground[4][2][6][1][3][5]
        \draw[thick] (1.3, -1) to[out=90, in=270] (1.3, 0) to[out=90, in=0] (1, 0.15) to[out=60, in=270] (1.3, 1);
        \draw[thick] (0, -1) to[out=45, in=180] (1, 0.15) to[out=60, in=225] (2, 1);
        \draw[thick] (0.7, -1) -- (0.7, -0.2) to[out=90, in=180] (1, 0.15) to[out=45, in=180] (2, 0.5);
        \draw[thick] (0, 0.5) to[out=0, in=270] (0.7, 1);
        \node[draw, circle, inner sep=0pt, minimum size=3pt, fill=white] at (1, 0.15) {};
    \end{tikzpicture}
    \right)
\end{align*}
where the last equality comes from the height exchange relation, $(R_6)$.
Thus,
$$\varphi_{\mathcal{E}}(y_1)
= (x_2 x_4 x_3)^{-1} (x_2 x_4 x_3) \varphi_{\mathcal{E}}(y_1)
= (x_2 x_4 x_3)^{-1} \varphi_{\mathcal{E}}(\psi_{\mathcal{E}}(x_2 x_4 x_3)y_1)$$
$$= qx_3^{-1}x_4^{-1}x_2^{-1}x_2x_4x_2 + qx_3^{-1}x_4^{-1}x_2^{-1} x_2x_1^{2} + q^{-4}x_3^{-1}x_4^{-1}x_2^{-1} x_1x_3x_5 + q^{-3}x_3^{-1}x_4^{-1}x_2^{-1} x_3^{2}x_4$$
$$= q^{-1}x_2x_3^{-1} + q^{-1}x_2^{-1}x_3 + qx_1^{2}x_3^{-1}x_4^{-1} + q^{2}x_1x_2^{-1}x_4^{-1}x_5.$$

\begin{align*}
    \psi_{\mathcal{E}}\left( x_1 x_3 \right) y_2 &= q^{-1}
    \begin{tikzpicture}[baseline=-1]
            \MarkedTorusBackground[2][1]
            \draw[thick] (0, 0.15) -- (1, 0.15) -- (2, 0.15);
            \node[draw, circle, inner sep=0pt, minimum size=3pt, fill=white] at (1, 0.15) {};
        \end{tikzpicture}
        \begin{tikzpicture}[baseline=-1]
            \MarkedTorusBackground[2][1]
            \draw[thick] (0.5, -1) -- (0.5, -0.2) to[out=90, in=180] (1, 0.15) to[out=45, in=180] (2, 0.5);
            \draw[thick] (0, 0.5) to[out=0, in=270] (0.5, 1);
            \node[draw, circle, inner sep=0pt, minimum size=3pt, fill=white] at (1, 0.15) {};
        \end{tikzpicture}
        \begin{tikzpicture}[baseline=-1]
            \MarkedTorusBackground
            \draw[thick] (0.3, -1) -- (0.3, 1);
            \node[draw, circle, inner sep=0pt, minimum size=3pt, fill=white] at (1, 0.15) {};
        \end{tikzpicture}\\
        &= q^{-1}
        \begin{tikzpicture}[baseline=-1]
            \MarkedTorusBackground[2][1]
            \draw[thick] (0, 0.15) -- (1, 0.15) -- (2, 0.15);
            \node[draw, circle, inner sep=0pt, minimum size=3pt, fill=white] at (1, 0.15) {};
        \end{tikzpicture}
        \begin{tikzpicture}[baseline=-1]
            \MarkedTorusBackground[2][1]
            \draw[thick] (0.3, -1) -- (0.3, 1);
            \draw[thick] (0.5, -1) -- (0.5, -0.2) to[out=90, in=180] (1, 0.15) to[out=45, in=180] (2, 0.5);
            \draw[line width=3mm, gray!40] (0.05, 0.5) to[out=0, in=270] (0.5, 0.95);
            \draw[thick] (0, 0.5) to[out=0, in=270] (0.5, 1);
            \node[draw, circle, inner sep=0pt, minimum size=3pt, fill=white] at (1, 0.15) {};
        \end{tikzpicture}\\
        &= q^{-1}
        \begin{tikzpicture}[baseline=-1]
            \MarkedTorusBackground[2][1]
            \draw[thick] (0, 0.15) -- (1, 0.15) -- (2, 0.15);
            \node[draw, circle, inner sep=0pt, minimum size=3pt, fill=white] at (1, 0.15) {};
        \end{tikzpicture}
        \left( q
        \begin{tikzpicture}[baseline=-1]
            \MarkedTorusBackground[2][1]
            \draw[thick] (0.3, -1) -- (0.3, 0) to[out=90, in=270] (0.5, 1);
            \draw[thick] (0.5, -1) -- (0.5, -0.2) to[out=90, in=180] (1, 0.15) to[out=45, in=180] (2, 0.5);
            \draw[thick] (0, 0.5) to[out=0, in=270] (0.3, 1);
            \node[draw, circle, inner sep=0pt, minimum size=3pt, fill=white] at (1, 0.15) {};
        \end{tikzpicture}
        + q^{-1}
        \begin{tikzpicture}[baseline=-1]
            \MarkedTorusBackground[2][1]
            \draw[thick] (0, 0.15) -- (1, 0.15) -- (2, 0.15);
            \node[draw, circle, inner sep=0pt, minimum size=3pt, fill=white] at (1, 0.15) {};
        \end{tikzpicture}
        \right)\\
        &= q^{-1} \left( q
        \begin{tikzpicture}[baseline=-1]
            \MarkedTorusBackground[4][2][][][3][1]
            \draw[thick] (0.3, -1) -- (0.3, 0) to[out=90, in=270] (0.5, 1);
            \draw[thick] (0.5, -1) -- (0.5, -0.2) to[out=90, in=180] (1, 0.15) to[out=45, in=180] (2, 0.5);
            \draw[thick] (0, 0.5) to[out=0, in=270] (0.3, 1);
            \draw[line width=3mm, gray!40] (0.2, 0.15) -- (0.5, 0.15);
            \draw[thick] (0, 0.15) -- (1, 0.15) -- (2, 0.15);
            \node[draw, circle, inner sep=0pt, minimum size=3pt, fill=white] at (1, 0.15) {};
        \end{tikzpicture}
        + q^{-1}
        \begin{tikzpicture}[baseline=-1]
            \MarkedTorusBackground[4][2][][][1][3]
            \draw[thick] (0, 0.15) -- (1, 0.15) -- (2, 0.15);
            \draw[thick] (0, 0.5) to[out=0, in=120] (1, 0.15) to[out=60, in=180] (2, 0.5);
            \node[draw, circle, inner sep=0pt, minimum size=3pt, fill=white] at (1, 0.15) {};
        \end{tikzpicture}
        \right)\\
        &= q^{-1} \left( q^2
        \begin{tikzpicture}[baseline=-1]
            \MarkedTorusBackground[4][2][][][3][1]
            \draw[thick] (0, -1) to[out=45, in=180] (1, 0.15) to[out=60, in=225] (2, 1);
            \draw[thick] (0.7, -1) -- (0.7, -0.2) to[out=90, in=180] (1, 0.15) to[out=45, in=180] (2, 0.5);
            \draw[thick] (0, 0.5) to[out=0, in=270] (0.7, 1);
            \node[draw, circle, inner sep=0pt, minimum size=3pt, fill=white] at (1, 0.15) {};
        \end{tikzpicture}
        +
        \begin{tikzpicture}[baseline=-1]
            \MarkedTorusBackground[4][2][][][3][1]
            \draw[thick] (0.7, -1) to[out=90, in=270] (0.7, 0) to[out=90, in=180] (1, 0.15) to[out=120, in=270] (0.7, 1);
            \draw[thick] (1.3, -1) to[out=90, in=270] (1.3, 0) to[out=90, in=0] (1, 0.15) to[out=60, in=270] (1.3, 1);
            \node[draw, circle, inner sep=0pt, minimum size=3pt, fill=white] at (1, 0.15) {};
        \end{tikzpicture}
        + q^{-1}
        \begin{tikzpicture}[baseline=-1]
            \MarkedTorusBackground[4][2][][][1][3]
            \draw[thick] (0, 0.15) -- (1, 0.15) -- (2, 0.15);
            \draw[thick] (0, 0.5) to[out=0, in=120] (1, 0.15) to[out=60, in=180] (2, 0.5);
            \node[draw, circle, inner sep=0pt, minimum size=3pt, fill=white] at (1, 0.15) {};
        \end{tikzpicture}
        \right)\\
        &= q^{-1} \left( q^3
        \begin{tikzpicture}[baseline=-1]
            \MarkedTorusBackground[2][4][][][3][1]
            \draw[thick] (0, -1) to[out=45, in=180] (1, 0.15) to[out=60, in=225] (2, 1);
            \draw[thick] (0.7, -1) -- (0.7, -0.2) to[out=90, in=180] (1, 0.15) to[out=45, in=180] (2, 0.5);
            \draw[thick] (0, 0.5) to[out=0, in=270] (0.7, 1);
            \node[draw, circle, inner sep=0pt, minimum size=3pt, fill=white] at (1, 0.15) {};
        \end{tikzpicture}
        + q^{-1}
        \begin{tikzpicture}[baseline=-1]
            \MarkedTorusBackground[4][3][][][2][1]
            \draw[thick] (0.7, -1) to[out=90, in=270] (0.7, 0) to[out=90, in=180] (1, 0.15) to[out=120, in=270] (0.7, 1);
            \draw[thick] (1.3, -1) to[out=90, in=270] (1.3, 0) to[out=90, in=0] (1, 0.15) to[out=60, in=270] (1.3, 1);
            \node[draw, circle, inner sep=0pt, minimum size=3pt, fill=white] at (1, 0.15) {};
        \end{tikzpicture}
        + q^{-1}
        \begin{tikzpicture}[baseline=-1]
            \MarkedTorusBackground[4][2][][][1][3]
            \draw[thick] (0, 0.15) -- (1, 0.15) -- (2, 0.15);
            \draw[thick] (0, 0.5) to[out=0, in=120] (1, 0.15) to[out=60, in=180] (2, 0.5);
            \node[draw, circle, inner sep=0pt, minimum size=3pt, fill=white] at (1, 0.15) {};
        \end{tikzpicture}
        \right)\\
        &\Rightarrow \varphi_{\mathcal{E}} \left( y_2 \right) = (x_1 x_3)^{-1} (x_1 x_3) \varphi_{\mathcal{E}} \left( y_2 \right) = (x_1 x_3)^{-1} \varphi_{\mathcal{E}} \left( \psi_{\mathcal{E}}(x_1 x_3) y_2 \right)\\
        &= q^{3}x_3^{-1}x_1^{-1} x_3^2 + q^{-1} x_3^{-1} x_1^{-1} x_4 x_2 + q^{-1} x_3^{-1} x_1^{-1} x_1^2\\
        &= qx_1x_3^{-1} + qx_1^{-1}x_3 + q^{-1}x_1^{-1}x_2x_3^{-1}x_4.
    \end{align*}
    
    \begin{align*}
        \psi_{\mathcal{E}}\left( x_4 x_2 x_1 \right) y_3 &= \left( q^{-3/2}
        \begin{tikzpicture}[baseline=-1]
            \MarkedTorusBackground[2][1]
            \draw[thick] (1.3, -1) to[out=90, in=270] (1.3, 0) to[out=90, in=0] (1, 0.15) to[out=60, in=270] (1.3, 1);
            \node[draw, circle, inner sep=0pt, minimum size=3pt, fill=white] at (1, 0.15) {};
        \end{tikzpicture}
        \begin{tikzpicture}[baseline=-1]
            \MarkedTorusBackground[2][1]
            \draw[thick] (0.7, -1) to[out=90, in=270] (0.7, 0) to[out=90, in=180] (1, 0.15) to[out=120, in=270] (0.7, 1);
            \node[draw, circle, inner sep=0pt, minimum size=3pt, fill=white] at (1, 0.15) {};
        \end{tikzpicture}
        \begin{tikzpicture}[baseline=-1]
            \MarkedTorusBackground[2][1]
            \draw[thick] (0, 0.15) -- (1, 0.15) -- (2, 0.15);
            \node[draw, circle, inner sep=0pt, minimum size=3pt, fill=white] at (1, 0.15) {};
        \end{tikzpicture}
        \right)
        \begin{tikzpicture}[baseline=-1]
            \MarkedTorusBackground
            \draw[thick] (0.5, -1) -- (0.5, -0.5) to[out=90, in=225] (0.7, 0.2) to[out=45, in=180] (1.5, 0.5) -- (2, 0.5);
            \draw[thick] (0, 0.5) to[out=0, in=270] (0.5, 1);
            \node[draw, circle, inner sep=0pt, minimum size=3pt, fill=white] at (1, 0.15) {};
        \end{tikzpicture}\\
        &= \left( q^{-3/2}
        \begin{tikzpicture}[baseline=-1]
            \MarkedTorusBackground[2][1]
            \draw[thick] (1.3, -1) to[out=90, in=270] (1.3, 0) to[out=90, in=0] (1, 0.15) to[out=60, in=270] (1.3, 1);
            \node[draw, circle, inner sep=0pt, minimum size=3pt, fill=white] at (1, 0.15) {};
        \end{tikzpicture}
        \begin{tikzpicture}[baseline=-1]
            \MarkedTorusBackground[2][1]
            \draw[thick] (0.7, -1) to[out=90, in=270] (0.7, 0) to[out=90, in=180] (1, 0.15) to[out=120, in=270] (0.7, 1);
            \node[draw, circle, inner sep=0pt, minimum size=3pt, fill=white] at (1, 0.15) {};
        \end{tikzpicture}
        \right)
        \begin{tikzpicture}[baseline=-1]
            \MarkedTorusBackground[2][1]
            \draw[thick] (0.5, -1) -- (0.5, -0.5) to[out=90, in=225] (0.7, 0.2) to[out=45, in=180] (1.5, 0.5) -- (2, 0.5);
            \draw[thick] (0, 0.5) to[out=0, in=270] (0.5, 1);
            \draw[line width=3mm, gray!40] (0.3, 0.15) -- (0.75, 0.15);
            \draw[thick] (0, 0.15) -- (1, 0.15) -- (2, 0.15);
            \node[draw, circle, inner sep=0pt, minimum size=3pt, fill=white] at (1, 0.15) {};
        \end{tikzpicture}\\
        &= \left( q^{3/2}
        \begin{tikzpicture}[baseline=-1]
            \MarkedTorusBackground[2][1]
            \draw[thick] (1.3, -1) to[out=90, in=270] (1.3, 0) to[out=90, in=0] (1, 0.15) to[out=60, in=270] (1.3, 1);
            \node[draw, circle, inner sep=0pt, minimum size=3pt, fill=white] at (1, 0.15) {};
        \end{tikzpicture}
        \begin{tikzpicture}[baseline=-1]
            \MarkedTorusBackground[2][1]
            \draw[thick] (0.7, -1) to[out=90, in=270] (0.7, 0) to[out=90, in=180] (1, 0.15) to[out=120, in=270] (0.7, 1);
            \node[draw, circle, inner sep=0pt, minimum size=3pt, fill=white] at (1, 0.15) {};
        \end{tikzpicture}
        \right) \left( q
        \begin{tikzpicture}[baseline=-1]
            \MarkedTorusBackground[2][1]
            \draw[thick] (0, 0.7) to[out=0, in=270] (0.4, 1);
            \draw[thick] (0.4, -1) -- (0.4, -0.2) to[out=90, in=180] (1, 0.15) to[out=45, in=180] (2, 0.5);
            \draw[thick] (0, 0.5) to[out=0, in=180] (2, 0.7);
            \node[draw, circle, inner sep=0pt, minimum size=3pt, fill=white] at (1, 0.15) {};
        \end{tikzpicture}
        + q^{-1}
        \begin{tikzpicture}[baseline=-1]
            \MarkedTorusBackground[2][1]
            \draw[thick] (1.3, -1) to[out=90, in=270] (1.3, 0) to[out=90, in=0] (1, 0.15) to[out=60, in=270] (1.3, 1);
            \node[draw, circle, inner sep=0pt, minimum size=3pt, fill=white] at (1, 0.15) {};
        \end{tikzpicture}
        \right)\\
        &= \left( q^{-3/2}
        \begin{tikzpicture}[baseline=-1]
            \MarkedTorusBackground[2][1]
            \draw[thick] (1.3, -1) to[out=90, in=270] (1.3, 0) to[out=90, in=0] (1, 0.15) to[out=60, in=270] (1.3, 1);
            \node[draw, circle, inner sep=0pt, minimum size=3pt, fill=white] at (1, 0.15) {};
        \end{tikzpicture}
        \right) \left( q
        \begin{tikzpicture}[baseline=-1]
            \MarkedTorusBackground[2][4][][][1][3]
            \draw[thick] (0, 0.5) to[out=0, in=180] (2, 0.7);
            \draw[line width=3mm, gray!40] (0.9, 0) to[out=120, in=270] (0.75, 0.8);
            \draw[thick] (0, 0.7) to[out=0, in=270] (0.4, 1);
            \draw[thick] (0.4, -1) -- (0.4, -0.2) to[out=90, in=180] (1, 0.15) to[out=45, in=180] (2, 0.5);
            \draw[thick] (0.7, -1) to[out=90, in=270] (0.7, -0.1) to[out=90, in=180] (1, 0.15) to[out=120, in=270] (0.7, 1);
            \draw[thick, black, fill=white] (1, 0) circle (0.15);
            \node[draw, circle, inner sep=0pt, minimum size=3pt, fill=white] at (1, 0.15) {};
        \end{tikzpicture}
        + q^{-1}
        \begin{tikzpicture}[baseline=-1]
            \MarkedTorusBackground[2][1][][][4][3]
            \draw[thick] (0.7, -1) to[out=90, in=270] (0.7, 0) to[out=90, in=180] (1, 0.15) to[out=120, in=270] (0.7, 1);
            \draw[thick] (1.3, -1) to[out=90, in=270] (1.3, 0) to[out=90, in=0] (1, 0.15) to[out=60, in=270] (1.3, 1);
            \node[draw, circle, inner sep=0pt, minimum size=3pt, fill=white] at (1, 0.15) {};
        \end{tikzpicture}
        \right)\\
        &= \left( q^{-3/2}
        \begin{tikzpicture}[baseline=-1, scale=0.9]
            \MarkedTorusBackground[2][1]
            \draw[thick] (1.3, -1) to[out=90, in=270] (1.3, 0) to[out=90, in=0] (1, 0.15) to[out=60, in=270] (1.3, 1);
            \node[draw, circle, inner sep=0pt, minimum size=3pt, fill=white] at (1, 0.15) {};
        \end{tikzpicture}
        \right) \left( q^2
        \begin{tikzpicture}[baseline=-1, scale=0.9]
            \MarkedTorusBackground[2][4][][][1][3]
            \draw[thick] (0, 0.15) -- (1, 0.15) -- (2, 0.15);
            \draw[thick] (0, -0.5) to[out=0, in=180] (1, 0.15) to[out=180, in=120] (0.7, -0.3) to[out=300, in=180] (1, -0.5) -- (2, -0.5);
            \node[draw, circle, inner sep=0pt, minimum size=3pt, fill=white] at (1, 0.15) {};
        \end{tikzpicture}
        +
        \begin{tikzpicture}[baseline=-1, scale=0.9]
            \MarkedTorusBackground[2][4][][][1][3]
            \draw[thick] (0, -1) to[out=45, in=180] (1, 0.15) to[out=60, in=225] (2, 1);
            \draw[thick] (0.7, -1) -- (0.7, -0.2) to[out=90, in=180] (1, 0.15) to[out=45, in=180] (2, 0.5);
            \draw[thick] (0, 0.5) to[out=0, in=270] (0.7, 1);
            \node[draw, circle, inner sep=0pt, minimum size=3pt, fill=white] at (1, 0.15) {};
        \end{tikzpicture}
        + q^{-1}
        \begin{tikzpicture}[baseline=-1, scale=0.9]
            \MarkedTorusBackground[2][1][][][4][3]
            \draw[thick] (0.7, -1) to[out=90, in=270] (0.7, 0) to[out=90, in=180] (1, 0.15) to[out=120, in=270] (0.7, 1);
            \draw[thick] (1.3, -1) to[out=90, in=270] (1.3, 0) to[out=90, in=0] (1, 0.15) to[out=60, in=270] (1.3, 1);
            \node[draw, circle, inner sep=0pt, minimum size=3pt, fill=white] at (1, 0.15) {};
        \end{tikzpicture}
        \right)\\
        &= q^{-3/2} \left( q^2
        \begin{tikzpicture}[baseline=-1]
            \MarkedTorusBackground[4][6][3][2][5][1]
            \draw[thick] (0, -0.5) to[out=0, in=180] (1, 0.15) to[out=180, in=120] (0.7, -0.3) to[out=300, in=180] (1, -0.5) -- (2, -0.5);
            \draw[line width=3mm, gray!40] (1.3, -0.7) -- (1.3, -0.3);
            \draw[thick] (1.3, -1) to[out=90, in=270] (1.3, 0) to[out=90, in=0] (1, 0.15) to[out=60, in=270] (1.3, 1);
            \draw[thick] (0, 0.15) -- (1, 0.15) -- (2, 0.15);
            \node[draw, circle, inner sep=0pt, minimum size=3pt, fill=white] at (1, 0.15) {};
        \end{tikzpicture}
        +
        \begin{tikzpicture}[baseline=-1]
            \MarkedTorusBackground[4][6][2][3][5][1]
            \draw[thick] (0, -1) to[out=45, in=180] (1, 0.15) to[out=60, in=225] (2, 1);
            \draw[thick] (0.7, -1) -- (0.7, -0.2) to[out=90, in=180] (1, 0.15) to[out=45, in=180] (2, 0.5);
            \draw[thick] (0, 0.5) to[out=0, in=270] (0.7, 1);
            \draw[thick] (1.3, -1) to[out=90, in=270] (1.3, 0) to[out=90, in=0] (1, 0.15) to[out=60, in=270] (1.3, 1);
            \node[draw, circle, inner sep=0pt, minimum size=3pt, fill=white] at (1, 0.15) {};
        \end{tikzpicture}
        + q^{-1}
        \begin{tikzpicture}[baseline=-1]
            \MarkedTorusBackground[4][3][6][2][1][5]
            \draw[thick] (0.7, -1) to[out=90, in=270] (0.7, 0) to[out=90, in=180] (1, 0.15) to[out=120, in=270] (0.7, 1);
            \draw[thick] (1.3, -1) to[out=90, in=270] (1.3, 0) to[out=90, in=0] (1, 0.15) to[out=60, in=270] (1.3, 1);
            \draw[thick] (1.5, -1) to[out=90, in=270] (1.5, 0) to[out=90, in=0] (1, 0.15) to[out=60, in=270] (1.5, 1);
            \node[draw, circle, inner sep=0pt, minimum size=3pt, fill=white] at (1, 0.15) {};
        \end{tikzpicture}
        \right)\\
        &= q^{-3/2} \left( q^3
        \begin{tikzpicture}[baseline=-1, scale=0.9]
            \MarkedTorusBackground[4][6][3][2][5][1]
            \draw[thick] (0, -0.5) to[out=0, in=180] (1, 0.15);
            \draw[thick] (1, 0.15) to[out=0, in=180] (2, -0.5);
            \draw[thick] (0, 0.15) -- (1, 0.15) -- (2, 0.15);
            \draw[thick] (1, 0.15) to[out=180, in=120] (0.7, -0.3) to[out=300, in=90] (1, -1);
            \draw[thick] (1, 0.15) -- (1, 1);
            \node[draw, circle, inner sep=0pt, minimum size=3pt, fill=white] at (1, 0.15) {};
        \end{tikzpicture}
        + q
        \begin{tikzpicture}[baseline=-1, scale=0.9]
            \MarkedTorusBackground[4][6][3][2][5][1]
            \draw[thick] (0, 0.15) -- (1, 0.15) -- (2, 0.15);
            \draw[thick] (0.5, -1) -- (0.5, -0.2) to[out=90, in=180] (1, 0.15) to[out=45, in=180] (2, 0.5);
            \draw[thick] (0, 0.5) to[out=0, in=270] (0.5, 1);
            \draw[thick] (1, 0.15) to[out=0, in=90] (1.3, -0.2) to[out=270, in=0] (1, -0.5) to[out=180, in=270] (0.7, -0.2) to[out=90, in=180] (1, 0.15);
            \node[draw, circle, inner sep=0pt, minimum size=3pt, fill=white] at (1, 0.15) {};
        \end{tikzpicture}
        +
        \begin{tikzpicture}[baseline=-1, scale=0.9]
            \MarkedTorusBackground[4][6][2][3][5][1]
            \draw[thick] (0, -1) to[out=45, in=180] (1, 0.15) to[out=60, in=225] (2, 1);
            \draw[thick] (0.7, -1) -- (0.7, -0.2) to[out=90, in=180] (1, 0.15) to[out=45, in=180] (2, 0.5);
            \draw[thick] (0, 0.5) to[out=0, in=270] (0.7, 1);
            \draw[thick] (1.3, -1) to[out=90, in=270] (1.3, 0) to[out=90, in=0] (1, 0.15) to[out=60, in=270] (1.3, 1);
            \node[draw, circle, inner sep=0pt, minimum size=3pt, fill=white] at (1, 0.15) {};
        \end{tikzpicture}
        + q^{-1}
        \begin{tikzpicture}[baseline=-1, scale=0.9]
            \MarkedTorusBackground[4][3][6][2][1][5]
            \draw[thick] (0.7, -1) to[out=90, in=270] (0.7, 0) to[out=90, in=180] (1, 0.15) to[out=120, in=270] (0.7, 1);
            \draw[thick] (1.3, -1) to[out=90, in=270] (1.3, 0) to[out=90, in=0] (1, 0.15) to[out=60, in=270] (1.3, 1);
            \draw[thick] (1.5, -1) to[out=90, in=270] (1.5, 0) to[out=90, in=0] (1, 0.15) to[out=60, in=270] (1.5, 1);
            \node[draw, circle, inner sep=0pt, minimum size=3pt, fill=white] at (1, 0.15) {};
        \end{tikzpicture}
        \right)\\
        &= q^{-3/2} \left( q^7
        \begin{tikzpicture}[baseline=-1, scale=0.9]
            \MarkedTorusBackground[2][6][4][1][3][5]
            \draw[thick] (0, -0.5) to[out=0, in=180] (1, 0.15);
            \draw[thick] (1, 0.15) to[out=0, in=180] (2, -0.5);
            \draw[thick] (0, 0.15) -- (1, 0.15) -- (2, 0.15);
            \draw[thick] (1, 0.15) to[out=180, in=120] (0.7, -0.3) to[out=300, in=90] (1, -1);
            \draw[thick] (1, 0.15) -- (1, 1);
            \node[draw, circle, inner sep=0pt, minimum size=3pt, fill=white] at (1, 0.15) {};
        \end{tikzpicture}
        + q^2
        \begin{tikzpicture}[baseline=-1, scale=0.9]
            \MarkedTorusBackground[2][6][4][5][3][1]
            \draw[thick] (0, 0.15) -- (1, 0.15) -- (2, 0.15);
            \draw[thick] (0.5, -1) -- (0.5, -0.2) to[out=90, in=180] (1, 0.15) to[out=45, in=180] (2, 0.5);
            \draw[thick] (0, 0.5) to[out=0, in=270] (0.5, 1);
            \draw[thick] (1, 0.15) to[out=0, in=90] (1.3, -0.2) to[out=270, in=0] (1, -0.5) to[out=180, in=270] (0.7, -0.2) to[out=90, in=180] (1, 0.15);
            \node[draw, circle, inner sep=0pt, minimum size=3pt, fill=white] at (1, 0.15) {};
        \end{tikzpicture}
        + q^{-1}
        \begin{tikzpicture}[baseline=-1, scale=0.9]
            \MarkedTorusBackground[6][4][2][3][5][1]
            \draw[thick] (0, -1) to[out=45, in=180] (1, 0.15) to[out=60, in=225] (2, 1);
            \draw[thick] (0.7, -1) -- (0.7, -0.2) to[out=90, in=180] (1, 0.15) to[out=45, in=180] (2, 0.5);
            \draw[thick] (0, 0.5) to[out=0, in=270] (0.7, 1);
            \draw[thick] (1.3, -1) to[out=90, in=270] (1.3, 0) to[out=90, in=0] (1, 0.15) to[out=60, in=270] (1.3, 1);
            \node[draw, circle, inner sep=0pt, minimum size=3pt, fill=white] at (1, 0.15) {};
        \end{tikzpicture}
        + q^{-1}
        \begin{tikzpicture}[baseline=-1, scale=0.9]
            \MarkedTorusBackground[4][3][6][2][1][5]
            \draw[thick] (0.7, -1) to[out=90, in=270] (0.7, 0) to[out=90, in=180] (1, 0.15) to[out=120, in=270] (0.7, 1);
            \draw[thick] (1.3, -1) to[out=90, in=270] (1.3, 0) to[out=90, in=0] (1, 0.15) to[out=60, in=270] (1.3, 1);
            \draw[thick] (1.5, -1) to[out=90, in=270] (1.5, 0) to[out=90, in=0] (1, 0.15) to[out=60, in=270] (1.5, 1);
            \node[draw, circle, inner sep=0pt, minimum size=3pt, fill=white] at (1, 0.15) {};
        \end{tikzpicture}
        \right)\\
        &\Rightarrow \varphi_{\mathcal{E}}(y_3) = (x_4 x_2 x_1)^{-1} (x_4 x_2 x_1) \varphi_{\mathcal{E}}(y_3) = (x_4 x_2 x_1)^{-1} \varphi_{\mathcal{E}}\left( \psi_{\mathcal{E}} (x_4 x_2 x_1) y_3 \right)\\
        &= q^7x_1^{-1}x_2^{-1}x_4^{-1}x_2x_1^2 + q^{2}x_1^{-1}x_2^{-1}x_4^{-1}x_5x_1x_3 + q^{-1}x_1^{-1}x_2^{-1}x_3^{2} + q^{-1} x_1^{-1}x_4\\
        &= q^{-1}x_1x_4^{-1} + q^{-1}x_1^{-1}x_4 + q^{-1}x_1^{-1}x_2^{-1}x_3^{2} + x_2^{-1}x_3x_4^{-1}x_5.
    \end{align*}

    By Proposition~\ref{prop:embedIntoQT}, $\varphi_{\mathcal{E}}$ is an algebra embedding and so $\left(q^2 - q^{-2}\right)^{-1}[\varphi_{\mathcal{E}}(y_i), \varphi_{\mathcal{E}}(y_{i+1})]_q = \varphi_{\mathcal{E}}(y_{i+2})$ for $i\mod 3$ follows. We have also used a computer to independently verified the $q$-commuter relations between the images of the $y_i$. 

    Finally,
    \begin{align*}
        \psi_{\mathcal{E}}(x_1) \cdot \partial
        &= q^{-1/2}
        \begin{tikzpicture}[baseline=-1]
            \MarkedTorusBackground[2][1]
            \draw[thick] (1, 0) circle (0.4);
            \draw[line width=2mm, gray!40] (0.5, 0.15) -- (0.7, 0.15);
            \draw[line width=2mm, gray!40] (1.3, 0.15) -- (1.5, 0.15);
            \draw[thick] (0, 0.15) -- (2, 0.15);
            \node[draw, circle, inner sep=0pt, minimum size=3pt, fill=white] at (1, 0.15) {};
        \end{tikzpicture} \\
        &= q^{3/2}
        \begin{tikzpicture}[baseline=-1]
            \MarkedTorusBackground[2][1]
            \draw[thick] (0, 0.15) -- (1, 0.15) to[out=180, in=120] (0.7, -0.3) to[out=300, in=180] (1,-0.45) to[out=0, in=180] (2, 0.15);
            \node[draw, circle, inner sep=0pt, minimum size=3pt, fill=white] at (1, 0.15) {};
        \end{tikzpicture} + q^{-1/2}
        \begin{tikzpicture}[baseline=-1]
            \MarkedTorusBackground[2][1]
            \draw[thick] (0, 0.5) -- (2, 0.5);
            \draw[thick] (1, 0.15) to[out=0, in=90] (1.3, -0.2) to[out=270, in=0] (1, -0.5) to[out=180, in=270] (0.7, -0.2) to[out=90, in=180] (1, 0.15);
            \node[draw, circle, inner sep=0pt, minimum size=3pt, fill=white] at (1, 0.15) {};
        \end{tikzpicture} + q^{-5/2}
        \begin{tikzpicture}[baseline=-1]
            \MarkedTorusBackground[2][1]
            \draw[thick] (2, 0.15) -- (1, 0.15) to[out=0, in=60] (1.3, -0.3) to[out=240, in=0] (1,-0.45) to[out=180, in=0] (0, 0.15);
            \node[draw, circle, inner sep=0pt, minimum size=3pt, fill=white] at (1, 0.15) {};
        \end{tikzpicture} \\
        \Rightarrow \varphi_{\mathcal{E}}(\partial)
        &= x_{1}^{-1} \left[ q^{3/2}\left( q^{-7/2} x_1 x_2 x_4^{-1} + q^{-1/2} x_3 x_4^{-1} x_5 \right) + \left( q^{-1}x_2x_3^{-1} + q^{-1}x_2^{-1}x_3 + qx_1^{2}x_3^{-1}x_4^{-1}\right.\right. \\
        &\phantom{=} \left.\left. + q^{2}x_1x_2^{-1}x_4^{-1}x_5 \right)x_5 + q^{-5/2} \left( q^{-1/2}x_3^{-1}x_4x_5 + q^{11/2}x_1^{2}x_2^{-1}x_3^{-1}x_5 + q^{1/2}x_1x_2^{-1}x_4 \right) \right] \\
        &= q^{-2}x_2^{-1}x_4 + q^{-2}x_2x_4^{-1} + qx_1^{-1}x_3x_4^{-1}x_5 + qx_1x_3^{-1}x_4^{-1}x_5 + q^{-3}x_1^{-1}x_3^{-1}x_4x_5 \\
        &\phantom{.=} + q^{3}x_1x_2^{-1}x_3^{-1}x_5 + q^{-1}x_1^{-1}x_2x_3^{-1}x_5 + q^{-1}x_1^{-1}x_2^{-1}x_3x_5 + q^{2}x_2^{-1}x_4^{-1}x_5^{2}.
    \end{align*}
    Both of the calculations of $\varphi_{\mathcal{E}}\left( X_{1,\frac{1}{2}}(+,+)\right)$ and $\varphi_{\mathcal{E}}\left( X_{1,-\frac{1}{2}}(+,+)\right)$ for our substitutions can be found in Appendix \ref{appendix:DiagCalc}
\end{proof}

\subsection{A Module of Laurent Polynomials}
The goal of this section is to construct an action of $\mathscr{S}(T^2 \setminus D^2)$ on $\mathbb{C}[x^{\pm 1}, y^{\pm 1}, z^{\pm 1}, w^{\pm 1}]$.

\begin{prop}
    Let $\mathbb{T}^n(Q)$ be the quantum torus of $n$ generators, $\frac{\mathbb{K} \langle x_1^{\pm 1}, \cdots, x_n^{\pm 1} \rangle }{x_i x_j = q^{Q_{i,j}} x_j x_i}$, where $Q$ is the corresponding skew-symmetric integral matrix.
    If $k$ is the number of non-central generators of $\mathbb{T}^n$ and $q^{\frac{Q_{i,j}}{2}} \in \mathbb{K}$ for all $i,j$, then the commutative ring $\mathbb{K}[y_1^{\pm 1}, \cdots, y_{k-1}^{\pm 1}]$ has a well defined $\mathbb{T}^n$-module structure.
    In particular, if the first $k$ generators, $\{ x_1, \cdots, x_k \}$, are our non-commutative generators, then for each $i \in \{1, \cdots, k-1\}$ and $m \in \{ k+1, \cdots, n \}$, the operators
    $$x_i \cdot f(y_1, y_2, \cdots, y_{k-1}) := y_i f(q^{Q_{i, 1}/2}y_1, q^{Q_{i, 2}/2}y_2, \cdots, q^{Q_{i, {k-1}}/2}y_{k-1})$$
    $$x_{k} \cdot f(y_1, y_2, \cdots, y_{k-1}) := f(q^{Q_{k, 1}}y_1, q^{Q_{k, 2}}y_2, \cdots, q^{Q_{k, {k-1}}}y_{k-1})$$
    $$x_m \cdot f(y_1, y_2, \cdots, y_{k-1}) := f(y_1, y_2, \cdots, y_{k-1}).$$
    define a $\mathbb{T}^n$-module.
\end{prop}
\begin{proof}
    Our relations $x_i x_j = q^{Q_{i,j}} x_j x_i$ hold as for all $i,j \in \{ 1, \cdots, k-1 \}$
    \begin{align*}
        x_i x_j \cdot f(y_1, \cdots, y_{k-1}) &= q^{\frac{Q_{i, j}}{2}} y_i y_j f(q^{\frac{Q_{j, 1} + Q_{i, 1}}{2}}y_1, \cdots, q^{\frac{Q_{j, k-1} + Q_{i, k-1}}{2}}y_{k-1})\\
        &= q^{Q_{i, j}} x_j x_i \cdot f(y_1, \cdots, y_{k-1})\\
        x_i x_k \cdot f(y_1, \cdots, y_{k-1}) &= y_i f(q^{\frac{Q_{i, 1}}{2} + Q_{k, 1}}y_1, \cdots, q^{\frac{Q_{i, k-1}}{2} + Q_{k, {k-1}}}y_{k-1})\\
        &= q^{Q_{i, k}} x_k x_i \cdot f(y_1, \cdots, y_{k-1}).
    \end{align*}
    Clearly, for any $m, m^{\prime} \in \{ k+1, \cdots, n \}$ we have $x_i x_m \cdot f(y_1, \cdots, y_{k-1}) = x_m x_i \cdot f(y_1, \cdots, y_{k-1})$ and $x_m x_{m^{\prime}} \cdot f(y_1, \cdots, y_{k-1}) = x_m x_{m^{\prime}} \cdot f(y_1, \cdots, y_{k-1})$.
\end{proof}

By this proposition, the ring of Laurent polynomials in $4$ variables is a module over our quantum $6$-torus, $\mathbb{T}^6$. Since $\mathscr{S}\left( T^2 \setminus D^2 \right) \hookrightarrow \mathbb{T}^6$, the following actions endow $\mathbb{C}[x^{\pm 1}, y^{\pm 1}, z^{\pm 1}, w^{\pm 1}]$ with the structure of a $\mathscr{S}\left( T^2 \setminus D^2 \right)$-module.
\begin{alignat*}{2}
    x_1 \cdot f(x,y,z,w) &= x f(x,qy,qz,q^{-1}w) &\qquad x_2 \cdot f(x,y,z,w) &= y f(q^{-1}x,y,q^{-1}z,q^{-2}w)\\
    x_3 \cdot f(x,y,z,w) &= z f(q^{-1}x,qy,z,q^{-1}w)  &\qquad x_4 \cdot f(x,y,z,w) &= w f(qx,q^{2}y,qz,w)\\
    x_5 \cdot f(x,y,z,w) &= f(x,y,z,w)  &\qquad x_6 \cdot f(x,y,z,w) &= f(q^{4}x,q^{4}y,q^{4}z,q^{4}w)
\end{alignat*}
Thus, the actions of $y_1$, $y_2$, $y_3$, and $\partial$ on this module are (and have been verified via Python)

\begin{align*}
    y_1 \cdot f(x,y,z,w) &= yz^{-1} f(x, q^{-1}y, q^{-1}z, q^{-1}w) + y^{-1}z f(x, qy, qz, qw)\\
    &\quad + x^2z^{-1}w^{-1} f(x, q^{-1}y, qz, q^{-1}w) + xy^{-1}w^{-1} f(x, q^{-1}y, qz, qw)\\
    y_2 \cdot f(x,y,z,w) &= xz^{-1} f(qx, y, qz, w) + x^{-1}z f(q^{-1}x, y, q^{-1}z, w) + x^{-1}yz^{-1}w f(qx, y, q^{-1}z, w)\\
    y_3 \cdot f(x,y,z,w) &= xw^{-1} f(q^{-1}x, q^{-1}y, z, q^{-1}w) + x^{-1}w f(qx, qy, z, qw)\\
    &\quad + x^{-1}y^{-1}z^2 f(q^{-1}x, qy, z, qw) + y^{-1}zw^{-1} f(q^{-1}x, q^{-1}y, z, qw).\\
    \partial \cdot f(x,y,z,w) &= \frac{z f{\left(\frac{x}{q^{2}},\frac{y}{q^{2}},\frac{z}{q^{2}},w \right)}}{x w} + \frac{z f{\left(x,y, z, q^{2}w \right)}}{x y} + \frac{w f{\left(q^{2} x,q^{2} y,q^{2} z,q^{2} w \right)}}{y} \\
    &\quad + \frac{w f{\left(q^{2} x,y, z, q^{2}w \right)}}{z x} + \frac{y f{\left(\frac{x}{q^{2}},\frac{y}{q^{2}},\frac{z}{q^{2}},\frac{w}{q^{2}} \right)}}{w} + \frac{y f{\left(x,\frac{y}{q^{2}},\frac{z}{q^{2}},w \right)}}{z x} \\
    &\quad + \frac{x f{\left(x,\frac{y}{q^{2}},z,w \right)}}{w z} + \frac{x f{\left(q^{2} x,y,q^{2} z,q^{2} w \right)}}{y z} + \frac{f{\left(x,\frac{y}{q^{2}}, z, q^{2}w \right)}}{y w}
\end{align*}
where $\partial = q y_1 y_2 y_3 - q^2 y_1^2 - q^{-2}y_2^2 - q^2 y_3^2 + q^2 + q^{-2}$ is the closed curve parallel to the boundary.

Unfortunately, $\partial$ does not have any eigenvalues due to grade shifts, however, it does have an invariant subspace, $\mathbb{C}\left[ \left(\frac{x}{yz}\right)^{\pm 1}, \left(\frac{y}{zx}\right)^{\pm 1}, \left(\frac{z}{xy}\right)^{\pm 1}, \left(\frac{w}{y}\right)^{\pm 1} \right]$.
To condense the notation a bit, let $\kappa = k_1 + k_2 + k_3 + k_4$ and $\mathbf{k} = (k_1,k_2,k_3,k_4)$.
$$\partial \cdot \sum_{\substack{\kappa=-n \\ |k_i| \leq n }}^n c_{\mathbf{k}} \left(\frac{x}{yz}\right)^{k_1} \left(\frac{y}{zx}\right)^{k_2} \left(\frac{z}{xy}\right)^{k_3} \left(\frac{w}{y}\right)^{k_4}$$
$$= \sum_{\substack{\kappa=-n \\ |k_i| \leq n }}^n c_{\mathbf{k}} \left[
q^{2\kappa}\left(\frac{x}{yz}\right)^{k_1} \left(\frac{y}{zx}\right)^{k_2} \left(\frac{z}{xy}\right)^{k_3+1} \left(\frac{w}{y}\right)^{k_4-1}
+ q^{2k_4}\left(\frac{x}{yz}\right)^{k_1} \left(\frac{y}{zx}\right)^{k_2} \left(\frac{z}{xy}\right)^{k_3+1} \left(\frac{w}{y}\right)^{k_4}\right.$$
$$+ q^{2(\kappa - k_4)}\left(\frac{x}{yz}\right)^{k_1} \left(\frac{y}{zx}\right)^{k_2} \left(\frac{z}{xy}\right)^{k_3} \left(\frac{w}{y}\right)^{k_4-1}
+ q^{-2(\kappa - k_4)} \left(\frac{x}{yz}\right)^{k_1} \left(\frac{y}{zx}\right)^{k_2} \left(\frac{z}{xy}\right)^{k_3} \left(\frac{w}{y}\right)^{k_4+1}$$
$$+ q^{2(\kappa - 2k_2 + k_4)} \left(\frac{x}{yz}\right)^{k_1+1} \left(\frac{y}{zx}\right)^{k_2} \left(\frac{z}{xy}\right)^{k_3+1} \left(\frac{w}{y}\right)^{k_4-1}
+ q^{2(\kappa - 2k_2)} \left(\frac{x}{yz}\right)^{k_1+1} \left(\frac{y}{zx}\right)^{k_2} \left(\frac{z}{xy}\right)^{k_3} \left(\frac{w}{y}\right)^{k_4-1}$$
$$+ q^{2(k_4-2k_2)} \left(\frac{x}{yz}\right)^{k_1+1} \left(\frac{y}{zx}\right)^{k_2} \left(\frac{z}{xy}\right)^{k_3} \left(\frac{w}{y}\right)^{k_4}
+ q^{2(2k_1 + k_4)} \left(\frac{x}{yz}\right)^{k_1} \left(\frac{y}{zx}\right)^{k_2+1} \left(\frac{z}{xy}\right)^{k_3} \left(\frac{w}{y}\right)^{k_4}$$
$$\left. + q^{2(k_1 - k_2 - k_3 + k_4)} \left(\frac{x}{yz}\right)^{k_1} \left(\frac{y}{zx}\right)^{k_2+1} \left(\frac{z}{xy}\right)^{k_3} \left(\frac{w}{y}\right)^{k_4+1} \right]\\
= \sum_{\substack{\kappa=-(n+1) \\ |k_i| \leq n+1 }}^{n+2} c_{\mathbf{k}}^\prime \left(\frac{x}{yz}\right)^{k_1} \left(\frac{y}{zx}\right)^{k_2} \left(\frac{z}{xy}\right)^{k_3} \left(\frac{w}{y}\right)^{k_4}$$

Furthermore, our operators $y_1$, $y_2$, and $y_3$ also have invariant subspaces. For example, $\mathbb{C}\left[ \left( \frac{x}{z} \right)^{\pm 1}, \left( \frac{yw}{xz}\right)^{\pm 1} \right]$ is an invariant subspace with respect to the action of $y_2$ on this module. To see this, we compute the following.
$$y_2 \cdot \sum_{\substack{k_1+k_2=-n \\ |k_i| \leq n}}^n c_{k_1, k_2} \left(\frac{x}{z}\right)^{k_1}\left(\frac{y w}{x z}\right)^{k_2}$$
$$= \sum_{\substack{k_1+k_2=-n \\ |k_i| \leq n}}^n c_{k_1, k_2} \left[ q^{2 k_2}\left(\frac{x}{z}\right)^{k_1-1}\left(\frac{y w}{x z}\right)^{k_2} + q^{-2 k_2} \left(\frac{x}{z}\right)^{k_1+1}\left(\frac{y w}{x z}\right)^{k_2}  + q^{2 k_1} \left(\frac{x}{z}\right)^{k_1}\left(\frac{y w}{x z}\right)^{k_2+1} \right]$$
$$= \sum_{\substack{k_1+k_2=-(n+1) \\ |k_i| \leq n+1}}^{n+1} c_{k_1, k_2}^{\prime} \left(\frac{x}{z}\right)^{k_1}\left(\frac{y w}{x z}\right)^{k_2}$$
\section{DAHA modules from 3-manifolds}\label{section:mods3mflds}

The definition of a stated skein algebra presented in Section~\ref{section:SSA} differs slightly from the conventional definition found in the literature. However, these two definitions are isomorphic as described in \cite[Section 4.4]{LY22} -- for a further discussion see the thesis of the first author.
We will also need to  incorporate the conventional model, as it helps facilitate the explanation of the module structures discussed in this section.

Let $\Sigma^\prime$ be a (possibly punctured) oriented surface with (possibly empty) boundary, and $\mathcal{P} \subset \partial \Sigma^\prime$ be a finite nonempty set such that every connected component of $\partial \Sigma^\prime$ has at least one point in $\mathcal{P}$. Then $\Sigma = \Sigma^\prime \setminus \mathcal{P}$ is called a \textit{punctured bordered surface}. The surface $\Sigma^\prime$ is always uniquely determined by its punctured bordered surface, $\Sigma$. An \textit{ideal arc} on $\Sigma$ is an immersion $ \alpha : [0,1] \to \Sigma^\prime$ such that $\alpha(0), \alpha(1) \in \mathcal{P}$ and the restriction of $\alpha$ onto $(0,1)$ is an embedding into $\Sigma$.

For any tuple $(s,t) \in \Sigma \times [0,1]$, the number $t$ is called the \textit{height} and we say a vector at $(s,t)$ is \textit{vertical} if it is parallel to $s \times [0,1]$. A \textit{stated $\partial \Sigma$-tangle} is a tuple, $(\alpha, s)$ where $\alpha \subset \Sigma \times [0,1]$ is an unoriented, framed, compact, properly embedded $1$-dimensional submanifold such that at every point in $\partial \alpha = \alpha \cap \left( \partial \Sigma \times [0,1] \right)$ the framing is \textit{vertical}, and for every boundary edge $b \subset \Sigma$, $\partial \alpha \cap (b \times [0,1])$ have distinct heights.

The stated skein algebra of a punctured bordered surface, denoted $\mathscr{S}^{pb}(\Sigma)$ or $\mathscr{S}^{pb}(\Sigma^\prime)$, is the $\mathbb{C}$-module freely spanned by isotopy classes of stated $\partial \Sigma$-tangles modulo the following local relations.
\begin{center}\resizebox{0.9\width}{!}{
    \begin{tikzpicture}
        \draw[gray!40, fill=gray!40] (0,0) circle (1);
        \draw[thick] (-0.71, 0.71) -- (0.71, -0.71);
        \draw[line width=3mm, gray!40] (0.70, 0.70) -- (-0.70, -0.70);
        \draw[thick] (0.71, 0.71) -- (-0.71, -0.71);
        \node[text width=1cm] at (1.75,0) {$= q$};
        \draw[gray!40, fill=gray!40] (3,0) circle (1);
        \draw[thick] (2.29, 0.71) to[out=-60, in=60] (2.29, -0.71);
        \draw[thick] (3.71, 0.71) to[out=-120, in=120] (3.71, -0.71);
        \node[text width=1.25cm] at (4.75,0) {$+$ $q^{-1}$};
        \draw[gray!40, fill=gray!40] (6.25,0) circle (1);
        \draw[thick] (5.54, 0.71) to[out=-60, in=-120] (6.96, 0.71);
        \draw[thick] (5.54, -0.71) to[out=60, in=120] (6.96, -0.71);
        \node at (3,-1.4) {$(R_1^{pb})$ Skein Relation};
        \draw[gray!40, fill=gray!40] (10,0) circle (1);
        \draw[thick] (10,0) circle (.5);
        \node[text width=2.4cm] at (12.4,0) {$= (-q^2 - q^{-2})$};
        \draw[gray!40, fill=gray!40] (14.7,0) circle (1);
        \node at (12.4,-1.4) {$(R_2^{pb})$ Trivial Knot Relation};
    \end{tikzpicture}}
\end{center}
\begin{center}\resizebox{0.9\width}{!}{
    \begin{tikzpicture}
        \draw[gray!40, thick, fill=gray!40, domain=-45:225] plot ({cos(\x)}, {sin(\x)}) to[out=45, in=130] (0.71, -0.71);
        \draw[thick] (-0.71, -0.71) to[out=45, in=135] (0.71, -0.71);
        \draw[thick] (-0.25, -0.45) to[out=120, in=180] (0, 0.4) to[out=0, in=60] (0.25, -0.45);
        \path[thick, tips, ->] (-0.71, -0.71) to[out=45, in=135] (0.71, -0.71);
        \node (s1) at (-0.6, -0.3) {$-$};
        \node (s2) at (0.6, -0.3) {$+$};
        \node[text width=1.4cm] at (2,0) {$= q^{-1/2}$};
        \draw[gray!40, thick, fill=gray!40, domain=-45:225] plot ({3.8+cos(\x)}, {sin(\x)}) to[out=45, in=130] (4.51, -0.71);
        \draw[thick] (3.09, -0.71) to[out=45, in=130] (4.51, -0.71);
        \path[thick, tips, ->] (3.09, -0.71) to[out=45, in=130] (4.51, -0.71);
        \node at (2, -1.4) {$(R_3^{pb})$ Trivial Arc Relation 1};
        \draw[gray!40, thick, fill=gray!40, domain=-45:225] plot ({7.5+cos(\x)}, {sin(\x)}) to[out=45, in=130] (8.21, -0.71);
        \draw[thick] (6.79, -0.71) to[out=45, in=130] (8.21, -0.71);
        \draw[thick] (7.25, -0.45) to[out=120, in=180] (7.5,0.4) to[out=0, in=60] (7.75, -0.45);
        \path[thick, tips, ->] (6.79, -0.71) to[out=45, in=130] (8.21, -0.71);
        \node (s1) at (6.9,-0.3) {$-$};
        \node (s2) at (8.1,-0.3) {$-$};
        \node[text width=1.1cm] at (9.25,0) {$= 0 =$};
        \draw[gray!40, thick, fill=gray!40, domain=-45:225] plot ({11+cos(\x)}, {sin(\x)}) to[out=45, in=130] (11.71, -0.71);
        \draw[thick] (10.29, -0.71) to[out=45, in=130] (11.71, -0.71);
        \draw[thick] (10.75, -0.45) to[out=120, in=180] (11, 0.4) to[out=0, in=60] (11.25, -0.45);
        \path[thick, tips, ->] (10.29, -0.71) to[out=45, in=130] (11.71, -0.71);
        \node (s1) at (10.4,-0.3) {$+$};
        \node (s2) at (11.6,-0.3) {$+$};
        \node at (9.25,-1.4) {$(R_4^{pb})$ Trivial Arc Relation 2};
    \end{tikzpicture}}
\end{center}
\begin{center}\resizebox{0.9\width}{!}{
    \begin{tikzpicture}
        \draw[gray!40, thick, fill=gray!40, domain=-45:225] plot ({cos(\x)}, {sin(\x)}) to[out=45, in=130] (0.71, -0.71);
        \draw[thick] (-0.71, -0.71) to[out=45, in=135] (0.71, -0.71);
        \draw[thick] (-0.25, -0.45) -- (-0.71, 0.71);
        \draw[thick] (0.25, -0.45) -- (0.71, 0.71);
        \path[thick, tips, ->] (-0.71, -0.71) to[out=45, in=135] (0.71, -0.71);
        \node (s1) at (-0.6,-0.3) {$+$};
        \node (s2) at (0.6,-0.3) {$-$};
        \node[text width=1.1cm] at (1.75,0) {$= q^{-2}$};
        \draw[gray!40, thick, fill=gray!40, domain=-45:225] plot ({3.5+cos(\x)}, {sin(\x)}) to[out=45, in=130] (4.21, -0.71);
        \draw[thick] (2.79, -0.71) to[out=45, in=130] (4.21, -0.71);
        \draw[thick] (3.25, -0.45) -- (2.79, 0.71);
        \draw[thick] (3.75, -0.45) -- (4.21, 0.71);
        \path[thick, tips, ->] (2.79, -0.71) to[out=45, in=130] (4.21, -0.71);
        \node (s3) at (2.9,-0.3) {$-$};
        \node (s4) at (4.1,-0.3) {$+$};
        \node[text width=1.2cm] at (5.25,0) {$+$ $q^{1/2}$};
        \draw[gray!40, thick, fill=gray!40, domain=-45:225] plot ({7+cos(\x)}, {sin(\x)}) to[out=45, in=130] (7.71, -0.71);
        \draw[thick] (6.29, -0.71) to[out=45, in=130] (7.71, -0.71);
        \path[thick, tips, ->] (6.29, -0.71) to[out=45, in=130] (7.71, -0.71);
        \draw[thick] (6.29, 0.71) to[out=-60, in=180] (7,0) to[out=0, in=240] (7.71,0.71);
        \node at (3.5,-1.4) {$(R_5^{pb})$ State Exchange Relation};
    \end{tikzpicture}}
\end{center}

A stated $\partial \Sigma$-tangle, $\alpha$, is said to be in \textit{generic position} if the natural projection $\pi : \Sigma \times [0,1] \to \Sigma$ restricts to an embedding of $\alpha$, except for the possibility of transverse double points in the interior of $\Sigma$. Every stated $\partial \Sigma$-tangle is isotopic to one in generic position.

Finally, let $\mathfrak{o}$ be an orientation of $\partial \Sigma$, which may differ from the orientation inherited from $\Sigma$. We say a $\partial \Sigma$-tangle diagram, $D$, is \textit{$\mathfrak{o}$-ordered} if for each boundary component, the points of $\partial D$ increase when traversing in the direction of $\mathfrak{o}$. Every $\partial \Sigma$-tangle can be presented, after an appropriate isotopy, by an $\mathfrak{o}$-ordered $\partial \Sigma$-tangle diagram. Figure~\ref{fig:OrPBSEx} shows examples of $\mathfrak{o}$-ordered $\partial \Sigma$-tangle diagrams (without states) when our surface is a torus with boundary.

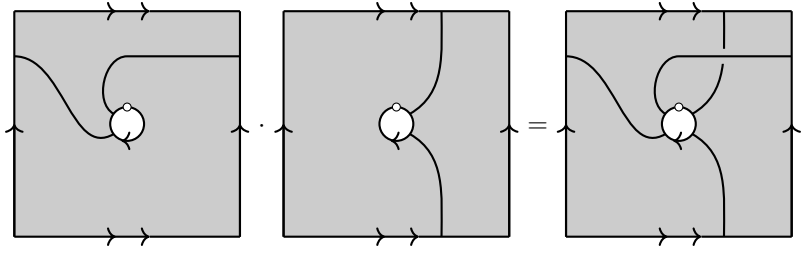
\begin{figure}
    \centering
    \begin{tikzpicture}[scale=1.5, baseline=-3]
        \MarkedTorusBackground
        \node[draw, circle, inner sep=0pt, minimum size=3pt, fill=white] at (1, 0.15) {};
        \draw[thick] (0, 0.6) to[out=0, in=210] (0.88, -0.09);
        \draw[thick] (0.88, 0.09) to[out=150, in=180] (1, 0.6) -- (2, 0.6);
        \path[thick, tips, ->] (1.5, -0.15) -- (0.95, -0.15);
    \end{tikzpicture} $\cdot$ 
    \begin{tikzpicture}[scale=1.5, baseline=-3]
        \MarkedTorusBackground
        \node[draw, circle, inner sep=0pt, minimum size=3pt, fill=white] at (1, 0.15) {};
        \draw[thick] (1.4, 1) to[out=270, in=30] (1.12, 0.09);
        \draw[thick] (1.12, -0.09) to[out=330, in=90] (1.4, -1);
        \path[thick, tips, ->] (1.5, -0.15) -- (0.95, -0.15);
    \end{tikzpicture} $=$
    \begin{tikzpicture}[scale=1.5, baseline=-3]
        \MarkedTorusBackground
        \node[draw, circle, inner sep=0pt, minimum size=3pt, fill=white] at (1, 0.15) {};
        \draw[thick] (1.4, 1) to[out=270, in=30] (1.12, 0.09);
        \draw[thick] (1.12, -0.09) to[out=330, in=90] (1.4, -1);
        \draw[line width=2mm, gray!40] (1.2, 0.6) -- (1.8, 0.6);
        \draw[thick] (0, 0.6) to[out=0, in=210] (0.88, -0.09);
        \draw[thick] (0.88, 0.09) to[out=150, in=180] (1, 0.6) -- (2, 0.6);
        \path[thick, tips, ->] (1.5, -0.15) -- (0.95, -0.15);
    \end{tikzpicture}
    \caption[An $\mathfrak{o}$-ordered diagram on a punctured bordered surface]{A product of $\mathfrak{o}$-ordered (stateless) $\partial \Sigma$-diagram where $\Sigma^\prime = T^2 \setminus D^2$ and $\mathfrak{o}$ has a clockwise orientation.}
    \label{fig:OrPBSEx}
\end{figure}

\subsection{General construction}\label{section:GenConst}

Just as the Kauffman bracket case in Section \ref{section:kbsm}, the action of $\mathscr{S}(T^2 \setminus D^2) \cong \mathscr{S}^{pb}(T^2 \setminus D^2)$ on $\mathscr{S}(M_K)$ for some knot $K$ should be dictated by the peripheral map, i.e. $T^2 \setminus D^2$ should somehow be ``glued'' to the boundary of $M_K$. Since $K$ is a knot in $S^3$, there is a canonical identification of $T^2$ with the boundary of $M_K$: one of the generating cycles in $T^2$ should map to a curve that generates $H_1(M_K) \cong \mathbb Z$ (the meridian), and the other should map to $0$ in $H_1(M_K)$ (the longitude). In order to incorporate the data of our marking, we'll need to add a marking to the knot complement and identify both markings in such a way that aligns with a module structure. Define $\mathcal{N}_K: [0,1] \hookrightarrow \partial M_K$ to be this marking and consider $\mathscr{S}(M_K, \mathcal{N}_K)$.

We will first explain this picture purely topologically and not yet consider the corresponding skein algebras. Let $\mathcal{P} \in \partial \left( T^2 \setminus D^2 \right)$ be our ideal point and denote $\Sigma^\prime := \left( T^2 \setminus D^2 \right) \setminus \mathcal{P}$ as the corresponding punctured bordered surface with clockwise orientation, $\mathfrak{o}$, on our boundary edge, $\partial \Sigma^\prime$. Identify $\partial \Sigma^\prime$ with the map $b : (0,1) \to \Sigma^\prime$ such that the induced orientation from $(0,1)$ is compatible with $\mathfrak{o}$.
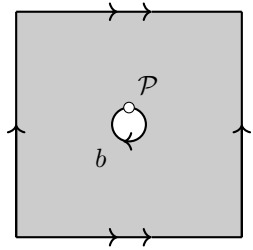
\begin{center}
    \begin{tikzpicture}[scale=1.5]
        \MarkedTorusBackground
        \path[thick, tips, ->] (1.5, -0.15) -- (0.95, -0.15);
        \node[draw, circle, inner sep=0pt, minimum size=4pt, fill=white] at (1, 0.15) {};
        \node at (1.17, 0.35) {$\mathcal{P}$};
        \node at (0.75, -0.3) {$b$};
    \end{tikzpicture}
\end{center}
Glue $T^2 \times [0,1]$ to the boundary of $M_K$ so that $T^2 \times \{0\}$ is identified with $\partial M_K$ via the peripheral embedding, $T^2 \times \{0\} \hookrightarrow \partial M_K$. At this point, it should be the case that and $T^2 \times (0,1] \bigcap M_K = \emptyset$ and $\mathcal{N}_K$ is on the interior of $\widetilde{M}_K := M_K \bigcup \left(T^2 \times [0,1]\right)$. We can compose embeddings of our thickened punctured bordered surface $$\iota_K : \Sigma^\prime \hookrightarrow T^2 \hookrightarrow \widetilde{M}_K$$ such that $b([\epsilon, 1-\epsilon]) \times \{0\}$ is identified with $\mathcal{N}_K$ for some small $\epsilon > 0$. We will eventually identify $b([\epsilon, 1-\epsilon]) \times \{1\}$ as the new marking.

Let $\alpha \in \mathscr{S}^{pb}(\Sigma^\prime)$ and $D \in \mathscr{S}(M_K)$. First isotope any closed curves in $D$ away from the boundary and isotope the endpoints of any stated $\partial M_K$-tangles of $D$ down the marking to $\mathcal{N}_K\left((0, \frac{1}{2})\right)$.
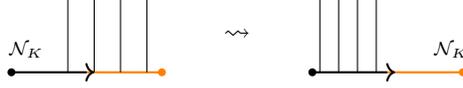
\begin{center}
    \begin{tikzpicture}
        \draw[thick] (0, 0) -- (1, 0);
        \draw[thick, orange] (1, 0) -- (2, 0);
        \node[circle, inner sep=0pt, minimum size=3pt, fill=black] at (0, 0) {};
        \node[circle, inner sep=0pt, minimum size=3pt, fill=orange] at (2, 0) {};
        \path[thick, tips, ->] (0, 0) -- (1.1, 0);
        \draw (0.75, 0) -- (0.75, 1);
        \draw (1.1, 0) -- (1.1, 1);
        \draw (1.45, 0) -- (1.45, 1);
        \draw (1.8, 0) -- (1.8, 1);
        \node at (0.2, 0.3) {$\scriptstyle{\mathcal{N}_K}$};
        \node at (3, 0.5) {$\rightsquigarrow$};
        \draw[thick] (4, 0) -- (5, 0);
        \draw[thick, orange] (5, 0) -- (6, 0);
        \node[circle, inner sep=0pt, minimum size=3pt, fill=black] at (4, 0) {};
        \node[circle, inner sep=0pt, minimum size=3pt, fill=orange] at (6, 0) {};
        \path[thick, tips, ->] (4, 0) -- (5.1, 0);
        \draw (4.1, 0) -- (4.1, 1);
        \draw (4.35, 0) -- (4.35, 1);
        \draw (4.6, 0) -- (4.6, 1);
        \draw (4.85, 0) -- (4.85, 1);
        \node at (5.85, 0.3) {$\scriptstyle{\mathcal{N}_K}$};
    \end{tikzpicture}
\end{center}
Similarly, isotope the endpoints of any stated $\partial \left(T^2 \setminus D^2\right)$-tangles in $\alpha$, remaining in generic position while doing so, up to $\mathcal{N}\left((\frac{1}{2}, 1-\epsilon)\right) \times [0,1]$.

We can extend any stated $\mathcal{N}_K$-tangles to end on $b \times \{1\}$ instead of $b \times \{0\}$. Specifically, let $(m,s)$ be a stated $\mathcal{N}_K$-tangle (using the definition in Section \ref{section:SSA}) and $e_1, e_2$ be the endpoints of $m$. In addition to lying on $\mathcal{N}_K$ in $M_K$, we can view $e_1$ and $e_2$ as lying in $b \times [0,1]$ in $\widetilde{M}_K$. Define $$\Tilde{m} = m \bigcup_{i = 1,2} \left( e_i \times [0,1] \right)$$ where $e_i \times [0,1] \subset b \times [0,1]$. Then $(\Tilde{m}, \Tilde{s})$ is a stated $\widetilde{\mathcal{N}}_K$-tangle where $\Tilde{s}(e_i \times \{1\}) = s(e_i)$ and $\widetilde{\mathcal{N}}_K := \iota_K( b([\epsilon, 1-\epsilon]) \times \{1\} )$.

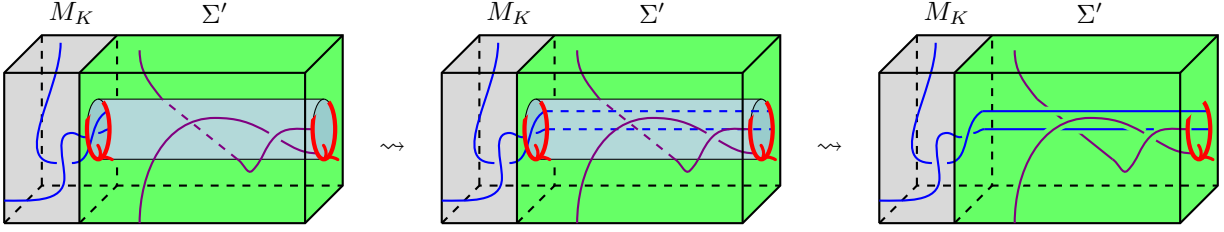
\begin{figure}
    \centering
    \begin{tikzpicture}[baseline=-2]
        \path[fill=gray!30] (0, -1) -- (0, 1) -- (0.5, 1.5) -- (1.5, 1.5) -- (1, 1) -- (1, -1) -- (0, -1);
        \path[fill=green!60] (1, -1) -- (1, 1) -- (1.5, 1.5) -- (4.5, 1.5) -- (4.5, -0.5) -- (4, -1) -- (1, -1);
        \draw[thick] (0, -1) -- (0, 1) -- (0.5, 1.5) -- (4.5, 1.5) -- (4.5, -0.5) -- (4, -1) -- (0, -1);
        \draw[thick, dashed] (0, -1) -- (0.5, -0.5) -- (4.5, -0.5);
        \draw[thick, dashed] (0.5, -0.5) -- (0.5, 1.5);
        \draw[thick, dashed] (1, -1) -- (1.5, -0.5) -- (1.5, 1.5);
        \path[fill=teal!30] (1.25, 0.65) -- (4.25, 0.65) -- (4.25, -0.15) -- (1.25, -0.15) -- (1.25, 0.65);
        \draw[fill=teal!40] (1.25, 0.25) ellipse (0.15 and 0.4);
        \draw[fill=teal!40] (4.25, 0.25) ellipse (0.15 and 0.4);
        \draw[thick, blue] (1.37, 0.49) to[out=210, in=0] (0.9, -0.2);
        \draw[thick, blue] (0.7, -0.2) to[out=180, in=270] (0.75, 1.4);
        \draw[thick, blue] (1.4, 0.25) to[out=180, in=0] (1.25, 0.2);
        \draw[thick, blue] (1.1, 0.15) to[out=200, in=0] (0.9, 0.2) to[out=180, in=90] (0.8, -0.4) to[out=270, in=0] (0, -0.7);
        \draw (1.25, 0.65) -- (4.25, 0.65);
        \draw (1.25, -0.15) -- (4.25, -0.15);
        \draw[thick, violet] (4.16, -0.07) to[out=180, in=330] (3.68, 0.1);
        \draw[thick, violet] (3.52, 0.2) to[out=150, in=0] (2.8, 0.4) to[out=180, in=90] (1.8, -1);
        \draw[thick, violet] (4.1, 0.25) to[out=180, in=45] (3.6, 0.15) to[out=225, in=45] (3.3, -0.3) to[out=225, in=330] (3.1, -0.15);
        \draw[thick, dashed, violet] (3.1, -0.15) -- (2.1, 0.65);
        \draw[thick, violet] (2.1, 0.65) to[out=129, in=270] (1.8, 1.3);
        \begin{scope}
            \clip (4, 0.3) -- (4.5, 0.8) -- (4.5, -0.5) -- (4, -1) -- (4, 0.3);
            \draw[ultra thick, red] (4.25, 0.25) ellipse (0.15 and 0.4);
            \path[very thick, red, tips, ->] (4.5, 0.1) -- (4.25, -0.15);
        \end{scope}
        \begin{scope}
            \clip (1, 0.3) -- (1.5, 0.8) -- (1.5, -0.5) -- (1, -1) -- (1, 0.3);
            \draw[ultra thick, red] (1.25, 0.25) ellipse (0.15 and 0.4);
            \path[very thick, red, tips, ->] (1.5, 0.1) -- (1.25, -0.15);
        \end{scope}
        \draw[thick] (0, 1) -- (4, 1);
        \draw[thick] (1, -1) -- (1, 1) -- (1.5, 1.5);
        \draw[thick] (4, -1) -- (4, 1) -- (4.5, 1.5);
        \node at (0.9, 1.8) {$M_K$};
        \node at (2.8, 1.8) {$\Sigma^\prime$};
    \end{tikzpicture} $\quad\rightsquigarrow\quad$ 
    \begin{tikzpicture}[baseline=-2]
        \path[fill=gray!30] (0, -1) -- (0, 1) -- (0.5, 1.5) -- (1.5, 1.5) -- (1, 1) -- (1, -1) -- (0, -1);
        \path[fill=green!60] (1, -1) -- (1, 1) -- (1.5, 1.5) -- (4.5, 1.5) -- (4.5, -0.5) -- (4, -1) -- (1, -1);
        \draw[thick] (0, -1) -- (0, 1) -- (0.5, 1.5) -- (4.5, 1.5) -- (4.5, -0.5) -- (4, -1) -- (0, -1);
        \draw[thick, dashed] (0, -1) -- (0.5, -0.5) -- (4.5, -0.5);
        \draw[thick, dashed] (0.5, -0.5) -- (0.5, 1.5);
        \draw[thick, dashed] (1, -1) -- (1.5, -0.5) -- (1.5, 1.5);
        \path[fill=teal!30] (1.25, 0.65) -- (4.25, 0.65) -- (4.25, -0.15) -- (1.25, -0.15) -- (1.25, 0.65);
        \draw[fill=teal!40] (1.25, 0.25) ellipse (0.15 and 0.4);
        \draw[fill=teal!40] (4.25, 0.25) ellipse (0.15 and 0.4);
        \draw[thick, blue] (1.37, 0.49) to[out=210, in=0] (0.9, -0.2);
        \draw[thick, blue] (0.7, -0.2) to[out=180, in=270] (0.75, 1.4);
        \draw[thick, blue] (1.4, 0.25) to[out=180, in=0] (1.25, 0.2);
        \draw[thick, blue] (1.1, 0.15) to[out=200, in=0] (0.9, 0.2) to[out=180, in=90] (0.8, -0.4) to[out=270, in=0] (0, -0.7);
        \draw[thick, dashed, blue] (1.37, 0.49) -- (4.37, 0.49);
        \draw[thick, dashed, blue] (1.4, 0.25) -- (4.4, 0.25);
        \draw (1.25, 0.65) -- (4.25, 0.65);
        \draw (1.25, -0.15) -- (4.25, -0.15);
        \draw[thick, violet] (4.16, -0.07) to[out=180, in=330] (3.68, 0.1);
        \draw[thick, violet] (3.52, 0.2) to[out=150, in=0] (2.8, 0.4) to[out=180, in=90] (1.8, -1);
        \draw[thick, violet] (4.1, 0.25) to[out=180, in=45] (3.6, 0.15) to[out=225, in=45] (3.3, -0.3) to[out=225, in=330] (3.1, -0.15);
        \draw[thick, dashed, violet] (3.1, -0.15) -- (2.1, 0.65);
        \draw[thick, violet] (2.1, 0.65) to[out=129, in=270] (1.8, 1.3);
        \begin{scope}
            \clip (4, 0.3) -- (4.5, 0.8) -- (4.5, -0.5) -- (4, -1) -- (4, 0.3);
            \draw[ultra thick, red] (4.25, 0.25) ellipse (0.15 and 0.4);
            \path[very thick, red, tips, ->] (4.5, 0.1) -- (4.25, -0.15);
        \end{scope}
        \begin{scope}
            \clip (1, 0.3) -- (1.5, 0.8) -- (1.5, -0.5) -- (1, -1) -- (1, 0.3);
            \draw[ultra thick, red] (1.25, 0.25) ellipse (0.15 and 0.4);
            \path[very thick, red, tips, ->] (1.5, 0.1) -- (1.25, -0.15);
        \end{scope}
        \draw[thick] (0, 1) -- (4, 1);
        \draw[thick] (1, -1) -- (1, 1) -- (1.5, 1.5);
        \draw[thick] (4, -1) -- (4, 1) -- (4.5, 1.5);
        \node at (0.9, 1.8) {$M_K$};
        \node at (2.8, 1.8) {$\Sigma^\prime$};
    \end{tikzpicture} $\quad\rightsquigarrow\quad$ 
    \begin{tikzpicture}[baseline=-2]
        \path[fill=gray!30] (0, -1) -- (0, 1) -- (0.5, 1.5) -- (1.5, 1.5) -- (1, 1) -- (1, -1) -- (0, -1);
        \path[fill=green!60] (1, -1) -- (1, 1) -- (1.5, 1.5) -- (4.5, 1.5) -- (4.5, -0.5) -- (4, -1) -- (1, -1);
        \draw[thick] (0, -1) -- (0, 1) -- (0.5, 1.5) -- (4.5, 1.5) -- (4.5, -0.5) -- (4, -1) -- (0, -1);
        \draw[thick, dashed] (0, -1) -- (0.5, -0.5) -- (4.5, -0.5);
        \draw[thick, dashed] (0.5, -0.5) -- (0.5, 1.5);
        \draw[thick, dashed] (1, -1) -- (1.5, -0.5) -- (1.5, 1.5);
        \draw[thick, blue] (1.37, 0.49) to[out=210, in=0] (0.9, -0.2);
        \draw[thick, blue] (0.7, -0.2) to[out=180, in=270] (0.75, 1.4);
        \draw[thick, blue] (1.4, 0.25) to[out=180, in=0] (1.25, 0.2);
        \draw[thick, blue] (1.1, 0.15) to[out=200, in=0] (0.9, 0.2) to[out=180, in=90] (0.8, -0.4) to[out=270, in=0] (0, -0.7);
        \draw[thick, violet] (3.1, -0.15) -- (2.1, 0.65);
        \draw[line width=1.2mm, green!60] (2, 0.49) -- (2.5, 0.49);
        \draw[line width=1.2mm, green!60] (2.5, 0.25) -- (2.7, 0.25);
        \draw[thick, blue] (1.37, 0.49) -- (4.37, 0.49);
        \draw[thick, blue] (1.4, 0.25) -- (3.65, 0.25);
        \draw[thick, blue] (4.15, 0.25) -- (4.4, 0.25);
        \draw[thick, violet] (4.1, 0.25) to[out=180, in=45] (3.6, 0.15) to[out=225, in=45] (3.3, -0.3) to[out=225, in=330] (3.1, -0.15);
        \draw[thick, violet] (2.1, 0.65) to[out=129, in=270] (1.8, 1.3);
        \draw[line width=1.3mm, green!60] (3.52, 0.2) to[out=150, in=0] (2.8, 0.4) to[out=180, in=90] (1.8, -0.9);
        \draw[thick, violet] (4.16, -0.07) to[out=180, in=330] (3.68, 0.1);
        \draw[thick, violet] (3.52, 0.2) to[out=150, in=0] (2.8, 0.4) to[out=180, in=90] (1.8, -1);
        \begin{scope}
            \clip (4, 0.3) -- (4.5, 0.8) -- (4.5, -0.5) -- (4, -1) -- (4, 0.3);
            \draw[ultra thick, red] (4.25, 0.25) ellipse (0.15 and 0.4);
            \path[very thick, red, tips, ->] (4.5, 0.1) -- (4.25, -0.15);
        \end{scope}
        \draw[thick] (0, 1) -- (4, 1);
        \draw[thick] (1, -1) -- (1, 1) -- (1.5, 1.5);
        \draw[thick] (4, -1) -- (4, 1) -- (4.5, 1.5);
        \node at (0.9, 1.8) {$M_K$};
        \node at (2.8, 1.8) {$\Sigma^\prime$};
    \end{tikzpicture}
    \caption[Extending stated endpoints in a knot complement]{Extending the endpoints of tangles (blue) in $M_K$ to end on the new boundary marking.}
    \label{fig:ExtComplementTangles}
\end{figure}

In other words, we can identify the stated endpoints of $m$ on $\mathcal{N}_K$ as living on $b \times \{0\}$ and then ``pull'' these endpoints along the $[0,1]$ component of $\partial \Sigma^\prime \times [0,1]$ so that the endpoints are now on $b \times \{1\}$.

Define $\alpha \cdot D$ as the stated $\widetilde{\mathcal{N}}_K$-tangle diagram in $\mathscr{S}(\widetilde{M}_K, \widetilde{\mathcal{N}}_K)$ consisting of the union of $\Tilde{m}$ for each stated $\mathcal{N}_K$-tangle $m$ in $D$, along with the induced stated $\widetilde{\mathcal{N}}_K$-tangle diagram, $\iota_{K}^{\ast}(\alpha)$. Because $\alpha$ is in generic position and any of its endpoints lie on $b\left((\frac{1}{2}, 1-\epsilon)\right)$, there is no overlap in this union, ensuring that this is a well-defined element in $\mathscr{S}(\widetilde{M}_K, \widetilde{\mathcal{N}}_K)$.

It is clear that $\mathscr{S}\left(\widetilde{M}_K, \widetilde{\mathcal{N}}_K \right)$ is isomorphic to $\mathscr{S}(M_K, \mathcal{N}_K)$ as stated skein algebras. Consequently, $\alpha \cdot D$ induces a left $\mathscr{S}^{pb}(\Sigma^\prime)$-module structure on $\mathscr{S}(M_K)$, and therefore, a left $\mathscr{S}(T^2 \setminus D^2)$-module structure.

\subsection{The torus}\label{section:solidtorus}
Let $K \subset S^3$ be the unknot. Then $M_K = S^3 \setminus K$ is the solid torus, which is homeomorphic to the thickened annulus. Therefore, $\mathscr{S}(M_K)$ also has an algebra structure and its stated skein algebra (with one marking) is isomorphic to the stated skein algebra of the once punctured monogon. This specific skein algebra was discussed in \cite[Prop.~4.26]{CL22} and is related to Majid's transmutation of $\mathcal{O}_q(SL_2)$ \cite{Maj95}. As an algebra, it is generated by the elements $\{ v_{\scriptscriptstyle{+},\scriptscriptstyle{+}}, v_{\scriptscriptstyle{+},\scriptscriptstyle{-}}, v_{\scriptscriptstyle{-},\scriptscriptstyle{+}}, v_{\scriptscriptstyle{-},\scriptscriptstyle{-}} \}$ where
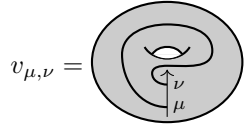
\begin{align*}
    v_{\mu,\nu} = \begin{tikzpicture}[baseline=-4]
        \draw[thick, fill=gray!40] (1, 0) ellipse (1 and 0.8);
        \draw[thick] (0.7, 0.2) to[out=300, in=240] (1.3, 0.2);
        \draw[thick, fill=white] (0.8, 0.1) to[out=60, in=120] (1.2, 0.1);
        \draw[fill=white] (0.8, 0.1) to[out=330, in=210] (1.2, 0.1);
        \draw[->] (1, -0.73) -- (1, -0.1);
        \draw[thick] (1, -0.3) to[out=180, in=270] (0.8, -0.15) to[out=90, in=180] (1, -0.05) to[out=0, in=270] (1.6, 0.15) to[out=90, in=0] (1, 0.5) to[out=180, in=90] (0.4, 0.1) to[out=270, in=180] (1, -0.6);
        \node at (1.16, -0.3) {$\scriptstyle{\nu}$};
        \node at (1.16, -0.6) {$\scriptstyle{\mu}$};
    \end{tikzpicture}
\end{align*}

\begin{lemma}
The following commutativity relations hold in this algebra:
\begin{align*}
    v_{-,-} v_{+,+} &= q^8 v_{+,+} v_{-,-} + q^{8}\left( q^2 - q^{-2} \right) v_{-,+}^{2} - q^{6}\left( q^2 - q^{-2} \right) v_{-,+}v_{+,-} - q^{5} \left( q^{4} - q^{-4} \right) \\
    v_{+,-} v_{-,+} &= v_{-,+} v_{+,-} \\
    v_{-,-} v_{-,+} &= q^4 v_{-,+} v_{-,-} \\
    v_{-,-} v_{+,-} &= v_{+,-} v_{-,-} + q^{4}\left( q^2 - q^{-2} \right) v_{-,-} v_{-,+} \\
    v_{-,+} v_{+,+} &= q^{4} v_{+,+} v_{-,+} \\
    v_{+,-} v_{+,+} &= v_{+,+} v_{+,-} + q^{4} \left( q^2 - q^{-2} \right) v_{+,+} v_{-,+}
\end{align*}
\end{lemma}
\begin{proof}
A diagrammatic proof for the first relation is provided in Appendix~\ref{appendix:AnnCommRel}. The calculations for the remaining relations are similar.
\end{proof}

Using Section~\ref{section:GenConst}, we see that
\begin{align*}
    X_{1,0}(\mu, \nu) \cdot f &= v_{\mu, \nu} \phantom{\cdot} f \\
    X_{2,0}(\mu, \nu) \cdot 1_K &= C_{\mu}^{\nu} \phantom{\cdot} 1_K \\
    X_{3,0}(\mu, \nu) \cdot 1_K &= -q^{-3} \phantom{\cdot} v_{\mu, \nu}
\end{align*}
where $1_K$ is the empty link (the identity) in $\mathscr{S}(M_K)$, $f \in \mathscr{S}(M_K)$, and $C_{\mu}^{\nu} =
\resizebox{0.6\width}{!}{\begin{tikzpicture}[baseline=-3]
    \draw[gray!40, thick, fill=gray!40, domain=-45:225] plot ({cos(\x)}, {sin(\x)}) to[out=45, in=130] (0.71, -0.71);
    \draw[thick] (-0.71, -0.71) to[out=45, in=135] (0.71, -0.71);
    \node[draw, circle, inner sep=0pt, minimum size=4pt, fill=white] (p1) at (0,-0.41) {};
    \draw[thick] (p1) to[out=45, in=-120] (0.31,0) to[out=60, in=0] (0,0.45) to[out=180, in=120] (-0.31,0) to[out=-60, in=150] (-0.15,-0.2);
    \node (s1) at (-0.6,-0.3) {$\mu$};
    \node (s2) at (0.6,-0.3) {$\nu$};
\end{tikzpicture}}$. Unlike in the Kauffman bracket case (see \cite{BS16}), the action of $X_{2,0}(\mu, \nu)$ is not diagonalizable. For example,
$$X_{2,0}(-,-) \cdot v_{+,-} = q^{-5/2}(q^2 - q^{-2}) v_{-,-}$$
and so the action is more complicated.

Once again, let $1_K$ be the empty link diagram in $\mathscr{S}(M_K)$ and $f \in \mathscr{S}(M_K)$. Denote $\partial_K$ as the closed curve, parallel to the boundary in $M_K$. 

\begin{lemma} We have the following identities in $\mathscr{S}(M_K)$:
\begin{align*}
    Y_1 \cdot f &= \partial_K \phantom{\cdot} f \\
    Y_2 \cdot 1_K &= \left( -q^2-q^{-2} \right) 1_K \\
    Y_2 \cdot v_{\mu, \nu} &= \left( -q^{4} - q^{-4} \right) v_{\mu, \nu} \\
    Y_3 \cdot 1_K &= -q^{-3} \phantom{\cdot} \partial_K \\
    \partial \cdot 1_K &= \left( -q^2-q^{-2} \right) 1_K \\
    \partial \cdot v_{\mu, \nu} &= \left( -q^6 - q^{-6} \right) v_{\mu, \nu} - (q^2 - q^{-2})^2 C_{\mu}^{\nu} \partial_K
\end{align*}
\end{lemma}
\begin{proof}
    These are all short diagrammatic computations that we omit.
\end{proof}

\appendix

\section{Appendix: Diagrammatic Calculations}\label{appendix:DiagCalc} 
Throughout all of these calculation, we use positive integers placed at the bottom of the diagrams to indicate the relative height ordering of the tangle endpoints, where larger values correspond to lower heights. For each diagram, starting from the leftmost endpoint moving clockwise with respect to our marked point, we assign these integers to the endpoints. The integers are read from left to right at the bottom, corresponding to this clockwise order. For example, $X_{1,0}(-,-)X_{2,0}(+,+) = \begin{tikzpicture}[scale=0.9, baseline=-1]
    \MarkedTorusBackground[4][2][][][3][1]
    \draw[thick] (0, 0.15) -- (1, 0.15) -- (2, 0.15);
    \draw[thick] (0.7, -1) to[out=90, in=270] (0.7, 0) to[out=90, in=180] (1, 0.15) to[out=120, in=270] (0.7, 1);
    \node[draw, circle, inner sep=0pt, minimum size=3pt, fill=white] at (1, 0.15) {};
    \node at (0.83, -0.25) {\footnotesize{$+$}};
    \node at (0.7, 0.28) {\footnotesize{$-$}};
    \node at (1.05, 0.4) {\footnotesize{$+$}};
    \node at (1.3, 0) {\footnotesize{$-$}};
\end{tikzpicture}$ where our heights correspond to
\begin{tikzpicture}[scale=0.8, baseline=-1]
    \draw[gray!40, thick, fill=gray!40, domain=-45:225] plot ({cos(\x)}, {sin(\x)}) to[out=45, in=130] (0.71, -0.71);
    \draw[thick] (-0.71, -0.71) to[out=45, in=135] (0.71, -0.71);
    \node[draw, circle, inner sep=0pt, minimum size=4pt, fill=white] (p1) at (0,-0.41) {};
    \draw[thick] (p1) -- (-0.93, -0.37);
    \draw[thick] (p1) -- (-0.91, 0.41);
    \draw[thick] (p1) -- (-0.2, 0.98);
    \draw[thick] (p1) -- (0.88,0.48);
    \node at (-0.7, -0.2) {\footnotesize{4}};
    \node at (-0.5, 0.3) {\footnotesize{2}};
    \node at (0, 0.4) {\footnotesize{3}};
    \node at (0.6, -0.1) {\footnotesize{1}};
\end{tikzpicture}.
Additionally, since all states are positive, we omit the states from the diagrams for simplicity.

\subsection{$\varphi_{\mathcal{E}}\left(X_{1, \frac{1}{2}}(+,+)\right)$}
\begin{align*}
    \psi_{\mathcal{E}}(x_2 x_3)
    \begin{tikzpicture}[baseline=-1]
        \MarkedTorusBackground[2][1]
        \draw[thick] (2, 0.15) -- (1, 0.15) to[out=0, in=60] (1.3, -0.3) to[out=240, in=0] (1,-0.45) to[out=180, in=0] (0, 0.15);
        \node[draw, circle, inner sep=0pt, minimum size=3pt, fill=white] at (1, 0.15) {};
    \end{tikzpicture}
    &= q^{-1} \left(
    \begin{tikzpicture}[baseline=-1]
        \MarkedTorusBackground[2][1]
        \draw[thick] (0.7, -1) to[out=90, in=270] (0.7, 0) to[out=90, in=180] (1, 0.15) to[out=120, in=270] (0.7, 1);
        \node[draw, circle, inner sep=0pt, minimum size=3pt, fill=white] at (1, 0.15) {};
    \end{tikzpicture}
    \begin{tikzpicture}[baseline=-1]
        \MarkedTorusBackground[2][1]
        \draw[thick] (0.5, -1) -- (0.5, -0.2) to[out=90, in=180] (1, 0.15) to[out=45, in=180] (2, 0.5);
        \draw[thick] (0, 0.5) to[out=0, in=270] (0.5, 1);
        \node[draw, circle, inner sep=0pt, minimum size=3pt, fill=white] at (1, 0.15) {};
    \end{tikzpicture}
    \right)
    \begin{tikzpicture}[baseline=-1]
        \MarkedTorusBackground[2][1]
        \draw[thick] (2, 0.15) -- (1, 0.15) to[out=0, in=60] (1.3, -0.3) to[out=240, in=0] (1,-0.45) to[out=180, in=0] (0, 0.15);
        \node[draw, circle, inner sep=0pt, minimum size=3pt, fill=white] at (1, 0.15) {};
    \end{tikzpicture}\\
    &= q^{-1}
    \begin{tikzpicture}[baseline=-1]
        \MarkedTorusBackground[2][1]
        \draw[thick] (0.7, -1) to[out=90, in=270] (0.7, 0) to[out=90, in=180] (1, 0.15) to[out=120, in=270] (0.7, 1);
        \node[draw, circle, inner sep=0pt, minimum size=3pt, fill=white] at (1, 0.15) {};
    \end{tikzpicture}
    \begin{tikzpicture}[baseline=-1]
        \MarkedTorusBackground[2][1][][][4][3]
        \draw[thick] (2, 0.15) -- (1, 0.15) to[out=0, in=60] (1.3, -0.3) to[out=240, in=0] (1,-0.45) to[out=180, in=0] (0, 0.15);
        \draw[line width=3mm, gray!40] (0.5, -0.8) -- (0.5, -0.2) to[out=90, in=180] (1, 0.15);
        \draw[thick] (0.5, -1) -- (0.5, -0.2) to[out=90, in=180] (1, 0.15) to[out=45, in=180] (2, 0.5);
        \draw[thick] (0, 0.5) to[out=0, in=270] (0.5, 1);
        \draw[thick, black, fill=white] (1, 0) circle (0.15);
        \node[draw, circle, inner sep=0pt, minimum size=3pt, fill=white] at (1, 0.15) {};
    \end{tikzpicture}\\
    &= q^{-1}
    \begin{tikzpicture}[baseline=-1]
        \MarkedTorusBackground[2][1]
        \draw[thick] (0.7, -1) to[out=90, in=270] (0.7, 0) to[out=90, in=180] (1, 0.15) to[out=120, in=270] (0.7, 1);
        \node[draw, circle, inner sep=0pt, minimum size=3pt, fill=white] at (1, 0.15) {};
    \end{tikzpicture}
    \left( q
    \begin{tikzpicture}[baseline=-1]
        \MarkedTorusBackground[2][1][][][4][3]
        \draw[thick] (1.5, -1) to[out=90, in=270] (1.5, 0) to[out=90, in=0] (1, 0.15) to[out=60, in=270] (1.5, 1);
        \draw[thick] (1, 0.15) to[out=0, in=90] (1.3, -0.2) to[out=270, in=0] (1, -0.5) to[out=180, in=270] (0.7, -0.2) to[out=90, in=180] (1, 0.15);
        \node[draw, circle, inner sep=0pt, minimum size=3pt, fill=white] at (1, 0.15) {};
    \end{tikzpicture}
    + q^{-1}
    \begin{tikzpicture}[baseline=-1]
        \MarkedTorusBackground[2][1][][][4][3]
        \draw[thick] (2, 1) -- (1, 0.15) to[out=0, in=60] (1.3, -0.3) to[out=240, in=0] (1,-0.45) to[out=180, in=45] (0, -1);
        \draw[thick] (0, 0.15) -- (1, 0.15) -- (2, 0.15);
        \node[draw, circle, inner sep=0pt, minimum size=3pt, fill=white] at (1, 0.15) {};
    \end{tikzpicture}
    \right)\\
    &=
    \begin{tikzpicture}[baseline=-1]
        \MarkedTorusBackground[4][2][1][3][6][5]
        \draw[thick] (1.5, -1) to[out=90, in=270] (1.5, 0) to[out=90, in=0] (1, 0.15) to[out=60, in=270] (1.5, 1);
        \draw[thick] (1, 0.15) to[out=0, in=90] (1.3, -0.2) to[out=270, in=0] (1, -0.5) to[out=180, in=270] (0.7, -0.2) to[out=90, in=180] (1, 0.15);
        \draw[thick] (0.5, -1) to[out=90, in=270] (0.5, 0) to[out=90, in=180] (1, 0.15) to[out=120, in=270] (0.5, 1);
        \node[draw, circle, inner sep=0pt, minimum size=3pt, fill=white] at (1, 0.15) {};
    \end{tikzpicture}
    + q^{-2}
    \begin{tikzpicture}[baseline=-1]
        \MarkedTorusBackground[2][4][1][3][6][5]
        \draw[thick] (2, 1) -- (1, 0.15) to[out=0, in=60] (1.3, -0.3) to[out=240, in=0] (1,-0.45) to[out=180, in=45] (0, -1);
        \draw[thick] (0, 0.15) -- (1, 0.15) -- (2, 0.15);
        \draw[line width=3mm, gray!40] (0.7, -0.8) to[out=90, in=270] (0.7, 0);
        \draw[thick] (0.7, -1) to[out=90, in=270] (0.7, 0) to[out=90, in=180] (1, 0.15) to[out=120, in=270] (0.7, 1);
        \node[draw, circle, inner sep=0pt, minimum size=3pt, fill=white] at (1, 0.15) {};
    \end{tikzpicture}\\
    &=
    \begin{tikzpicture}[baseline=-1]
        \MarkedTorusBackground[4][2][1][3][6][5]
        \draw[thick] (1.5, -1) to[out=90, in=270] (1.5, 0) to[out=90, in=0] (1, 0.15) to[out=60, in=270] (1.5, 1);
        \draw[thick] (1, 0.15) to[out=0, in=90] (1.3, -0.2) to[out=270, in=0] (1, -0.5) to[out=180, in=270] (0.7, -0.2) to[out=90, in=180] (1, 0.15);
        \draw[thick] (0.5, -1) to[out=90, in=270] (0.5, 0) to[out=90, in=180] (1, 0.15) to[out=120, in=270] (0.5, 1);
        \node[draw, circle, inner sep=0pt, minimum size=3pt, fill=white] at (1, 0.15) {};
    \end{tikzpicture}
    + q^{-1}
    \begin{tikzpicture}[baseline=-1]
        \MarkedTorusBackground[2][4][1][3][6][5]
        \draw[thick] (0, 0.15) -- (1, 0.15) -- (2, 0.15);
        \draw[thick] (0, 0.5) to[out=0, in=120] (1, 0.15) to[out=60, in=180] (2, 0.5);
        \draw[thick] (1, 0.15) to[out=0, in=90] (1.3, -0.2) to[out=270, in=0] (1, -0.5) to[out=180, in=270] (0.7, -0.2) to[out=90, in=180] (1, 0.15);
        \node[draw, circle, inner sep=0pt, minimum size=3pt, fill=white] at (1, 0.15) {};
    \end{tikzpicture}
    + q^{-3}
    \begin{tikzpicture}[baseline=-1]
        \MarkedTorusBackground[2][4][1][3][6][5]
        \draw[thick] (0, 0.15) -- (1, 0.15) -- (2, 0.15);
        \draw[thick] (1.3, -1) to[out=90, in=270] (1.3, 0) to[out=90, in=0] (1, 0.15) to[out=60, in=270] (1.3, 1);
        \draw[thick] (0.5, -1) -- (0.5, -0.2) to[out=90, in=180] (1, 0.15) to[out=45, in=180] (2, 0.5);
        \draw[thick] (0, 0.5) to[out=0, in=270] (0.5, 1);
        \node[draw, circle, inner sep=0pt, minimum size=3pt, fill=white] at (1, 0.15) {};
    \end{tikzpicture}\\
    &= q^{-2}
    \begin{tikzpicture}[baseline=-1]
        \MarkedTorusBackground[6][2][1][4][3][5]
        \draw[thick] (1.5, -1) to[out=90, in=270] (1.5, 0) to[out=90, in=0] (1, 0.15) to[out=60, in=270] (1.5, 1);
        \draw[thick] (1, 0.15) to[out=0, in=90] (1.3, -0.2) to[out=270, in=0] (1, -0.5) to[out=180, in=270] (0.7, -0.2) to[out=90, in=180] (1, 0.15);
        \draw[thick] (0.5, -1) to[out=90, in=270] (0.5, 0) to[out=90, in=180] (1, 0.15) to[out=120, in=270] (0.5, 1);
        \node[draw, circle, inner sep=0pt, minimum size=3pt, fill=white] at (1, 0.15) {};
    \end{tikzpicture}
    + q^{-6}
    \begin{tikzpicture}[baseline=-1]
        \MarkedTorusBackground[6][4][2][1][3][5]
        \draw[thick] (0, 0.15) -- (1, 0.15) -- (2, 0.15);
        \draw[thick] (0, 0.5) to[out=0, in=120] (1, 0.15) to[out=60, in=180] (2, 0.5);
        \draw[thick] (1, 0.15) to[out=0, in=90] (1.3, -0.2) to[out=270, in=0] (1, -0.5) to[out=180, in=270] (0.7, -0.2) to[out=90, in=180] (1, 0.15);
        \node[draw, circle, inner sep=0pt, minimum size=3pt, fill=white] at (1, 0.15) {};
    \end{tikzpicture}
    + q^{-7}
    \begin{tikzpicture}[baseline=-1]
        \MarkedTorusBackground[4][2][6][3][1][5]
        \draw[thick] (0, 0.15) -- (1, 0.15) -- (2, 0.15);
        \draw[thick] (1.3, -1) to[out=90, in=270] (1.3, 0) to[out=90, in=0] (1, 0.15) to[out=60, in=270] (1.3, 1);
        \draw[thick] (0.5, -1) -- (0.5, -0.2) to[out=90, in=180] (1, 0.15) to[out=45, in=180] (2, 0.5);
        \draw[thick] (0, 0.5) to[out=0, in=270] (0.5, 1);
        \node[draw, circle, inner sep=0pt, minimum size=3pt, fill=white] at (1, 0.15) {};
    \end{tikzpicture}\\
    \Rightarrow \varphi_{\mathcal{E}}(X_{1,\frac{1}{2}})
    &= (x_2 x_3)^{-1} (x_2 x_3) \varphi_{\mathcal{E}}\left( X_{1,\frac{1}{2}} \right)\\
    &= (x_2 x_3)^{-1} \varphi_{\mathcal{E}}\left( \psi_{\mathcal{E}} (x_2 x_3) X_{1,\frac{1}{2}} \right)\\
    &= x_3^{-1}x_2^{-1} \varphi_{\mathcal{E}} \left( q^{-2}x_2x_4x_5 + q^{-6}x_1^2x_5 + q^{-7}x_1x_3x_4 \right)\\
    &= q^{-1/2}x_3^{-1}x_4x_5 + q^{11/2}x_1^{2}x_2^{-1}x_3^{-1}x_5 + q^{1/2}x_1x_2^{-1}x_4.
\end{align*}

\subsection{$\varphi_{\mathcal{E}}\left(X_{1,-\frac{1}{2}}\right)$}
\begin{align*}
    \psi_{\mathcal{E}}(x_4) \cdot
    \begin{tikzpicture}[baseline=-1]
        \MarkedTorusBackground[2][1]
        \draw[thick] (0, 0.15) -- (1, 0.15) to[out=180, in=120] (0.7, -0.3) to[out=300, in=180] (1,-0.45) to[out=0, in=180] (2, 0.15);
        \node[draw, circle, inner sep=0pt, minimum size=3pt, fill=white] at (1, 0.15) {};
    \end{tikzpicture}
    &= q^{-1/2}
    \begin{tikzpicture}[baseline=-1]
        \MarkedTorusBackground[2][1]
        \draw[thick] (0, 0.15) -- (1, 0.15) to[out=180, in=120] (0.7, -0.3) to[out=300, in=180] (1,-0.45) to[out=0, in=180] (2, 0.15);
        \draw[line width=2mm, gray!40] (1.3, -0.5) -- (1.3, 0);
        \draw[thick] (1.3, -1) to[out=90, in=270] (1.3, 0) to[out=90, in=0] (1, 0.15) to[out=60, in=270] (1.3, 1);
        \node[draw, circle, inner sep=0pt, minimum size=3pt, fill=white] at (1, 0.15) {};
    \end{tikzpicture} \\
    &= q^{1/2}
    \begin{tikzpicture}[baseline=-1]
        \MarkedTorusBackground[4][3][][][2][1]
        \draw[thick] (0.7, -1) to[out=90, in=270] (0.7, 0) to[out=90, in=180] (1, 0.15) to[out=120, in=270] (0.7, 1);
        \draw[thick] (0, 0.15) -- (2, 0.15);
        \node[draw, circle, inner sep=0pt, minimum size=3pt, fill=white] at (1, 0.15) {};
    \end{tikzpicture} + q^{-3/2}
    \begin{tikzpicture}[baseline=-1]
        \MarkedTorusBackground[4][3][][][2][1]
        \draw[thick] (1, 0.15) to[out=0, in=90] (1.3, -0.2) to[out=270, in=0] (1, -0.5) to[out=180, in=270] (0.7, -0.2) to[out=90, in=180] (1, 0.15);
        \draw[thick] (0.5, -1) -- (0.5, -0.2) to[out=90, in=180] (1, 0.15) to[out=45, in=180] (2, 0.5);
        \draw[thick] (0, 0.5) to[out=0, in=270] (0.5, 1);
        \node[draw, circle, inner sep=0pt, minimum size=3pt, fill=white] at (1, 0.15) {};
    \end{tikzpicture} \\
    &= q^{3/2}
    \begin{tikzpicture}[baseline=-1]
        \MarkedTorusBackground[4][2][][][3][1]
        \draw[thick] (0.7, -1) to[out=90, in=270] (0.7, 0) to[out=90, in=180] (1, 0.15) to[out=120, in=270] (0.7, 1);
        \draw[thick] (0, 0.15) -- (2, 0.15);
        \node[draw, circle, inner sep=0pt, minimum size=3pt, fill=white] at (1, 0.15) {};
    \end{tikzpicture} + q^{1/2}
    \begin{tikzpicture}[baseline=-1]
        \MarkedTorusBackground[4][2][][][1][3]
        \draw[thick] (1, 0.15) to[out=0, in=90] (1.3, -0.2) to[out=270, in=0] (1, -0.5) to[out=180, in=270] (0.7, -0.2) to[out=90, in=180] (1, 0.15);
        \draw[thick] (0.5, -1) -- (0.5, -0.2) to[out=90, in=180] (1, 0.15) to[out=45, in=180] (2, 0.5);
        \draw[thick] (0, 0.5) to[out=0, in=270] (0.5, 1);
        \node[draw, circle, inner sep=0pt, minimum size=3pt, fill=white] at (1, 0.15) {};
    \end{tikzpicture} \\
    \Rightarrow \varphi_{\mathcal{E}}(X_{1,-\frac{1}{2}}(+,+))
    &= x_4^{-1}\left( q^{5/2} x_1 x_2 + q^{3/2} x_3 x_5 \right) \\
    &= q^{-7/2} x_1 x_2 x_4^{-1} + q^{-1/2} x_3 x_4^{-1} x_5
\end{align*}

\subsection{$v_{-,-} v_{+,+}$ Commutativity Relation}\label{appendix:AnnCommRel}
\begin{align*}
    v_{-,-} v_{+,+} =
    \begin{tikzpicture}[baseline=-1, scale=0.97]
        \PuncturedMonogonBackground
        \draw[thick] (-0.7, 0.7) to[out=-30, in=160] (0, 0.4) to[out=-20, in=90] (0.4, 0) to[out=270, in=20] (0, -0.4) to[out=200, in=30] (-0.7, -0.7);
        \draw[line width=2.5mm, gray!40] (0.1, 0.4) to[out=200, in=90] (-0.4, 0) to[out=270, in=160] (0.1, -0.4);
        \draw[thick] (0.7, 0.7) to[out=210, in=20] (0, 0.4) to[out=-160, in=90] (-0.4, 0) to[out=270, in=160] (0, -0.4) to[out=-20, in=150] (0.7, -0.7);
        \draw[thick, black, fill=white] (0, 1) circle (0.1);
        \node at (-0.85, 0.85) {$+$};
        \node at (-0.85, -0.85) {$+$};
        \node at (0.85, 0.85) {$-$};
        \node at (0.85, -0.85) {$-$};
    \end{tikzpicture}
    &= q^2 \begin{tikzpicture}[baseline=-1, scale=0.97]
        \PuncturedMonogonBackground
        \draw[thick] (-0.7, 0.7) to[out=-30, in=160] (0, 0.4) to[out=-20, in=90] (0.4, 0) to[out=270, in=20] (0, -0.4) to[out=200, in=30] (-0.7, -0.7);
        \draw[line width=2.5mm, gray!40] (0.1, 0.4) to[out=200, in=90] (-0.4, 0) to[out=270, in=160] (0.1, -0.4);
        \draw[thick] (0.7, 0.7) to[out=210, in=20] (0, 0.4) to[out=-160, in=90] (-0.4, 0) to[out=270, in=160] (0, -0.4) to[out=-20, in=150] (0.7, -0.7);
        \draw[thick, black, fill=white] (0, 1) circle (0.1);
        \node at (-0.85, 0.85) {$+$};
        \node at (-0.85, -0.85) {$-$};
        \node at (0.85, 0.85) {$-$};
        \node at (0.85, -0.85) {$+$};
    \end{tikzpicture}
    - q^{5/2} \begin{tikzpicture}[baseline=-1, scale=0.97]
        \PuncturedMonogonBackground
        \draw[thick] (-0.7, 0.7) to[out=-30, in=160] (0, 0.4) to[out=-20, in=90] (0.4, 0) to[out=270, in=0] (0, -0.4) to[out=180, in=270] (-0.4, 0);
        \draw[line width=2.5mm, gray!40] (0.6, 0.6) to[out=210, in=20] (0, 0.4) to[out=200, in=90] (-0.4, 0);
        \draw[thick] (-0.4, 0) to[out=90, in=200] (0, 0.4) to[out=20, in=210] (0.7, 0.7);
        \draw[thick, black, fill=white] (0, 1) circle (0.1);
        \node at (-0.85, 0.85) {$+$};
        \node at (0.85, 0.85) {$-$};
    \end{tikzpicture} \\
    &= q^2 \begin{tikzpicture}[baseline=-1, scale=0.97]
        \PuncturedMonogonBackground
        \draw[thick] (-0.7, 0.7) to[out=-30, in=160] (0, 0.4) to[out=-20, in=90] (0.4, 0) to[out=270, in=20] (0, -0.4) to[out=200, in=30] (-0.7, -0.7);
        \draw[line width=2.5mm, gray!40] (0.1, 0.4) to[out=200, in=90] (-0.4, 0) to[out=270, in=160] (0.1, -0.4);
        \draw[thick] (0.7, 0.7) to[out=210, in=20] (0, 0.4) to[out=-160, in=90] (-0.4, 0) to[out=270, in=160] (0, -0.4) to[out=-20, in=150] (0.7, -0.7);
        \draw[thick, black, fill=white] (0, 1) circle (0.1);
        \node at (-0.85, 0.85) {$+$};
        \node at (-0.85, -0.85) {$-$};
        \node at (0.85, 0.85) {$-$};
        \node at (0.85, -0.85) {$+$};
    \end{tikzpicture}
    - q^{3/2} \begin{tikzpicture}[baseline=-1, scale=0.97]
        \PuncturedMonogonBackground
        \draw[thick] (-0.7, 0.7) to[out=-30, in=210] (0.7, 0.7);
        \draw[thick] (0, 0) circle (0.35);
        \draw[thick, black, fill=white] (0, 1) circle (0.1);
        \node at (-0.85, 0.85) {$+$};
        \node at (0.85, 0.85) {$-$};
    \end{tikzpicture}
    - q^{7/2} \begin{tikzpicture}[baseline=-1, scale=0.97]
        \PuncturedMonogonBackground
        \draw[thick] (-0.7, -0.7) to[out=30, in=150] (0.7, -0.7);
        \draw[thick, black, fill=white] (0, 1) circle (0.1);
        \node at (-0.85, -0.85) {$+$};
        \node at (0.85, -0.85) {$-$};
    \end{tikzpicture} \\
    &= q^4 \begin{tikzpicture}[baseline=-1, scale=0.97]
        \PuncturedMonogonBackground
        \draw[thick] (-0.7, 0.7) to[out=-30, in=160] (0, 0.4) to[out=-20, in=90] (0.4, 0) to[out=270, in=20] (0, -0.4) to[out=200, in=30] (-0.7, -0.7);
        \draw[line width=2.5mm, gray!40] (0.1, 0.4) to[out=200, in=90] (-0.4, 0) to[out=270, in=160] (0.1, -0.4);
        \draw[thick] (0.7, 0.7) to[out=210, in=20] (0, 0.4) to[out=-160, in=90] (-0.4, 0) to[out=270, in=160] (0, -0.4) to[out=-20, in=150] (0.7, -0.7);
        \draw[thick, black, fill=white] (0, 1) circle (0.1);
        \node at (-0.85, 0.85) {$+$};
        \node at (-0.85, -0.85) {$-$};
        \node at (0.85, 0.85) {$+$};
        \node at (0.85, -0.85) {$-$};
    \end{tikzpicture}
    + q^{3/2} \begin{tikzpicture}[baseline=-1, scale=0.97]
        \PuncturedMonogonBackground
        \draw[thick] (-0.7, 0.7) to[out=-30, in=210] (0.7, 0.7);
        \draw[thick] (0, 0) circle (0.35);
        \draw[thick, black, fill=white] (0, 1) circle (0.1);
        \node at (-0.85, 0.85) {$+$};
        \node at (0.85, 0.85) {$-$};
    \end{tikzpicture}
    - q^{3/2} \begin{tikzpicture}[baseline=-1, scale=0.97]
        \PuncturedMonogonBackground
        \draw[thick] (-0.7, 0.7) to[out=-30, in=210] (0.7, 0.7);
        \draw[thick] (0, 0) circle (0.35);
        \draw[thick, black, fill=white] (0, 1) circle (0.1);
        \node at (-0.85, 0.85) {$+$};
        \node at (0.85, 0.85) {$-$};
    \end{tikzpicture} + q \\
    &= q^6 \begin{tikzpicture}[baseline=-1, scale=0.97]
        \PuncturedMonogonBackground
        \draw[thick] (-0.7, 0.7) to[out=-30, in=160] (0, 0.4) to[out=-20, in=90] (0.4, 0) to[out=270, in=20] (0, -0.4) to[out=200, in=30] (-0.7, -0.7);
        \draw[line width=2.5mm, gray!40] (0.1, 0.4) to[out=200, in=90] (-0.4, 0) to[out=270, in=160] (0.1, -0.4);
        \draw[thick] (0.7, 0.7) to[out=210, in=20] (0, 0.4) to[out=-160, in=90] (-0.4, 0) to[out=270, in=160] (0, -0.4) to[out=-20, in=150] (0.7, -0.7);
        \draw[thick, black, fill=white] (0, 1) circle (0.1);
        \node at (-0.85, 0.85) {$-$};
        \node at (-0.85, -0.85) {$+$};
        \node at (0.85, 0.85) {$+$};
        \node at (0.85, -0.85) {$-$};
    \end{tikzpicture}
    + q^{7/2} \begin{tikzpicture}[baseline=-1, scale=0.97]
        \PuncturedMonogonBackground
        \draw[thick] (-0.7, 0.7) to[out=-30, in=210] (0.7, 0.7);
        \draw[thick] (0, 0) circle (0.35);
        \draw[thick, black, fill=white] (0, 1) circle (0.1);
        \node at (-0.85, 0.85) {$-$};
        \node at (0.85, 0.85) {$+$};
    \end{tikzpicture}
    + q \\
    &= q^8 \begin{tikzpicture}[baseline=-1, scale=0.97]
        \PuncturedMonogonBackground
        \draw[thick] (-0.7, 0.7) to[out=-30, in=160] (0, 0.4) to[out=-20, in=90] (0.4, 0) to[out=270, in=20] (0, -0.4) to[out=200, in=30] (-0.7, -0.7);
        \draw[line width=2.5mm, gray!40] (0.1, 0.4) to[out=200, in=90] (-0.4, 0) to[out=270, in=160] (0.1, -0.4);
        \draw[thick] (0.7, 0.7) to[out=210, in=20] (0, 0.4) to[out=-160, in=90] (-0.4, 0) to[out=270, in=160] (0, -0.4) to[out=-20, in=150] (0.7, -0.7);
        \draw[thick, black, fill=white] (0, 1) circle (0.1);
        \node at (-0.85, 0.85) {$-$};
        \node at (-0.85, -0.85) {$-$};
        \node at (0.85, 0.85) {$+$};
        \node at (0.85, -0.85) {$+$};
    \end{tikzpicture}
    - q^{17/2} \begin{tikzpicture}[baseline=-1, scale=0.97]
        \PuncturedMonogonBackground
        \draw[thick] (-0.7, 0.7) to[out=-30, in=160] (0, 0.4) to[out=-20, in=90] (0.4, 0) to[out=270, in=0] (0, -0.4) to[out=180, in=270] (-0.4, 0);
        \draw[line width=2.5mm, gray!40] (0.6, 0.6) to[out=210, in=20] (0, 0.4) to[out=200, in=90] (-0.4, 0);
        \draw[thick] (-0.4, 0) to[out=90, in=200] (0, 0.4) to[out=20, in=210] (0.7, 0.7);
        \draw[thick, black, fill=white] (0, 1) circle (0.1);
        \node at (-0.85, 0.85) {$-$};
        \node at (0.85, 0.85) {$+$};
    \end{tikzpicture}
    + q^{7/2} \begin{tikzpicture}[baseline=-1, scale=0.97]
        \PuncturedMonogonBackground
        \draw[thick] (-0.7, 0.7) to[out=-30, in=210] (0.7, 0.7);
        \draw[thick] (0, 0) circle (0.35);
        \draw[thick, black, fill=white] (0, 1) circle (0.1);
        \node at (-0.85, 0.85) {$-$};
        \node at (0.85, 0.85) {$+$};
    \end{tikzpicture}
    + q \\
    &= q^8 v_{+,+} v_{-,-} - q^{11/2}\left( q^2 - q^{-2} \right) \begin{tikzpicture}[baseline=-1, scale=0.97]
        \PuncturedMonogonBackground
        \draw[thick] (-0.7, 0.7) to[out=-30, in=210] (0.7, 0.7);
        \draw[thick] (0, 0) circle (0.35);
        \draw[thick, black, fill=white] (0, 1) circle (0.1);
        \node at (-0.85, 0.85) {$-$};
        \node at (0.85, 0.85) {$+$};
    \end{tikzpicture}
    - q^{5} \left( q^{4} - q^{-4} \right) \\
    &= q^8 v_{+,+} v_{-,-} - q^{11/2}\left( q^2 - q^{-2} \right) v_{-,+} \left( q^{1/2}v_{+,-} - q^{5/2}v_{-,+} \right) - q^{5} \left( q^{4} - q^{-4} \right) \\
    &= q^8 v_{+,+}v_{-,-} + q^{8}\left( q^2 - q^{-2} \right) v_{-,+}^{2} - q^{6}\left( q^2 - q^{-2} \right) v_{-,+}v_{+,-} - q^{5} \left( q^{4} - q^{-4} \right)
\end{align*}

\section{Appendix: Towards a PBW basis}\label{section:pbw}
It would be ideal if we could  find a PBW basis for $\mathscr{S}\left(T^2 \setminus D^2\right)$, which would let us prove that we have found enough relations to give a presentation of this algebra. However, establishing the basis for this algebra has proven to be quite challenging, primarily due to the rapid escalation of calculations. Attempts to apply similar techniques used in the Kauffman bracket case have not been successful. While we have the relation $\frac{1}{q^2 - q^{-2}}[Y_i, Y_{i+1}]_q = Y_{i+2}$ as the Kauffman bracket case, computing an analogous relation on tangles instead yields equations \eqref{eq:shiftup} and \eqref{eq:shiftdown}, indicating the possible need to consider these half-twists when establishing a basis. Paradoxically, we also have the equality
\begin{align}\label{eq:NotLI}
    Y_i = q^{1/2} X_{i,r}(+,-) - q^{5/2} X_{i,r}(-,+)
\end{align}
for all $r \in \frac{1}{2}\mathbb{Z}$, further complicating things. Note that this also extends to all closed $(p,q)$-curves and their corresponding $(p,q,r)$-tangles.

Another possible route is to employing the embedding technique described in Section~\ref{section:modsQTori} and studying the image of $\mathscr{S}(T^2 \setminus D^2)$. Although some progress has been made, extracting useful patterns from this embedding is particularly arduous. Notably, the image of seemingly simple tangle elements quickly becomes unwieldy in size as $r \in \frac{1}{2}\mathbb{Z}$ moves further away from $0$.

For convenience and further use, we have calculated, and verified when $k=0$ using a computer program, all $16$ commuting relations among $X_{1,k}$ and $X_{2,k}$ for all states. Note that these calculations assume both tangles share the same number of twists with respect to our classification in Theorem~\ref{theorem:classification}. Here, $\widetilde{X}_{3,k}(\mu, \nu)$ corresponds to the $(1,-1)$-tangle with $k$ twists and $\widetilde{Y}_3$ is the closed $(1,-1)$-curve.

\begin{align*}
    X_{1,k}(+,+) X_{2,k}(+,+) &= q^{2} X_{2,k}(+,+) X_{1,k}(+,+) \\
    X_{1,k}(+,+) X_{2,k}(+,-) &= q^{-2} X_{2,k}(+,-) X_{1,k}(+,+) + q^{-3/2}(q^2 - q^{-2})X_{3,k}(+,+)\\
    X_{1,k}(+,+) X_{2,k}(-,+) &= q^{-2} X_{2,k}(-,+) X_{1,k}(+,+) \\
    X_{1,k}(+,+) X_{2,k}(-,-) &= q^{-6} X_{2,k}(-,-) X_{1,k}(+,+) + q^{-3/2}(q^2 - q^{-2}) X_{3,k}(-,+) \\
    X_{1,k}(+,-) X_{2,k}(+,+) &= q^{6} X_{2,k}(+,+) X_{1,k}(+,-) - q^{7/2}(q^2 - q^{-2})X_{2,k}(+,+)Y_1\\
    &\phantom{=} - q^{5/2}(q^2 - q^{-2})\left( q^{2} \widetilde{X}_{3,k}(+,+) + q^{-2}\widetilde{X}_{3,k-\frac{1}{2}}(+,+) \right)\\
    X_{1,k}(+,-) X_{2,k}(+,-) &= q^{2} X_{2,k}(+,-) X_{1,k}(+,-) + (q^2 - q^{-2})\widetilde{Y}_3 \\
    &\phantom{=} - q^{-1/2}(q^2 - q^{-2}) \left( q \widetilde{X}_{3,k}(+,-) + X_{2}(+,-)Y_1 - q^{-1} X_{3}(+,-) \right) \\
    X_{1,k}(+,-) X_{2,k}(-,+) &= q^{2} X_{2,k}(-,+) X_{1,k}(+,-) - q^{-1/2}(q^2 - q^{-2}) \widetilde{X}_{3,k-\frac{1}{2}}(-,+)\\
    X_{1,k}(+,-) X_{2,k}(-,-) &= q^{-2} X_{2,k}(-,-) X_{1,k}(+,-) + q^{-3/2}(q^2 - q^{-2}) X_{3,k}(-,-) \\
    X_{1,k}(-,+) X_{2,k}(+,+) &= q^{6} X_{2,k}(+,+) X_{1,k}(-,+) - q^{5/2}(q^2 - q^{-2}) \widetilde{X}_{3,k-\frac{1}{2}}(+,+) \\
    X_{1,k}(-,+) X_{2,k}(+,-) &= q^{2} X_{2,k}(+,-) X_{1,k}(-,+) - q^{1/2}(q^2 - q^{-2}) \widetilde{X}_{3,k}(-,+) \\
    X_{1,k}(-,+) X_{2,k}(-,+) &= q^{2} X_{2,k}(-,+) X_{1,k}(-,+) \\
    X_{1,k}(-,+) X_{2,k}(-,-) &= q^{-2} X_{2,k}(-,-) X_{1,k}(-,+) \\
    X_{1,k}(-,-) X_{2,k}(+,+) &= q^{10} X_{2,k}(+,+) X_{1,k}(-,-) - q^{13/2} (q^2 - q^{-2})\left( q^{4} X_{3,k}(-,+) + q^{-4} \widetilde{X}_{3,k}(+,-) \right)\\
    &\phantom{=} - q^{11/2}(q^2 - q^{-2}) \left( q^3 \widetilde{X}_{3,k-\frac{1}{2}}(-,+) + q^{-3} \widetilde{X}_{3,k-\frac{1}{2}}(+,-)\right)\\
    &\phantom{=} - q^{7}(q^2 - q^{-2})(q^3 + q^{-3}) \widetilde{Y}_{3,k}\\
    X_{1,k}(-,-) X_{2,k}(+,-) &= q^{6} X_{2,k}(+,-) X_{1,k}(-,-)\\
    &\phantom{=} - q^{5/2}(q^2 - q^{-2}) \left( q X_{2,k}(-,-)Y_1 + (q^{2} + q^{-2})\widetilde{X}_{3,k}(-,-) \right) \\
    X_{1,k}(-,-) X_{2,k}(-,+) &= q^{6} X_{2,k}(-,+) X_{1,k}(-,-) - q^{7/2} (q^2 - q^{-2}) \widetilde{X}_{3,k-\frac{1}{2}}(-,-) \\
    X_{1,k}(-,-) X_{2,k}(-,-) &= q^{2} X_{2,k}(-,-) X_{1,k}(-,-)
\end{align*}

\bibliography{bibtex}
\bibliographystyle{alpha}

\end{document}